\documentclass[a4paper,11pt]{amsart}
\usepackage{amsmath}
\usepackage{amsthm,amsfonts,amssymb,graphics}
\usepackage{mathtools}
\usepackage[foot]{amsaddr}
\usepackage{hyperref}
\usepackage{enumerate}
\usepackage{geometry}
\usepackage{dsfont}
\usepackage{caption,subcaption}
\usepackage{circledsteps}

\usepackage{tikz}
\usetikzlibrary{patterns, calc, arrows.meta}

\newtheorem{theorem}{Theorem}[section]
\newtheorem{lemma}[theorem]{Lemma}
\newtheorem{proposition}[theorem]{Proposition}
\newtheorem{corollary}[theorem]{Corollary}
\newtheorem{assumption}{Assumption}
\newtheorem{question}{Question}
\theoremstyle{remark}
\newtheorem{remark}{Remark}
\newtheorem{example}{Example}

\numberwithin{equation}{section}

\newcommand {\R} {\mathbb{R}}
\newcommand {\E} {\mathbb{E}}
\newcommand {\N} {\mathbb{N}}
\newcommand {\Z} {\mathbb{Z}}
\renewcommand{\P} {\mathbb{P}}
\newcommand {\Var} {\mathrm{Var}}
\newcommand {\Vol} {\mathrm{Vol}}
\newcommand {\SA} {\mathrm{SA}}
\newcommand {\EC} {\mathrm{EC}}

\newcommand {\diam} {\mathrm{diam}}
\newcommand {\Cov} {\mathrm{Cov}}
\newcommand {\dist} {\mathrm{dist}}
\newcommand{\ind}{\mathds{1}}

\begin{document}
\title[Three CLTs for the unbounded excursion component of a Gaussian field]{Three central limit theorems for the unbounded excursion component of a Gaussian field} 
\author{Michael McAuley\textsuperscript{1}}
\address{\textsuperscript{1}School of Mathematics and Statistics, Technological University Dublin}
\email{m.mcauley@cantab.net}
\begin{abstract}
For a smooth, stationary Gaussian field $f$ on Euclidean space with fast correlation decay, there is a critical level $\ell_c$ such that the excursion set $\{f\geq\ell\}$ contains a (unique) unbounded component if and only if $\ell$ is below $\ell_c$. We prove central limit theorems for the volume, surface area and Euler characteristic of this unbounded component restricted to a growing box. For planar fields, the results hold at all supercritical levels (i.e.\ all $\ell<\ell_c$). In higher dimensions the results hold at all sufficiently low levels (all $\ell<-\ell_c<\ell_c$) but could be extended to all supercritical levels by proving the decay of truncated connection probabilities. Our proof is based on the martingale central limit theorem.
\end{abstract}
\date{\today}
\thanks{}

\maketitle
\section{Introduction and main results}\label{s:introduction}
\subsection{Background and motivation}\label{ss:Motivation}
Percolation theory studies the long-range connectivity properties of random processes. Aside from being an elegant mathematical theory, this topic is motivated by a desire to understand phase transitions in different areas of science \cite{sah23}, particularly statistical physics \cite{ghm01}. Classical percolation models typically consist of random subsets of a lattice or a large graph. (See \cite{gri99} for a thorough background on classical percolation or \cite{dc18} for a shorter overview.)

More recently, progress has been made in understanding the percolative properties of continuous models, in particular the excursion sets of Gaussian fields (see \cite{bel23} for a recent survey). Let $f:\R^d\to\R$ be a $C^2$-smooth stationary Gaussian field. The excursion sets and level sets of $f$ are defined respectively as
\begin{displaymath}
\{f\geq\ell\}:=\left\{x\in\R^d\;\middle|\;f(x)\geq\ell\right\}\quad\text{and}\quad\{f=\ell\}:=\left\{x\in\R^d\;\middle|\;f(x)=\ell\right\}\quad\text{for }\ell\in\R.
\end{displaymath}
From the perspective of percolation, one may then ask about the topological behaviour of these sets on large scales and how this depends on $\ell$ or the distribution of $f$.

Let us mention a few sources of motivation for studying percolation of Gaussian fields:
\begin{enumerate}
    \item \underline{Applications}\\
    Smooth Gaussian fields are used for modelling phenomena across science (e.g.\ in medical imaging \cite{wor96} and quantum chaos \cite{js17}) and so understanding their topological/geometric properties frequently has applications in such areas. For example, the `component count' of a field (defined as the number of connected components of $\{f\geq\ell\}$ in some large domain) is a natural object of study in percolation theory and has been used in statistical testing of cosmological data \cite{pra19}.
    \item \underline{Stimulating mathematical theory}\\
    Gaussian fields offer new and interesting challenges in percolation theory because they often lack some of the important properties of classical discrete models (such as independence on disjoint domains, positive association, finite energy etc - see \cite{gri99} for definitions). Overcoming these obstacles requires new ideas, leading to a richer mathematical theory. As an example, \cite{mrv23} developed new general methods to prove the existence of a phase transition for Gaussian fields without positive association.
    \item \underline{Exploring universality}\\
    Certain properties of percolation models on discrete lattices have been observed numerically to be `universal' in the sense that they depend on the dimension of the model but not on the individual lattice \cite{lpps92,lh98}. Some of these properties also seem to be invariant under certain changes to the dependence structure of the model \cite{ck95,bs99}. There is great interest in understanding this universality; in particular, can we classify the models which have the same behaviour? Gaussian fields provide a new family of percolation models with which to probe these concepts, and numerical evidence \cite{bk13} suggests their behaviour may match that of some simple lattice models (at least for one important example of a Gaussian field). It would be of great interest to prove rigorous results in this direction for Gaussian fields.
\end{enumerate}
The most fundamental result in percolation theory is the existence of a (non-trivial) phase transition for long range connections. For a given Gaussian field $f$, this can be stated as follows: there exists $\ell_c\in\R$ such that
\begin{equation}\label{e:PhaseTransition}
    \P\big(\{f\geq\ell\}\text{ has an unbounded component}\big)=\begin{cases}
        1 &\text{if }\ell<\ell_c,\\
        0 &\text{if }\ell>\ell_c.
    \end{cases}
\end{equation}
In recent years, the existence of such a phase transition for Gaussian fields has been proven under increasingly general assumptions. We will describe precise conditions on $f$ which ensure this existence below. For now we just mention that if the covariance function of $f$ decays sufficiently quickly at infinity (and some other regularity conditions hold) then the phase transition occurs and the unbounded component of $\{f\geq\ell\}$ is unique almost surely for each $\ell<\ell_c$ (see Figure~\ref{fig:three graphs} for an illustration).

\begin{figure}[h]
     \centering
     \begin{subfigure}[b]{0.3\textwidth}
         \centering
         \includegraphics[width=\textwidth]{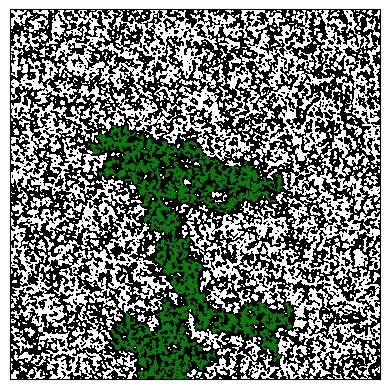}
         \caption{$\ell=0.05$}
         \label{fig:subcritical}
     \end{subfigure}
     \hfill
     \begin{subfigure}[b]{0.3\textwidth}
         \centering
         \includegraphics[width=\textwidth]{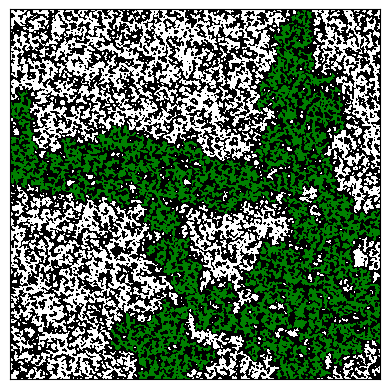}
         \caption{$\ell=\ell_c=0$}
         \label{fig:critical}
     \end{subfigure}
     \hfill
     \begin{subfigure}[b]{0.3\textwidth}
         \centering
         \includegraphics[width=\textwidth]{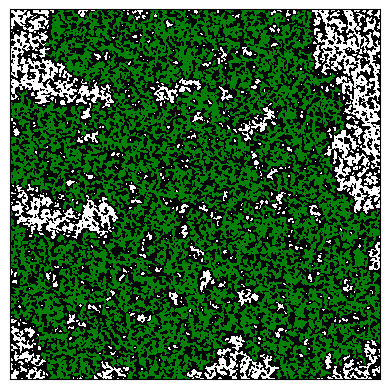}
         \caption{$\ell=-0.05$}
         \label{fig:supercritical}
     \end{subfigure}
        \caption{The excursion sets $\{f\geq\ell\}$ are shown in white for the Mat\'ern Gaussian field $f:\R^2\to\R$ with parameter $\nu=10$ (see Example~\ref{ex:fields} for a description of this field). The largest connected component of the excursion set is highlighted in green. The critical level $\ell_c$ for this model is known to be zero.}
        \label{fig:three graphs}
\end{figure}

Once this phase transition has been established, it is then a very natural progression to consider the geometry of the unbounded component in the supercritical regime (i.e.\ when $\ell<\ell_c$). In particular one might ask;
\begin{question}\label{q:Geometry}
    How do basic geometric statistics for the unbounded component behave?
\end{question}
\begin{question}\label{q:Dependence}
     How does this behaviour depend on the distribution of $f$?
\end{question}
Perhaps the most elementary geometric property of a subset of Euclidean space is its size or volume and this turns out to be an interesting quantity to study in the context of percolation theory. Let ${\{f\geq\ell\}_\infty}$ denote the unbounded component of $\{f\geq\ell\}$ (when it exists) and let $\Lambda_n=[-n,n]^d$. An application of the ergodic theorem (after imposing suitable conditions on $f$) shows that
\begin{displaymath}
    (2n)^{-d}\Vol\big[\{f\geq\ell\}_\infty\cap\Lambda_n\big]\to\P\big(0\in\{f\geq\ell\}_\infty\big)\qquad\text{as }n\to\infty,
\end{displaymath}
where $\Vol[\cdot]$ denotes the $d$-dimensional volume (Hausdorff measure) of a set and convergence occurs almost surely and in $L^1$. The latter term is the \emph{percolation probability} for the field $f$. We therefore see that, aside from being a natural quantity of interest in its own right, the normalised volume of the infinite component provides a consistent estimator for an important theoretical expression. Could one find the asymptotic variance and distribution of this estimator?

Another very natural geometric quantity for a set is the size of its boundary (according to some appropriate Hausdorff measure). The boundary of our excursion sets are $C^2$-smooth and so we consider their $(d-1)$-dimensional Hausdorff measure which we refer to as surface area (see Figure~\ref{fig:boundary}). Since the boundary of excursion sets (of a continuous function) are level sets, they have a clear interpretation. This also leads to the question of whether an unbounded component of the level set $\{f=\ell\}$ exists and, if so, what its geometric properties are. We discuss this briefly in Section~\ref{s:Discussion}.

\begin{figure}[h]
     \centering
     \hspace{0.15\textwidth}
     \begin{subfigure}[b]{0.3\textwidth}
         \centering
         \includegraphics[width=\textwidth]{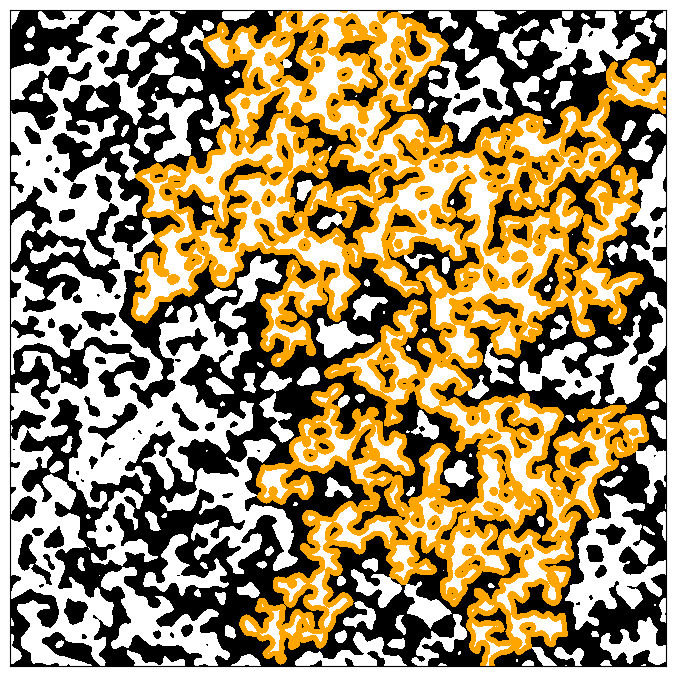}
         \caption{$\ell=-0.01$}
         \label{fig:boundary_mid}
     \end{subfigure}
     \hfill
     \begin{subfigure}[b]{0.3\textwidth}
         \centering
         \includegraphics[width=\textwidth]{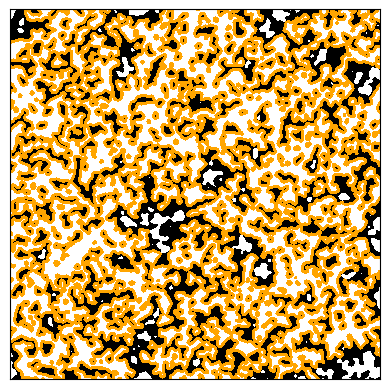}
         \caption{$\ell=-0.1$}
         \label{fig:boundary_low}
     \end{subfigure}
     \hspace{0.15\textwidth}
        \caption{The excursion sets $\{f\geq\ell\}$ are shown in white for the Bargmann-Fock field $f:\R^2\to\R$ (this model is described in Example~\ref{ex:fields}). The boundary of the largest connected component of the excursion set is highlighted in orange.}
        \label{fig:boundary}
\end{figure}

One might also ask; what is the topological shape of the unbounded excursion component? The Euler characteristic is an integer-valued quantity which provides topological information about a set. For a `nice' set $A\subset\R^2$, the Euler characteristic is the number of components of $A$ minus the number of `holes' in $A$ (see Figure~\ref{fig:ECDef}). Analogous definitions hold in higher dimensions. For a thorough background on the Euler characteristic and its study in the context of excursion sets of Gaussian fields we recommend \cite{at07}.
\begin{figure}[ht]
    \centering
    \includegraphics[width=0.3\textwidth]{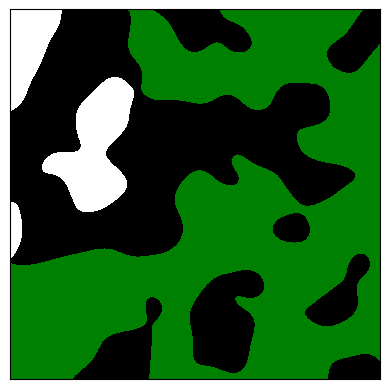}
    \caption{The largest component (highlighted in green) of the excursion set $\{f\geq\ell\}\cap\Lambda_n$ surrounds three `holes' and hence has an Euler characteristic of minus two. The Euler characteristic is well defined for the class of `basic complexes' which includes excursion sets of smooth Gaussian fields (see \cite[Chapter~6]{at07} for details).}
    \label{fig:ECDef}
\end{figure}

In this paper we study the previous three geometric quantities (volume, surface area and Euler characteristic) for the unbounded excursion component restricted to a large box. We show that if the covariance function of $f$ decays rapidly at infinity then each of these quantities satisfies a central limit theorem (CLT) as the size of the box increases. We discuss the significance of these results in Section~\ref{s:Discussion}, after stating them precisely.

\subsection{Statement of main results}
Let $f:\R^d\to\R$ be a stationary Gaussian field with mean zero, variance one and continuous covariance function $K(x-y):=\E[f(x)f(y)]$. We assume that $f$ can be represented as
\begin{displaymath}
f=q\ast W
\end{displaymath}
for some $q\in L^2(\R^d)$ which is Hermitian (i.e.\ $q(x)=q(-x)$) where $W$ denotes a Gaussian white noise on $\R^d$ and $\ast$ denotes convolution (we give a precise definition of the white noise at the beginning of Section~\ref{ss:Outline}). We note that such a representation always exists if $K\in L^1(\R^d)$; by Bochner's theorem $K$ is the Fourier transform of a probability measure and if $K$ is integrable this measure has a density $\rho$ with respect to Lebesgue measure, one can then set $q=\mathcal{F}[\sqrt{\rho}]$ (where $\mathcal{F}[g](x)=\int_{\R^d}e^{it\cdot x}g(t)\;dt$ denotes the Fourier transform) and observe that $q\ast W$ will have covariance $q\ast q=K$ as required.

Throughout our work we make the following assumptions:
\begin{assumption}
\label{a:clt}
\leavevmode
\begin{enumerate}
    \item (Smoothness 1) $K\in C^{2k+\epsilon}$ for some integer $k\geq 4$ and some $\epsilon>0$,
    \item (Smoothness 2) $q\in C^{3}(\R^d)$ and $\partial^\alpha q \in L^2(\R^d)$ for every $|\alpha| \le 3$,
    \item (Decay) There exist $\beta >d$ and $c \ge 1$ such that, for all $|x| \ge 1$,
    \[  \max_{|\alpha| \le 2} |\partial^\alpha q(x)|  < c |x|^{-\beta},  \]
    \item (Symmetry) $q(x)$ is invariant under sign changes and permuting the coordinates of $x$,
    \item (Positivity) If $d=2$ then $K=q\ast q\geq 0$, if $d\geq 3$ then $q\geq 0$.
\end{enumerate}
\end{assumption}
We note that the first smoothness condition above implies that $f$ is $C^{k}$ almost surely and the decay condition implies that
\begin{displaymath}
    \max_{\lvert\alpha\rvert\leq 2}\lvert\partial^\alpha K(x)\rvert< C\lvert x\rvert^{-\beta}
\end{displaymath}
for all $\lvert x\rvert\geq 1$. Conditions (2), (4) and (5) here will not play any role in our arguments, but we include them as they have been used to establish the phase transition for $f$.

We consider some examples of Gaussian fields which satisfy the above assumption:

\begin{example}\label{ex:fields}\leavevmode
\begin{enumerate}
    \item The \emph{Bargmann-Fock field} is the centred, real-analytic Gaussian field with covariance function
\begin{displaymath}
    K(x)=e^{-\frac{1}{2}\lvert x\rvert^2}.
\end{displaymath}
This field is of interest in algebraic geometry, where it arises as a local scaling limit of random homogeneous polynomials of high degree (see the introduction to \cite{bg17} for further details and motivation). The percolative properties of the Bargmann-Fock field are well studied, in part because its super-exponential decay of correlations is helpful in extending arguments from discrete percolation which rely on independence on disjoint domains. Using basic properties of the Fourier transform, one can verify that the Bargmann-Fock field has a white noise decomposition with $q(x):=(2/\pi)^{d/4}e^{-\lvert x\rvert^2}$ which evidently satisfies Assumption~\ref{a:clt} for any $\beta>d$.
\item The class of \emph{Mat\'ern fields} are defined by the covariance functions
\begin{displaymath}
    K(x)=\frac{1}{2^{\nu-1}\Gamma(\nu)}\lvert x\rvert^\nu K_\nu(\lvert x\rvert)
\end{displaymath}
where $\Gamma$ denotes the Gamma function, $K_\nu$ is the modified Bessel function of the second kind and $\nu>0$ is a parameter. These fields are widely used in machine learning \cite[Chapter~4]{rw06}. The spectral density of such a field is given by
\begin{displaymath}
    \rho(t)\propto (1+\lvert t\rvert^2)^{-\nu-d/2}.
\end{displaymath}
Hence the covariance function is $k$ times differentiable if and only if $\nu>k$, so that the parameter $\nu$ governs the smoothness of the field. Since $q=\mathcal{F}[\sqrt{\rho}]$ it follows that
\begin{displaymath}
    q(x)\propto \lvert x\rvert^{\nu^\prime}K_{\nu^\prime}(\lvert x\rvert)
\end{displaymath}
where $\nu^\prime:=\frac{\nu}{2}-\frac{d}{4}$. Using properties of the modified Bessel functions, one can then verify that all the conditions of Assumption~\ref{a:clt} are satisfied for $\nu>6+d/2$ and any $\beta>d$.
\item As an example of a field satisfying Assumption~\ref{a:clt} for a fixed value of $\beta$, one can set
\begin{displaymath}
    q(x)=(1+\lvert x\rvert^2)^{-\beta/2}.
\end{displaymath}
These fields have been less used in applications, but are of interest theoretically for understanding how the rate of decay of correlations (i.e.\ the value of $\beta$) affects the geometric properties of the field.
\end{enumerate}
\end{example}

It has previously been shown that Assumption~\ref{a:clt} ensures that the field undergoes a sharp version of the phase transition described in \eqref{e:PhaseTransition}, which we now make precise. Recall that $\Lambda_n=[-n,n]^d$ is the centred cube of side length $2n$ and for three sets $A,B,C\subseteq\R^d$ let us write $A\overset{B}{\longleftrightarrow}C$ for the event that there exists a continuous curve in $B$ which starts in $A$ and ends in $C$.
\begin{theorem}[\cite{mv20,riv21} for $d=2$, \cite{sev21,sev23} for $d\geq 3$]\label{t:PhaseTransition}
    Let $f$ satisfy Assumption~\ref{a:clt}, then there exists $\ell_c\in\R$ such that the following holds: 
\begin{itemize}
    \item for each $\ell>\ell_c$ there exists $C,c>0$ such that for all $n>2$
    \begin{displaymath}
        \P\Big(\Lambda_1\overset{\{f\geq\ell\}}{\longleftrightarrow}\partial\Lambda_n\Big)<Ce^{-cn},
    \end{displaymath}
    \item for each $\ell<\ell_c$ there exists $C,c>0$ such that for all $n>2$
    \begin{displaymath}
        \P\Big(\Lambda_n\cap\{f\geq\ell\}_\infty=\emptyset\Big)<Ce^{-cn^{d-1}}
    \end{displaymath}
    where $\{f\geq\ell\}_\infty$ denotes the union of all unbounded components of $\{f\geq\ell\}$.
\end{itemize}
Moreover for $\ell<\ell_c$, with probability one the excursion set $\{f\geq\ell\}$ has precisely one unbounded component (i.e.\ $\{f\geq\ell\}_\infty$ is connected).
\end{theorem}
We remark that in the planar case (i.e.\ when $d=2$) it was proven that $\ell_c=0$. In higher dimensions ($d\geq 3$) it is expected that $\ell_c>0$ and this has been proven \cite{drrv23} when Assumption~\ref{a:clt} is satisfied for a sufficiently large $\beta$ .

One more key ingredient is required for our results.
\begin{assumption}[Truncated connection decay]\label{a:Decay}
For a given $\ell<\ell_c$, the probability that $\Lambda_1$ intersects a bounded component of $\{f\geq\ell\}$ of diameter greater than $n$ decays super-polynomially in $n$.
\end{assumption}
Letting $\{f\geq\ell\}_{<\infty}:=\{f\geq\ell\}\setminus\{f\geq\ell\}_\infty$ be the union of all bounded excursion components, Assumption~\ref{a:Decay} is equivalent to saying that for any $k\in\N$
\begin{displaymath}
    n^k\cdot\P\Big(\Lambda_1\overset{\{f\geq\ell\}_{<\infty}}{\longleftrightarrow}\partial\Lambda_n\Big)\to 0\qquad\text{as }n\to\infty.
\end{displaymath}
Truncated connection decay is closely related to a stronger local form of uniqueness for the unbounded component which is conjectured to hold for all supercritical levels. We discuss this assumption further in Section~\ref{s:Discussion}. When $\ell<-\ell_c$, we show that truncated connection decay is an easy consequence of the subcritical connection decay of Theorem~\ref{t:PhaseTransition} and symmetry of the Gaussian distribution.

\begin{proposition}\label{p:BCDecay}
Let $f$ satisfy Assumption~\ref{a:clt}, then Assumption~\ref{a:Decay} holds for all $\ell<-\ell_c$. In particular if $d=2$ then Assumption~\ref{a:Decay} holds for the entire supercritical regime $\ell<\ell_c=0$.
\end{proposition}

This result essentially follows because any bounded component of $\{f\geq\ell\}$ must be `surrounded' by a component of $\{f\leq\ell\}$. If the former has large diameter then so must the latter, but for $\ell<-\ell_c$ this probability decays exponentially in the diameter by Theorem~\ref{t:PhaseTransition} (and the fact that $\{f\leq\ell\}$ has the same distribution as $\{f\geq -\ell\}$). The complete argument is given in Section~\ref{s:Topological}.

We can now state our main results. Let $f$ satisfy Assumption~\ref{a:clt} for some given $\ell<\ell_c$. (This choice of $f$ and $\ell$ is fixed throughout most of the paper.) For any compact Borel set $A$, we define our three quantities of interest as
\begin{align*}
    \mu_\Vol(A)&:=\Vol[A\cap\{f\geq\ell\}_\infty]\\
    \mu_\EC(A)&:=\EC[A\cap\{f\geq\ell\}_\infty]\\
    \mu_\SA(A)&:=\mathcal{H}^{d-1}[A\cap\partial(\{f\geq\ell\}_\infty)]
\end{align*}
where $\Vol[\cdot]$ denotes the volume, $\EC[\cdot]$ denotes the Euler characteristic and $\mathcal{H}^k[\cdot]$ denotes the $k$-dimensional Hausdorff measure. Our first significant result is a `law of large numbers' for these functionals as the domain size increases:

\begin{theorem}\label{t:FirstOrder}
    Let $f$ satisfy Assumption~\ref{a:clt} and $\ell<\ell_c$. Given $\star\in\{\Vol,\EC,\SA\}$ there exists $c_\star(\ell)\in\R$ such that
    \begin{displaymath}
        \frac{\mu_\star(\Lambda_n)}{(2n)^d}\to c_\star(\ell) \qquad\text{as }n\to\infty,
    \end{displaymath}
    where convergence occurs almost surely and in $L^1$. Furthermore
    \begin{align*}
        c_\Vol(\ell)&=\P(0\in\{f\geq\ell\}_\infty),\\
        c_\SA(\ell)&=\E[\lvert\nabla f(0)\rvert\ind_{0\in\{f\geq\ell\}_\infty}|f(0)=\ell]\phi(\ell),\qquad\text{and}\\
        c_\EC(\ell)&=(-1)^d\E[\det\nabla^2 f(0)\ind_{0\in\{f\geq\ell\}_\infty}|\nabla f(0)=0]\varphi_{\nabla f(0)}(0)
    \end{align*}
    where $\phi$ and $\varphi_{\nabla f(0)}$ denote the density of a standard Gaussian and $\nabla f(0)$ respectively. In particular,
    \begin{displaymath}
        0<c_\Vol(\ell)<1-\Phi(\ell),\qquad 0<c_\SA(\ell)<\E[\lvert\nabla f(0)\rvert]\phi(\ell),
    \end{displaymath}
    where $\Phi$ denotes the standard Gaussian cumulative distribution function, and as $\ell\to -\infty$,
    \begin{displaymath}
        c_\Vol(\ell)\to 1,\qquad c_\SA(\ell)\to0\qquad \text{and}\qquad c_\EC(\ell)\to 0.
    \end{displaymath}
    Moreover, if $d=2$, then $c_\EC(\ell)<0$.
\end{theorem}

The proof of this result follows from the ergodic theorem (along with some topological arguments to characterise the limiting functional $c_\EC(\ell)$).

Our next result describes the asymptotic variance and limiting distribution of our functionals. Given $k\geq 4$, which describes the smoothness of our field in Assumption~\ref{a:clt}, we define
\begin{displaymath}
    \beta_\Vol(k)=3d,\qquad\beta_\EC(k)=\frac{k-1}{k-3}3d,\qquad\beta_\SA(k)=\frac{9k(k+1)-42}{k(k+1)-8}d.
\end{displaymath}

\begin{theorem}\label{t:clt}
Let $\star\in\{\Vol,\EC,\SA\}$ and suppose Assumption~\ref{a:clt} holds for some $k\geq 4$ and $\beta>\beta_\star(k)$. Suppose also that Assumption~\ref{a:Decay} holds, then there exists $\sigma=\sigma_\star(\ell)\geq 0$ such that as $n\to\infty$
\begin{displaymath}
\frac{\Var[\mu_\star(\Lambda_n)]}{(2n)^d}\to\sigma^2\qquad\text{and}\qquad \frac{\mu_\star(\Lambda_n)-\E[\mu_\star(\Lambda_n)]}{(2n)^{d/2}}\xrightarrow{d}\sigma Z
\end{displaymath}
where $Z$ is a standard normal random variable.
\end{theorem}

A semi-explicit expression for $\sigma_\star(\ell)$ is given in Theorem~\ref{t:GenCLT}. By imposing slightly stronger conditions, we can ensure that the limiting variance in Theorem~\ref{t:clt} is strictly positive and hence the limiting distribution is non-degenerate:
\begin{theorem}
\label{t:posvar}
Let $\star\in\{\Vol,\EC,\SA\}$ and let $\ell<-\ell_c$. If $\star=\EC$ we assume that $d=2$. Suppose the conditions for Theorem~\ref{t:clt} are satisfied and in addition that $q \geq 0$, then
\begin{displaymath}
    \sigma_\star(\ell) > 0.
\end{displaymath}
\end{theorem}
    The assumption that $\ell<-\ell_c$ is of course vacuous for planar fields since $\ell_c=0$. We discuss how this condition could be weakened in the next subsection.
\begin{remark}
    The proof of positivity relies on the existence of sets of positive measure on which $c_\star$, defined in Theorem~\ref{t:FirstOrder}, takes different values (we explain this more in the next subsection). This assumption is verified in Theorem~\ref{t:FirstOrder} for all $d\geq 2$ when $\star\in\{\Vol,\SA\}$ but only for $d=2$ when $\star=\EC$, explaining the restrictions in the statement of Theorem~\ref{t:posvar}. The obstacle in the latter case is that $\mu_\EC$ can take positive and negative values, so that in principle it may have an (asymptotic) expectation of zero at all levels (although intuitively this seems very unlikely). If one could show that $c_\EC$ is non-zero on a set of positive measure for other values of $d$, this would immediately extend the conclusion of Theorem~\ref{t:posvar} to such values.
\end{remark}

Whilst the geometric quantities $\mu_\star$ for $\star\in\{\Vol,\SA,\EC\}$ are very natural from a theoretical perspective, they are less useful in applications as they cannot be observed from a bounded domain. That is, if we look at a realisation of $\{f\geq\ell\}$ on a large domain $\Lambda_n$ for some $\ell<\ell_c$, we expect the set to be dominated by one large component which is part of $\{f\geq\ell\}_\infty$. However we cannot tell from within $\Lambda_n$ which parts of the excursion set that intersect the boundary are contained in the unbounded component. Thus it would be useful to have a `finitary' version of our CLT. Fortunately such a statement follows without much difficulty from our main result (Theorem~\ref{t:clt}).

\begin{figure}[ht]
    \centering
\begin{tikzpicture}[scale=0.5]

\begin{scope}
\draw[thick,dashed] (-6,-6) rectangle (6,6);
\node[left] at (-6,6) {$\Lambda_{(1+\epsilon)n}$};
\draw [thick,dashed](-4,-4) rectangle (4,4);
\node[left] at (-4,4) {$\Lambda_n$};
\clip (-6,-6) rectangle (6,6);
\draw [thick,pattern=north west lines,pattern color=gray] plot [smooth cycle,tension=0.5] coordinates {(-3.5,6.5)(-2.5,5.5)(-2.5,5)(-3.5,4.5)(-3.5,3.5)(-4,2.5)(-5,2)(-6.5,2)(-6.5,0)(-5,0)(-3.5,-1)(-4.5,-2)(-6.5,-2.5)(-6.5,-3)(-5.5,-4)(-4,-3.5)(-3,-2)(-2,-1)(-0.5,0)(2,0)(3.5,-0.5)(5,-1)(6.5,-1)(6.5,4)(4,6.5)(2.5,6.5)(1.5,5)(0,3.5)(-1.5,4.5)(-1,6)(-0.5,6.5)};

\draw [thick,pattern=north west lines,pattern color=gray] plot [smooth cycle,tension=0.5] coordinates {(6.5,-4)(5,-3)(3.5,-2.5)(1.5,-2.5)(0.5,-4)(1,-5)(1.5,-6.5)(3,-6.5)(4,-5)(6.5,-5)};

\draw [thick,pattern=north west lines,pattern color=gray] plot [smooth cycle,tension=0.5] coordinates {(0,-2)(-2,-3.5)(-3,-5)(-1.5,-5)};

\draw [thick,pattern=north west lines,pattern color=gray] plot [smooth cycle,tension=0.5] coordinates {(-6.5,-5)(-5,-5)(-5,-6.5)(-6.5,-6.5)};
\end{scope}

\begin{scope}
\clip (-4,-4) rectangle (4,4);
\draw [thick,fill=gray, fill opacity=0.5] plot [smooth cycle,tension=0.5] coordinates {(-3.5,6.5)(-2.5,5.5)(-2.5,5)(-3.5,4.5)(-3.5,3.5)(-4,2.5)(-5,2)(-6.5,2)(-6.5,0)(-5,0)(-3.5,-1)(-4.5,-2)(-6.5,-2.5)(-6.5,-3)(-5.5,-4)(-4,-3.5)(-3,-2)(-2,-1)(-0.5,0)(2,0)(3.5,-0.5)(5,-1)(6.5,-1)(6.5,4)(4,6.5)(2.5,6.5)(1.5,5)(0,3.5)(-1.5,4.5)(-1,6)(-0.5,6.5)};

\draw [thick,fill=gray, fill opacity=0.5]plot [smooth cycle,tension=0.5] coordinates {(6.5,-4)(5,-3)(3.5,-2.5)(1.5,-2.5)(0.5,-4)(1,-5)(1.5,-6.5)(3,-6.5)(4,-5)(6.5,-5)};

\draw [thick,fill=white] plot [smooth cycle,tension=0.5] coordinates {(0,2)(2,2)(3,3)};

\end{scope}
\node[left] at (-7,0) {$\{f\geq\ell\}$};
\draw (-7,0)--(-3,1);
\draw (-7,0)--(3,-3);
\draw (-7,0)--(-1.5,-3.5);
\draw (-7,0)--(-5.7,-5.7);

\end{tikzpicture}
    \caption{The dashed area shows the excursion set $\{f\geq\ell\}$ restricted to $\Lambda_{(1+\epsilon)n}$. The dark grey area corresponds to $\{f\geq\ell\}_{n,\epsilon}$. In particular, the lower left excursion component in $\Lambda_n$ is not shaded because it is not connected to $\partial\Lambda_{(1+\epsilon)n}$.}
    \label{fig:Finitary}
\end{figure}
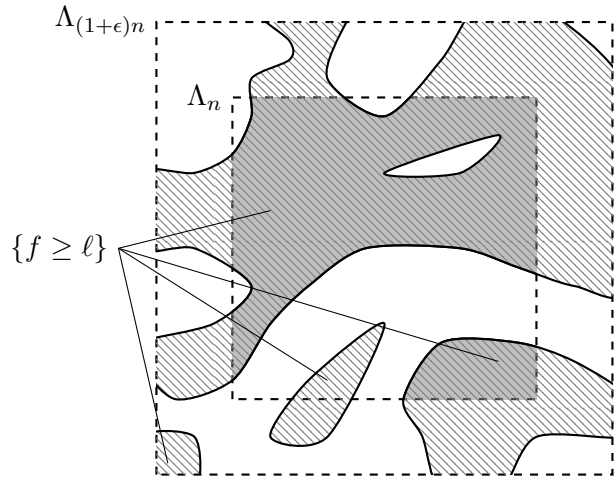
Let $\epsilon>0$ be fixed and for $n\in\N$ let $\{f\geq\ell\}_{n,\epsilon}$ be the union of all components of $\{f\geq\ell\}\cap\Lambda_n$ which are connected to $\partial\Lambda_{(1+\epsilon)n}$ in $\{f\geq\ell\}$ (see Figure~\ref{fig:Finitary}).
We then define
\begin{align*}
    \mu_\Vol(\Lambda_n,\epsilon)&:=\Vol[\{f\geq\ell\}_{n,\epsilon}]\\
    \mu_\EC(\Lambda_n,\epsilon)&:=\EC[\{f\geq\ell\}_{n,\epsilon}]\\
    \mu_\SA(\Lambda_n,\epsilon)&:=\mathcal{H}^{d-1}[\{f\geq\ell\}_{n,\epsilon}\cap\{f=\ell\}].
\end{align*}

As a consequence of Assumption~\ref{a:Decay}, components of $\{f\geq\ell\}\cap\Lambda_n$ have very small probability of connecting to $\partial\Lambda_{(1+\epsilon)n}$ unless they are part of the unbounded component. One can therefore show that $\{f\geq\ell\}_{n,\epsilon}$ and $\{f\geq\ell\}_\infty\cap\Lambda_n$ have essentially the same geometric statistics, meaning that the former also satisfies a CLT:
\begin{corollary}\label{c:FinitaryCLT}
    Let $\epsilon>0$ be fixed, then the statement of Theorem~\ref{t:clt} holds verbatim if $\mu_\star(\Lambda_n)$ is replaced by $\mu_\star(\Lambda_n,\epsilon)$. Moreover the limiting variance $\sigma^2_\star(\ell)$ takes the same value after this replacement.
\end{corollary}

\subsection{Outline of the proof}\label{ss:Outline}
We now describe our method of proof for Theorems~\ref{t:clt} and~\ref{t:posvar}:

\medskip
\underline{White noise CLT:} We define a Gaussian white noise (on $\R^d$) to be a centred Gaussian process $W$ indexed by $L^2(\R^d)$ such that for $g,h\in L^2(\R^d)$, $\E[W(g)W(h)]=\int_{\R^d}g(x)h(x)\;dx$. One can think of $W$ as a random distribution, with test functions in $L^2(\R^d)$, and we emphasise this by writing $W(g)=:\int g(x)W(dx)$. In particular, this justifies the definition of $f$ given earlier: for $x\in\R^d$
\begin{displaymath}
    f(x)=q\ast W(x):=\int q(x-y)W(dy).
\end{displaymath}
For a Borel set $A\subset \R^d$ (with finite Lebesgue measure) we define $W(A):=W(\ind_A)$. We fix a Gaussian white noise $W$, to be used throughout our analysis, and for $v\in\Z^d$ let $W_v$ be the restriction of $W$ to the cube $B_v:=v+[0,1]^d$ (i.e.\ $W_v(g):=W(g\ind_{B_v})$ for all $g\in L^2(\R^d)$). For each $n\in\N$ and $v\in\Z^d$ let $\mu(v+\Lambda_n):=F(v+\Lambda_n,W)$ where $F$ is some deterministic functional. We call $\mu$ a \emph{stationary white-noise functional} if it is invariant under a common translation of both arguments; in other words, for all $u,v\in\Z^d$ and $n\in\N$
\begin{displaymath}
    F(u+v+\Lambda_n,W(\cdot-u))=F(v+\Lambda_n,W),
\end{displaymath}
where $W(\cdot-u)$ denotes the functional defined by $A\mapsto W(A-u)$ for any Borel set $A$. We note that $\mu_\star$ is a stationary white-noise functional for each $\star\in\{\Vol,\SA,\EC\}$ (courtesy of the white noise representation for $f$).

Our first step is to apply (a slight generalisation of) the classical martingale-array CLT to the stationary white-noise functional $\mu$. Let $\preceq$ denote the standard lexicographic ordering on $\Z^d$. We define a family of $\sigma$-algebras $(\mathcal{F}_v)_{v\in\Z^d}$ by $\mathcal{F}_v:=\sigma(W_u\;|\;u\preceq v)$ and a collection of random variables
\begin{displaymath}
    S_{n,v}:=\frac{\E[\mu(\Lambda_n)\;|\;\mathcal{F}_v]-\E[\mu(\Lambda_n)]}{(2n)^{d/2}}\qquad n\in\N,v\in\Z^d.
\end{displaymath}
This collection forms a martingale array with respect to the lexicographic ordering (we define this notion precisely in Section~\ref{s:CompletingProof}). Moreover if we take each coordinate of $v$ to infinity then one can show that $S_{n,v}$ converges to $S_{n,\infty^\ast}:=(2n)^{-d/2}(\mu(\Lambda_n)-\E[\mu(\Lambda_n)])$. For $v\in\Z^d$ let $v^-$ denote the element of $\Z^d$ immediately preceding $v$ in the lexicographic ordering. We write $U_{n,v}:=S_{n,v}-S_{n,v^-}$ for the increment of our lexicographic martingale array. The martingale CLT states that $S_{n,\infty^\ast}$ is asymptotically Gaussian, provided that the increments $U_{n,v}$ satisfy certain probabilistic bounds. We show that these follow from corresponding bounds on the change in $\mu$ when locally resampling the white noise, which we now describe.

For $v\in\Z^d$ let $\widetilde{W}^{(v)}$ denote $W$ after resampling $W_v$ independently. That is, let $W^\prime$ be an independent copy of $W$ and for any Borel set $A$ we define
\begin{displaymath}
    \widetilde{W}^{(v)}(A)=W(A\setminus B_v)+W^\prime(A\cap B_v).
\end{displaymath}
We then define
\begin{displaymath}
    \Delta_v(u+\Lambda_n)=F(u+\Lambda_n,W)-F(u+\Lambda_n,\widetilde{W}^{(v)})
\end{displaymath}
to be the change in our functional when resampling the white noise on the cube $B_v$. Then by definition of $\mathcal{F}_v$
\begin{displaymath}
    U_{n,v}=(2n)^{-d/2}\E[\Delta_v(\Lambda_n)\;|\;\mathcal{F}_v]\quad\text{a.s.},
\end{displaymath}
which allows us to relate the conditions of the martingale CLT to $\Delta_v(\Lambda_n)$. By assuming also that $\mu$ is additive over distinct unit cubes, we finally show that a CLT holds for $S_{n,\infty^\ast}$ provided that for all $v,w\in\Z^d$
\begin{equation}\label{e:outline_sufficient}
    \E\left[\lvert\Delta_v(B_w)\rvert+\lvert\Delta_v(B_w)\rvert^{2+\epsilon}\right]\leq C(1+\lvert v-w\rvert)^{-3d-\delta}
\end{equation}
where $C,\epsilon,\delta>0$ are independent of $v$ and $w$.

\medskip\underline{Application to unbounded component:} Theorem~\ref{t:clt} is proven by verifying condition \eqref{e:outline_sufficient} for $\mu_\star$ when $\star\in\{\Vol,\SA,\EC\}$. In this setting, resampling the white noise on $B_v$ is equivalent to replacing $f$ by $\widetilde{f}_v:=q\ast\widetilde{W}^{(v)}$. We can control the difference between the excursion sets of these two functions using a basic Morse theoretic argument: consider the family of excursion sets $\{(1-t)f+t\widetilde{f}_v\geq\ell\}$ for $t\in[0,1]$, as $t$ increases, the level/excursion sets will deform continuously unless they pass through a critical point.

We now focus on the case $\star=\Vol$ and ask, how can the volume of the unbounded component contained in some fixed cube $B_w$ change as we vary $t$? If the level sets $\{(1-t)f+t\widetilde{f}_v=\ell\}$ pass through any critical points inside $B_w$ (including `boundary critical points' which will be defined later), then their topology may change (see case (i) in Figure~\ref{fig:LocalvsNonLocal}). We call such a cube \emph{locally unstable}. In this case we do not know which components of $\{f\geq\ell\}\cap B_w$ and $\{\widetilde{f}_v\geq\ell\}\cap B_w$ are contained in the corresponding unbounded components. We say that there is a \emph{local topological contribution} to the change in volume of the unbounded component and we bound this contribution trivially:
\begin{equation}\label{e:TrivialBound}
    \lvert\Delta_v(B_w)\rvert\leq 1.
\end{equation}

If the level sets do not pass through any critical points, then their topology inside $B_w$ is preserved and in this case the volume of the unbounded component (restricted to $B_w$) can change due to excursion components changing in volume (see case (ii) in Figure~\ref{fig:LocalvsNonLocal}). We describe this as a \emph{geometric contribution} to the change in volume. It is not difficult to control this contribution in terms of the regions where $f-\ell$ is small (relative to $f-\widetilde{f}_v$).

However there is one more way in which the volume of the unbounded component may change. Even if the topology of level sets within $B_w$ is unchanged, it is possible that topological changes elsewhere will disconnect some of the excursion components in $B_w$ from the unbounded component (see case (iii) in Figure~\ref{fig:LocalvsNonLocal}). We describe this as a \emph{non-local topological contribution}. In this case one of the components of $\{f\geq\ell\}\cap B_w$ must be contained in $\{f\geq\ell\}_\infty$ and the corresponding component of $\{\widetilde{f}_v\geq\ell\}\cap B_w$ must not be contained in $\{\widetilde{f}_v\geq\ell\}_\infty$ (or vice versa). This implies that $B_w$ must be connected by a bounded component of $\{f\geq\ell\}$ (or $\{\widetilde{f}_v\geq\ell\}$) to some locally unstable cube $B_u$. We note that $B_u$ can, in principle, be arbitrarily far from $B_w$. Once again, we bound such a contribution trivially by one.
\begin{figure}[ht]
    \centering
\begin{tikzpicture}[scale=0.8]
%\draw (-2.5,-2.5) rectangle (2.5,2.5);
\draw (-1,-1) rectangle (1,1);
\node[left] at (-1,0.6) {$B_w$};
\draw[thick,pattern=north west lines,pattern color=gray] plot [smooth,tension=0.5] coordinates {(3/2,0.8)(1,0.2)(0.8,0)(0.6,0)(0.4,0.2)(0.2,0.2)(0,0)(-0.1,-0.2)(0,-0.4)(0.2,-0.6)(0.4,-0.6)(0.6,-0.4)(0.8,-0.4)(1,-0.6)(3/2,-1.2)};
\node[fill=white] at (1.5,-0.2) {$\infty$};
\draw[thick,pattern=north west lines,pattern color=gray] (0.4,1)circle (0.4);
\node[below] at (-1,-1.2) {$\{f\geq\ell\}_\infty$};
\draw (-1,-1.2)--(0.2,-0.4);
\node[above] at (-1,1.4) {$\{f\geq\ell\}_{<\infty}$};
\draw (-1,1.4)--(0.2,1.1);

\draw[-{Triangle[width=18pt,length=8pt]}, line width=10pt](-3,-2.5) -- (-4,-3.5);
\draw[-{Triangle[width=18pt,length=8pt]}, line width=10pt](0,-2.5) -- (0,-3.5);
\draw[-{Triangle[width=18pt,length=8pt]}, line width=10pt](3,-2.5) -- (4,-3.5);

\begin{scope}[shift={(-6,-6)}]
%\draw (-2.5,-2.5) rectangle (2.5,2.5);
\node at (-2.8,1.5) {(i)};
\draw (-1,-1) rectangle (1,1);
\node[left] at (-1,0.6) {$B_w$};

\draw[thick,pattern=north west lines,pattern color=gray] (0.4,1)circle (0.2);
\draw[thick,pattern=north west lines,pattern color=gray] plot [smooth,tension=0.5] coordinates {(3/2,0.7)(1,0.1)(0.9,0)(0.8,-0.2)(0.9,-0.4)(1,-0.5)(3/2,-1.1)};
%\draw[pattern=north west lines,pattern color=gray] (0.3,-0.2)circle (0.2);
\draw[thick,pattern=north west lines,pattern color=gray] plot [smooth cycle,tension=0.5] coordinates {(0.6,-0.2)(0.5,0)(0.4,0.1)(0.2,0.1)(0.1,0)(0,-0.2)(0.1,-0.4)(0.2,-0.5)(0.4,-0.5)(0.5,-0.4)};
\node[fill=white] at (1.5,-0.2) {$\infty$};
\node[below] at (-1,-1.2) {$\{\widetilde{f}_v\geq\ell\}_\infty$};
\draw (-1,-1.2)--(1.1,-0.5);
\node[above] at (-1,1.4) {$\{\widetilde{f}_v\geq\ell\}_{<\infty}$};
\draw (-1,1.4)--(0.3,1.1);
\draw (-1,1.4)--(0.2,-0.1);
\end{scope}

\begin{scope}[shift={(0,-6)}]
%\draw (-2.5,-2.5) rectangle (2.5,2.5);
\draw (-1,-1) rectangle (1,1);
\node[left] at (-1,0.6) {$B_w$};
\node at (-2.8,1.5) {(ii)};

\draw[dashed] plot [smooth,tension=0.5] coordinates {(3/2,0.8)(1,0.2)(0.8,0)(0.6,0)(0.4,0.2)(0.2,0.2)(0,0)(-0.1,-0.2)(0,-0.4)(0.2,-0.6)(0.4,-0.6)(0.6,-0.4)(0.8,-0.4)(1,-0.6)(3/2,-1.2)};
\draw[dashed] (0.4,1)circle (0.4);

\draw[thick,pattern=north west lines,pattern color=gray] (0.4,1)circle (0.2);
\draw[thick,pattern=north west lines,pattern color=gray] plot [smooth,tension=0.5] coordinates {(3/2,0.7)(1,0.1)(0.9,0)(0.7,-0.1)(0.5,0)(0.4,0.1)(0.2,0.1)(0.1,0)(0,-0.2)(0.1,-0.4)(0.2,-0.5)(0.4,-0.5)(0.5,-0.4)(0.7,-0.3)(0.9,-0.4)(1,-0.5)(3/2,-1.1)};
%\draw[pattern=north west lines,pattern color=gray] (0.3,-0.2)circle (0.2);
%\draw[thick,pattern=north west lines,pattern color=gray] plot [smooth cycle,tension=0.5] coordinates {(0.6,-0.2)(0.5,0)(0.4,0.1)(0.2,0.1)(0.1,0)(0,-0.2)(0.1,-0.4)(0.2,-0.5)(0.4,-0.5)(0.5,-0.4)};
\node[fill=white] at (1.5,-0.2) {$\infty$};
\node[below] at (-1,-1.2) {$\{\widetilde{f}_v\geq\ell\}_\infty$};
\draw (-1,-1.2)--(1.1,-0.5);
\node[above] at (-1,1.4) {$\{\widetilde{f}_v\geq\ell\}_{<\infty}$};
\draw (-1,1.4)--(0.3,1.1);

\clip (-1,-1) rectangle (1,1);
\fill[color=gray,opacity=0.5] plot [smooth,tension=0.5] coordinates {(3/2,0.8)(1,0.2)(0.8,0)(0.6,0)(0.4,0.2)(0.2,0.2)(0,0)(-0.1,-0.2)(0,-0.4)(0.2,-0.6)(0.4,-0.6)(0.6,-0.4)(0.8,-0.4)(1,-0.6)(3/2,-1.2)(3/2,-1.1)(1,-0.5)(0.9,-0.4)(0.7,-0.3)(0.5,-0.4)(0.4,-0.5)(0.2,-0.5)(0.1,-0.4)(0,-0.2)(0.1,0)(0.2,0.1)(0.4,0.1)(0.5,0)(0.7,-0.1)(0.9,0)(1,0.1)(3/2,0.7)};

\draw ($(1,2)+(90:0.6)$)arc (90:450:0.4);
\fill[color=gray,opacity=0.5] ($(0.4,1)+(90:0.4)$)arc (90:450:0.4)--($(0.4,1)+(450:0.2)$) arc(450:90:0.2);
\draw[dashed] (0.4,1)circle (0.4);
\end{scope}

\begin{scope}[shift={(6,-6)}]
%\draw (-2.5,-2.5) rectangle (2.5,2.5);
\draw (-1,-1) rectangle (1,1);
\node[left] at (-1,0.6) {$B_w$};
\node at (-2.8,1.5) {(iii)};

\draw[thick] plot [smooth,tension=0.5] coordinates {(-0.3,1.5)(-0.2,1.2)(-0.1,0.9)(0.1,0.7)(0.3,0.6)(0.5,0.6)(0.7,0.7)(0.9,1)(1,1.1)(1.2,1.1)(1.3,1)(1,0.3)(0.8,0.1)(0.6,0.1)(0.4,0.3)(0.2,0.3)(0,0.1)(-0.1,-0.2)(0,-0.5)(0.2,-0.7)(0.4,-0.7)(0.6,-0.5)(0.8,-0.5)(1,-0.7)(3/2,-1.4)};

\fill[pattern=north west lines,pattern color=gray] plot [smooth cycle,tension=0.5] coordinates {(-0.3,1.5)(-0.2,1.2)(-0.1,0.9)(0.1,0.7)(0.3,0.6)(0.5,0.6)(0.7,0.7)(0.9,1)(1,1.1)(1.2,1.1)(1.3,1)(1,0.3)(0.8,0.1)(0.6,0.1)(0.4,0.3)(0.2,0.3)(0,0.1)(-0.1,-0.2)(0,-0.5)(0.2,-0.7)(0.4,-0.7)(0.6,-0.5)(0.8,-0.5)(1,-0.7)(3/2,-1.4)(1.7,-1.3)(1.8,-0.5)(1.8,0.7)(1.7,1.3)(1,1.5)(0.2,1.6)};

\node[fill=white] at (1.5,-0.2) {$\infty$};
\node[below] at (-1,-1.2) {$\{\widetilde{f}_v\geq\ell\}_\infty$};
\draw (-1,-1.2)--(1.1,-0.5);
\draw (-1,-1.2)--(0.3,1.1);

\end{scope}

\end{tikzpicture}

    \caption{Three ways in which the volume of the unbounded component restricted to a cube can change: in case (i) the topology within the cube changes and we bound $\lvert\Delta_v(B_w)\rvert$ by $1$. In case (ii) there is no topological change in $B_w$ or in any cube which is connected to $B_w$ by a finite excursion component so we bound $\lvert\Delta_v(B_w)\rvert$ by the volume of the solid grey region in the figure. In case (iii) there is a topological change outside $B_w$ affecting which components inside $B_w$ are part of the unbounded component, so we bound $\lvert\Delta_v(B_w)\rvert$ by $1$.}
    \label{fig:LocalvsNonLocal}
\end{figure}
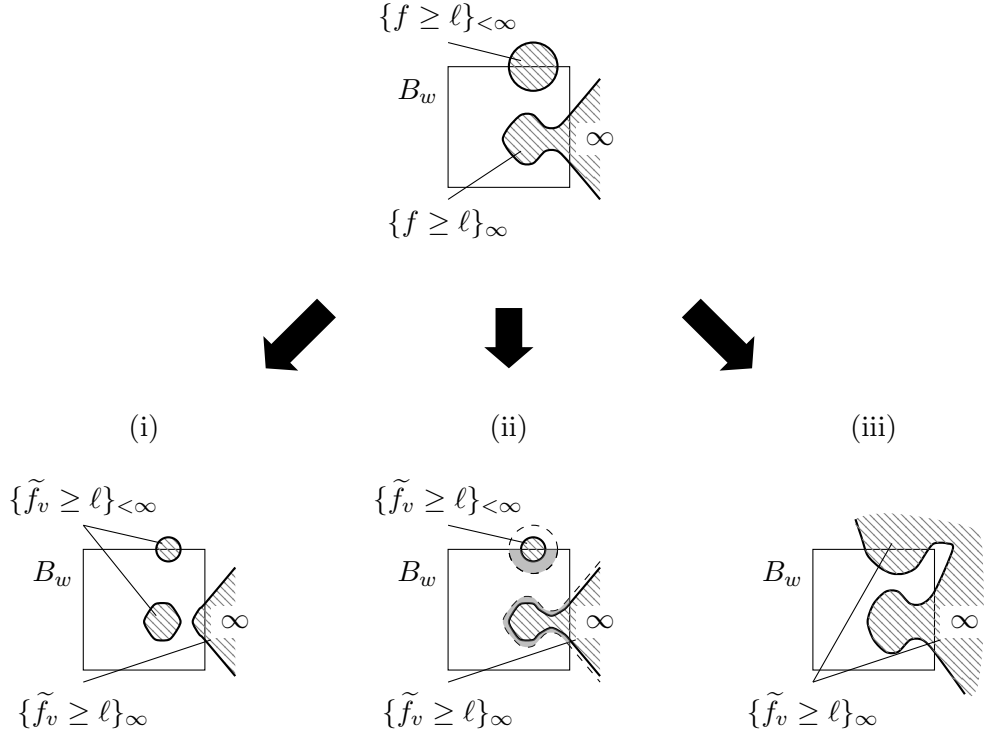

To prove condition \eqref{e:outline_sufficient} for $\mu_\Vol$, we simply sum each of the three contributions above. This requires us to have quantitative decay bounds on (i) the probability of $B_w$ being locally unstable when we resample $W_v$ as $\lvert w-v\rvert\to\infty$ and (ii) the probability of $B_w$ being connected to $B_u$ by a bounded component of $\{f\geq\ell\}$ as $\lvert u-w\rvert\to\infty$. The latter is precisely Assumption~\ref{a:Decay}. The former follows from the fact that
\begin{equation}\label{e:PerturbDef}
    p_v(x):=f(x)-\widetilde{f}_v(x)=\int_{B_v}q(x-u)\;d(W-\widetilde{W}^{(v)})(u)
\end{equation}
and its derivatives are small when $x\in B_w$, which in turn follows from our decay assumptions on $q$.

We prove condition~\eqref{e:outline_sufficient} for $\star=\SA$ and $\star=\EC$ in a very similar way by decomposing the change into local/non-local topological and geometric contributions. In place of the trivial upper bound \eqref{e:TrivialBound} we use the total surface area for $f$ and $\widetilde{f}_v$ in $B_w$ or their total number of critical points when $\star=\SA,\EC$ respectively. This is the underlying reason why Theorem~\ref{t:clt} requires stronger covariance decay (i.e.\ larger values of $\beta$) when $\star=\SA,\EC$ as compared to $\star=\Vol$: our argument interpolates between moment bounds for these upper bounds and decay of the probability that a cube is unstable. The finiteness of moments of surface area and critical points follows from smoothness of the underlying field (i.e.\ larger $k$ in Assumption~\ref{a:clt}) whereas the trivial constant bound on the volume has finite moments of all order without any smoothness assumptions.

\medskip\underline{Positivity of variance:} The proof of Theorem~\ref{t:clt} gives the following semi-explicit expression for the limiting variance:
\begin{equation}\label{e:LimitingVariance}
    \sigma_\star^2(\ell)=\E\Big[\E\big[\lim_{n\to\infty}\Delta_0(\Lambda_n)\;|\;\mathcal{F}_0\big]^2\Big]=\E\Big[\E\big[\lim_{n\to\infty}\mu_\star(\Lambda_n,f)-\mu_\star(\Lambda_n,\widetilde{f}_0)\;|\;\mathcal{F}_0\big]^2\Big].
\end{equation}
Applying the result to a rescaled version of the field $f$, we show that \eqref{e:LimitingVariance} still holds if we instead resample the white noise on $\Lambda_m$, for some large but fixed $m$, rather than $B_0$.  We then find a lower bound for $\sigma^2$ by resampling only the \emph{average value} of the white noise on $\Lambda_m$ (i.e.\ the value of $W(\Lambda_m)=\int\ind_{\Lambda_m}dW(u)$). Since $f=q\ast W$, this resampling is equivalent to adding $Z_0q\ast\ind_{\Lambda_m}$ to $f$ where $Z_0$ is a standard Gaussian variable. When $\int q(x)\;dx\neq 0$ and $m$ is large, the latter is roughly equivalent to perturbing $f$ by a constant on $\Lambda_{m}$ and by some decaying term on $\R^d\setminus\Lambda_{m}$.

Assuming that there is some $\ell^\prime$ for which $c_\star(\ell^\prime)\neq c_\star(\ell)$, where $c_\star$ is defined in Theorem~\ref{t:FirstOrder}, the perturbation on $\Lambda_{m}$ can cause a change in $\mu_\star(\Lambda_n)$ (where $n>>m$) of order $m^d$. Finally we show that the effect on $\mu_\star(\Lambda_n)$  of the perturbation outside $\Lambda_{m}$ is of order $o(m^d)$, so that by taking $m$ sufficiently large the overall effect of the perturbation is non-zero, yielding a positive variance. The latter argument relies on having a statement about truncated connection decay for a perturbed version of $f$. Such a statement follows from subcritical decay (i.e.\ the first part of Theorem~\ref{t:PhaseTransition}) when $\ell<-\ell_c$, which explains why we impose this assumption in Theorem~\ref{t:posvar}. However we note that the statement could alternatively be proven using a `sprinkled' version of truncated arm decay which would conjecturally hold for all $\ell<\ell_c$.

\subsection{Discussion}\label{s:Discussion}
We now discuss some possible extensions of our work and related open questions.

\medskip\underline{Related work:} Our proof of Theorem~\ref{t:clt} builds upon a versatile CLT for stationary functionals of spatial white noise due to Penrose \cite{pen01}. This result unified the approach of several other works which applied the classical martingale CLT to spatial probabilistic processes \cite{kl96,kz97,lee97,zha01}. In particular we note that Zhang \cite{zha01} proved a CLT for the size of the unbounded component of Bernoulli percolation on $\Z^d$ (a discrete analogue of our volume CLT) and Penrose \cite{pen01} proved a corresponding result for the largest component in a finite box.

Additional complications arise in our setting: the functionals $\mu_\star$ have infinite range of dependence on the underlying white noise process and controlling the behaviour of the unbounded component in this continuous setting presents some technical difficulties. Moreover, verifying positivity of the limiting variance requires new techniques. A central limit theorem for the number of connected components of the excursion/level set of a Gaussian field was proven in \cite{bmm24} using methods closely related to those in this work. Although controlling changes in the latter functional under perturbation is simpler as there are no geometric or non-local topological contributions. More recently these methods have been adapted to prove a functional central limit theorem for more general topological functionals (including the Betti numbers) \cite{hlr24}.

\medskip\underline{Truncated arm decay:} Whilst our results are valid for all supercritical levels in the planar case, the assumption of truncated connection decay has not been verified (for all levels) in higher dimensions. Based on the behaviour of Bernoulli percolation \cite{ccn87} and similarities between this model and fast-decay Gaussian fields (see \cite{bel23} and references therein) it is natural to conjecture that truncated connection probabilities should decay exponentially at any supercritical level for fields satisfying Assumption~\ref{a:clt}. That is, we would expect for each $\ell<\ell_c$ there exists $C,c>0$ such that
\begin{displaymath}
    \P\Big(\Lambda_1\overset{\{f\geq\ell\}_{<\infty}}{\longleftrightarrow}\partial\Lambda_n\Big)<Ce^{-cn}\qquad\text{for all }n\in\N.
\end{displaymath}
See \cite[Section~1.3]{sev21} for further details and references. A proof of the latter statement, which seems to require new ideas when compared to the case of Bernoulli percolation, would immediately extend our CLTs to all supercritical levels in $d\geq 3$.

\medskip\underline{Martingale methods for non-local functionals:} A geometric functional of a field is said to be \emph{local} if it can be expressed as an integral of a pointwise function of the field (and possibly its derivatives). For example, the volume of the excursion set of a field is local but the volume of the unbounded component of the excursion set is \emph{non-local} since one needs to know which parts of the excursion set belong to this component to compute the volume. The statistics of many local functionals are well understood: under conditions roughly equivalent to Assumption~\ref{a:clt}, CLTs have been proven for local functionals including the Lipschitz-Killing curvatures of excursion sets \cite{el16,mul17,kv18} and the number of critical points \cite{nic17}. These results are proven using the Wiener chaos expansion (see \cite[Chapter 2]{jan97} for background). The Wiener chaos expansion of non-local functionals is generally intractable, and so these quantities have proven more challenging to analyse. Although recently the chaos expansion method has been successfully applied to a non-local functional of a discrete Gaussian field \cite{mm25}.

One of the key takeaways we would like to impress upon the reader, is the potential value of martingale methods in studying non-local functionals. Our results show that such methods offer significant insight into $\mu_\Vol$, $\mu_\EC$ and $\mu_\SA$ (all of which are non-local). Moreover the same general approach was applied to the component count (which is also non-local) in \cite{bmm24} and to higher order Betti numbers in \cite{hlr24}. Theorem~\ref{t:GenCLT} gives general sufficient conditions for such functionals to satisfy a CLT. In Proposition~\ref{p:Sufficient_Moments}, we give a simpler set of sufficient conditions for functionals that are local when restricted to the unbounded component. With a little work, this result would likely allow one to prove CLTs for the following quantities (restricted to $\Lambda_n$ as $n\to\infty$):
\begin{itemize}
    \item the number of critical points of $f$ in $\{f\geq\ell\}_\infty$,
    \item the volume of $\{f\geq\ell\}_\infty$ restricted to some fixed lower dimensional hyperplane,
    \item the volume of $\{f_1\geq\ell\}_\infty\cap\{f_2\geq\ell\}_\infty$ for two independent fields $f_1$ and $f_2$.
\end{itemize}
One could also study `supercritical geometry' of Gaussian fields via the unbounded components of $\{f=\ell\}$ or $\{\ell_1\leq f\leq\ell_2\}$. These components have been proven to exist for certain fields and levels \cite{drrv23,mrv23} although more work is needed for a full characterisation.

The use of martingale methods in our work relies on the existence of a convolution-white noise representation (i.e.\ $f=q\ast W$ for suitable $q$). This imposes certain restrictions on $f$, including the fact that it must have integrable covariance function. It is an open question as to whether there are alternative representations of more general Gaussian fields which would also be amenable to martingale arguments.

\medskip\underline{Geometry and dependence structure:} Recall that the motivation for our work (Questions~\ref{q:Geometry} and~\ref{q:Dependence}) was to understand the geometry of the unbounded excursion component and its dependence on the distribution of $f$. Our results do not provide any evidence for major differences in supercritical geometric behaviour amongst fields, but known results for other functionals suggest that we should expect this. Let us consider three broad classes of dependence structure and their resulting geometric properties:
\begin{enumerate}
    \item \textbf{Short-range correlations}\\
    We call a (smooth, stationary) Gaussian field \emph{short-range correlated} if $K\in L^1(\R^d)$ (where $K$ is the covariance function). Morally speaking, this is roughly equivalent to Assumption~\ref{a:clt} for some $\beta>d$. The quintessential example is the Bargmann-Fock field although we also mention the family of Cauchy fields with covariance function $K_\beta(x):=(1+\lvert x\rvert^2)^{-\beta/2}$ for $\beta>d$.

    A variety of geometric functionals restricted to growing boxes $\Lambda_n$ are known to have similar behaviour for short-range correlated fields; namely, the mean and variance of these quantities scale like the volume $n^d$ and an asymptotic CLT holds with Gaussian limit. Specifically under Assumption~\ref{a:clt}, this behaviour has been verified for the Lipschitz-Killing curvatures \cite{el16,mul17,kv18}, the number of critical points \cite{nic17,agls25} and the component count \cite{bmm24} (although the last result was proven only in the case that $\beta>3d$).
    \item \textbf{Regularly varying long-range correlations}\\
    We say that a field $f$ is \emph{regularly varying with index $\beta\in(0,d)$ and remainder $L$} if $K(x)=\lvert x\rvert^{-\beta}L(\lvert x\rvert)$ where $L$ is slowly varying at infinity. (The reader unfamiliar with this terminology can consult \cite{bgt87} for a definition or just think of $K$ as decaying like $\lvert x\rvert^{-\beta}$ at infinity.). Cauchy fields with covariance kernel $K_\beta$ for $\beta\in(0,d)$ fall within this class.

    Geometric properties of such fields are less well characterised compared to the short-range correlated case but may be expected to generally exhibit super-volume order variance scaling. That is, the variance of the geometric functionals considered above (Lipschitz-Killing curvatures, component count etc) restricted to the box $\Lambda_n$ could have variance of order $n^{2d-\beta}L(n)$ for some slowly varying function $L$. In general the limiting fluctuations may be Gaussian or non-Gaussian; see \cite{leo99} for a thorough development of such results for local functionals. Analogous behaviour has recently been proven in \cite{mm25} for a non-local functional of a discrete Gaussian field, while upper and lower variance bounds of the conjectured order have been established for the component count of smooth fields \cite{bmm22,bmm24b}.
    \item \textbf{Oscillating long-range correlations}\\
    We say that a field has \emph{oscillating long-range correlations} if $K$ is not integrable and can take negative values. The most important representative of this class is the \emph{monochromatic random wave}: the field with spectral measure supported uniformly on $\mathbb{S}^{d-1}$.

    Fascinating geometric results have been proven for monochromatic random waves, mainly in the two-dimensional case. The broad picture which emerges for a number of functionals related to excursion sets is the following: at most levels $\ell$ the functionals of $\{f\geq\ell\}$ have super-volume order scaling for the variance and a CLT is known. However for a small number of \emph{anomalous levels} the functionals have variance of lower order (typically proportional to the volume, up to some logarithmic factors) and CLTs are sometimes known in these cases. This pattern of behaviour has been observed for the area \cite{mw11,mar14}, boundary length \cite{npr19,ros16}, Euler characteristic \cite{cmw16,cm18} and number of critical points \cite{cmw16b,cm20} of two-dimensional monochromatic random waves. (To be more precise, most of these results hold for families of Gaussian fields on the sphere which converge locally to the monochromatic random wave but we expect the behaviour to be similar.) Progress has been made recently in extending these results to more general covariance structures using spectral arguments \cite{mn24,gas25}.

    The phenomenon of anomalous levels was first observed in \cite{ber02} for the length of the nodal set (i.e.\ the zero level set) and is known, at least in this case, as \emph{Berry cancellation}. The previous results have been proven using either the Kac-Rice formula (see \cite[Chapter~6]{aw09} for background) or the Wiener chaos expansion. Both of these methods rely on locality of the functionals.
\end{enumerate}
So should we expect the previous relationship between geometric functionals and dependence structure to manifest for the unbounded excursion component?

Our results verify analogous behaviour for a fairly general subset of the short-range correlated fields, so we expect this holds for the entire short-range correlated class. For a large class of fields with regularly varying long-range correlations, a sharp phase transition is known to occur \cite{ms24,mui24} whilst for the monochromatic random wave, a non-trivial phase transition has been proven \cite{mrv23,mui23}. Hence one can legitimately ask about the geometry of the unbounded excursion component, but our results provide no insight as such fields fail to have a convolution-white noise decomposition. We can only pose the following questions:
\begin{question}
    If a field is regularly varying with index $\beta\in(0,d)$, do the local functionals of its unbounded excursion component restricted to $\Lambda_n$ have variance of order $n^{2d-\beta}$? Are there anomalous levels for which the variance has lower order? Are the limiting fluctuations Gaussian?
\end{question}
\begin{question}
    For monochromatic random waves in dimension $d$, do the local functionals of the unbounded excursion component restricted to $\Lambda_n$ have variance of order $n^{d+1}$? Are there anomalous levels with lower order variance? Are the limiting fluctuations Gaussian?
\end{question}
The suggested variance order $n^{d+1}$ in the latter case is based on results for the volume of excursion sets of a model related to monochromatic random waves in dimensions $d\geq 2$ \cite{mar15}.

Proving the existence of anomalous levels for non-local variables, such as those related to the unbounded component, could be particularly challenging as the Wiener chaos expansion may be intractable. We note that upper and lower bounds on the variance of the component count (a non-local variable) have been proven for fields with long-range correlations \cite{bmm22,bmm24b} and the former gives a necessary condition for anomalous levels. However it seems that the methods used in these works cannot be extended to prove the existence of anomalous levels. Whilst our methods cannot be applied directly to answer any of the above questions, the general approach of using an abstract martingale CLT along with some decomposition of the field might hold promise.

\subsection{Acknowledgements}
Part of this work was completed at the Department of Mathematics and Statistics at the University of Helsinki and was supported by the European Research Council Advanced Grant QFPROBA (grant number 741487). I would like to thank Stephen Muirhead for suggesting Theorem~\ref{t:FirstOrder} and conversations which led to generalising Theorem~\ref{t:posvar}. I would also like to thank Dmitry Beliaev for comments which improved the clarity of Section~\ref{s:introduction} and previous collaboration which inspired this work. Finally I would like to thank two anonymous reviewers for comments and corrections improving this work and in particular for suggesting the more general version of Theorem~\ref{t:FirstOrder} which now appears.

\section{A CLT for white noise functionals}\label{s:CompletingProof}
In this section we prove a CLT for stationary white noise functionals which generalises the result of Penrose \cite{pen01} to the case of infinite range of dependence. We then derive a simpler sufficient condition for this CLT to hold when the white noise functional is (approximately) additive.

We begin by stating an abstract CLT for martingale arrays. Recall that $\preceq$ denotes the standard lexicographic ordering on $\Z^d$. Let $\{\mathcal{F}_{n,v}\;|\;v\in\Z^d,n\in\N\}$ be a set of $\sigma$-algebras such that for each $n$, $v\preceq w$ implies $\mathcal{F}_{n,v}\subseteq\mathcal{F}_{n,w}$. We say that a collection of random variables $\{S_{n,v}\;|\;v\in\Z^d,n\in\N\}$ is a \emph{lexicographic martingale array} if the following holds:
\begin{enumerate}
    \item $S_{n,v}$ is $\mathcal{F}_{n,v}$-measurable for each $n$ and $v$,
    \item if $v\preceq w$ then $\E[S_{n,w}\;|\;\mathcal{F}_{n,v}]=S_{n,v}$ for each $n$.
\end{enumerate}
We say that the array is mean-zero if $\E[S_{n,v}]=0$ for all $n,v$ and square integrable if $\sup_{v\in\Z^d}\E[S_{n,v}^2]<\infty$ for each $n\in\N$. We say that a sequence of points in $\Z^d$ converges to $\pm\infty^*$ if all coordinates of the points converge to $\pm\infty$. For a square-integrable martingale array, the $L^p$-convergence theorems for forward/reverse martingales imply that the following limits exist
\begin{displaymath}
    S_{n,\infty^*}:=\lim_{v\to\infty^*}S_{n,v}\qquad\text{and}\qquad S_{n,-\infty^*}:=\lim_{v\to-\infty^*}S_{n,v}
\end{displaymath}
where convergence occurs almost surely and in $L^2$.

Finally we wish to impose a condition on our array which ensures that the martingale (or equivalently the family of $\sigma$-algebras $\mathcal{F}_{n,v}$) does not make any `jumps' when some coordinates tend to infinity. (To illustrate; for $d=2$ we would like to know that $\lim_{k\to\infty}S_{n,(0,k)}=\lim_{j\to-\infty}S_{n,(1,j)}$ almost surely, so that the important behaviour of the martingale can be captured on bounded domains.) We therefore say that $S_{n,v}$ is \emph{regular at infinity} if for all $n\in\N$, $i\in\{1,\dots,d-1\}$ and $(v_1,\dots,v_i)\in\Z^i$
\begin{displaymath}
    \lim_{v_{i+1},\dots,v_d\to\infty}S_{n,v}=\lim_{v_{i+1},\dots,v_d\to-\infty}S_{n,v^\prime}
\end{displaymath}
almost surely where $v=(v_1,\dots,v_d)$ and $v^\prime=(v_1,\dots,v_{i-1},v_i+1,v_{i+1},\dots,v_d)$. For $v\in\Z^d$ we let $v^-$ denote the element of $\Z^d$ immediately preceding $v$ in the lexicographic order and we define the differences of the martingale array as $U_{n,v}:=S_{n,v}-S_{n,v^-}$. A straightforward consequence of being regular at infinity is that
\begin{displaymath}
    S_{n,\infty^*}-S_{n,-\infty^*}=\sum_{v\in\Z^d}U_{n,v}
\end{displaymath}
whenever $\sum_{v\in\Z^d}\lvert U_{n,v}\rvert<\infty$.

The following result was proven in \cite{bmm24}. It is a straightforward generalisation of a classical CLT for finite martingale arrays which can be found, for example, in \cite[Chapter~3]{hh80}.

\begin{theorem}\label{t:Lex_MCLT}
Let $\{S_{n,v},\mathcal{F}_{n,v}:v\in\Z^d, n\in\N\}$ be a mean-zero square-integrable lexicographic martingale array which is regular at infinity such that $S_{n,-\infty^*}=0$ for each $n$. Suppose that
\begin{align}
&\sup_{v\in\Z^d}\lvert U_{n,v}\rvert\xrightarrow{p}0\quad\text{as }n\to\infty\label{e:MCLT1}\\
&\sup_n\E\Big[\sup_{v\in\Z^d}U_{n,v}^2\Big]<\infty\label{e:MCLT2}\\
&\sum_{v\in\Z^d}U_{n,v}^2\xrightarrow{L^1}\eta^2\in[0,\infty)\quad\text{as }n\to\infty\label{e:MCLT3}\\
&\E\Big[\sum_{v\in\Z^d}\lvert U_{n,v}\rvert\Big]<\infty\quad\text{for all }n\in\N.\label{e:MCLT4}
\end{align}
Then $\mathrm{Var}[S_{n,\infty^*}]\to\eta^2$ and $S_{n,\infty^*}\xrightarrow{d}Z$ as $n\to\infty$ where $Z\sim\mathcal{N}(0,\eta^2)$.
\end{theorem}

We now state our generalised CLT. Recall that $W$ is a Gaussian white noise on $\R^d$ and $W_v$ denotes its restriction to the unit cube $B_v$ for $v\in\Z^d$. For each $n\in\N$ and $v\in\Z^d$ we let $\mu(v+\Lambda_n):=F(v+\Lambda_n,W)$ where $F$ is some deterministic functional. We say that $\mu$ is a \emph{stationary white-noise functional} if $\mu(v+\Lambda_n)$ is unchanged when we translate $v$ and $W$ by the same vector in $\Z^d$ (i.e.\ $F(u+v+\Lambda_n,W(\cdot-u))=F(v+\Lambda_n,W)$ for any $u\in\Z^d$). We define the differences of $\mu$ as
\begin{displaymath}
    \Delta_u(v+\Lambda_n)=F(v+\Lambda_n,W)-F(v+\Lambda_n,\widetilde{W}^{(u)})
\end{displaymath}
where we recall that $\widetilde{W}^{(u)}$ denotes the white noise $W$ with $W_u$ resampled independently.

\begin{theorem}\label{t:GenCLT}
    Let $\mu$ be a stationary white-noise functional satisfying the following:
\begin{enumerate}
\item (Finite second moments) $\E[\mu(v+\Lambda_n)^2]<\infty$ for all $v\in \Z^d$ and $n\in\N$,
\item (Stabilisation) there exists a random variable $\Delta_0$ such that for any sequence of cubes $D_n:=v_n+\Lambda_n$ satisfying $\liminf_nD_n=\R^d$ we have
\begin{displaymath}
\Delta_0(D_n)\to \Delta_0    
\end{displaymath}
in probability as $n\to\infty$,
\item (Bounded moments) There exists $\epsilon>0$ such that
\begin{displaymath}
\sup_{v\in\Z^d,n\in\N}\E\left[\lvert\Delta_v(\Lambda_n)\rvert^{2+\epsilon}\right]<\infty,
\end{displaymath}
\item (Moment decay) There exists $c,\epsilon,\zeta>0$ and $\gamma>\max\{d,\zeta\}$ such that for all $v\in\Z^d\setminus\Lambda_n$
\begin{displaymath}
\E\left[\lvert\Delta_v(\Lambda_n)\rvert+\lvert\Delta_v(\Lambda_n)\rvert^{2+\epsilon}\right]\leq c\;n^\zeta\;\dist(v,\Lambda_n)^{-\gamma}
\end{displaymath}
where $\dist(v,\Lambda_n)$ denotes the Euclidean distance from $v$ to $\Lambda_n$.
\end{enumerate}
Let $\mathcal{F}_v$ be the $\sigma$-algebra generated by $\{W_u\;|\;u\preceq v\}$. Then as $n\to\infty$,
\begin{displaymath}
    \frac{\mathrm{Var}[\mu(\Lambda_n)]}{(2n)^d}\to\sigma^2\quad\text{and}\quad \frac{\mu(\Lambda_n)-\E[\mu(\Lambda_n)]}{(2n)^{d/2}}\xrightarrow{d}\sigma Z
\end{displaymath}
where $Z\sim\mathcal{N}(0,1)$ and $\sigma^2=\E\big[\E[\Delta_0|\mathcal{F}_0]^2\big]$.
\end{theorem}

\begin{remark}
    This result and its proof closely follow that of Penrose \cite[Theorem~2.1]{pen01} with two differences: first, Penrose considered more general sequences of growing domains whereas we work only with the cubes $(\Lambda_n)_{n\in\N}$ for simplicity, and second, Penrose assumed that $\mu(\Lambda_n)$ depends only on the white noise restricted to $\Lambda_n$ whereas we allow $\mu$ to have potentially infinite range of dependence. This is why we require a lexicographic version of the martingale CLT (i.e.\ Theorem~\ref{t:Lex_MCLT}) rather than the standard version. To control such infinite-range functionals, we impose condition (4) which is not required for Penrose's result.
\end{remark}

%Our proof follows that of Penrose \cite[Theorem~2.1]{pen01} with some additional arguments to deal with functionals having an infinite range of dependence. More or less the same proof was given for one particular such functional (the component count) in \cite{bmm24}.

First we specify the martingale to which we will apply Theorem~\ref{t:Lex_MCLT}:

\begin{lemma}\label{l:Regular}
    Let $\mu$ satisfy the conditions of Theorem~\ref{t:GenCLT}, then
    \begin{displaymath}
        S_{n,v}:=(2n)^{-d/2}(\E[\mu(\Lambda_n)\;|\;\mathcal{F}_v]-\E[\mu(\Lambda_n)])\quad\text{and}\quad\mathcal{F}_{n,v}:=\mathcal{F}_v
    \end{displaymath}
    define a mean-zero, square-integrable lexicographic martingale array which is regular at infinity. Moreover $S_{n,-\infty^*}=0$ and $S_{n,\infty^*}=(2n)^{-d/2}(\mu(\Lambda_n)-\E[\mu(\Lambda_n)])$.
\end{lemma}
\begin{proof}
    It is clear from the definition that $S_{n,v}$ is a mean-zero lexicographic martingale array. Since $\mu$ has finite second moments, the array is square integrable. By L\'evy's downward and upward convergence theorems respectively
    \begin{align*}
        \lim_{v\to-\infty^*}\E[\mu(\Lambda_n)\;|\;\mathcal{F}_v]&=\E[\mu(\Lambda_n)\;|\;\cap_{v\in\Z^d}\mathcal{F}_v]=\E[\mu(\Lambda_n)]\\
        \lim_{v\to\infty^*}\E[\mu(\Lambda_n)\;|\;\mathcal{F}_v]&=\E[\mu(\Lambda_n)\;|\;\sigma(\cup_{v\in\Z^d}\mathcal{F}_v)]=\mu(\Lambda_n)
    \end{align*}
    where the right-most equalities hold because the tail $\sigma$-algebra $\cap_{v\in\Z^d}\mathcal{F}_v$ is trivial and $W$ is measurable with respect to $\sigma(\cup_{v\in\Z^d}\mathcal{F}_v)$. Hence $S_{n,-\infty^*}=0$ and $S_{n,\infty^*}=\mu(\Lambda_n)$, as required.
    
    It remains to show that $S_{n,v}$ is regular at infinity. We fix two sequences $j^{(m)}$ and $k^{(m)}$ in $\Z^d$ such that
    \begin{align*}
        j^{(m)}&=(u_1,\dots,u_i,a^{(m)}_1,\dots,a^{(m)}_{d-i})\\
        k^{(m)}&=(u_1,\dots,u_i+1,b^{(m)}_1,\dots,b^{(m)}_{d-i})
    \end{align*}
    where $a^{(m)}_1,\dots,a^{(m)}_{d-i}\to\infty$ and $b^{(m)}_1,\dots,b^{(m)}_{d-i}\to-\infty$ with $m$. By L\'evy's upward and downward theorems, as $m\to\infty$ we have
    \begin{align*}
        \E[\mu(\Lambda_n)\;|\;\mathcal{F}_{j^{(m)}}]&\to\E[\mu(\Lambda_n)\;|\;\mathcal{F}^-]\\
        \E[\mu(\Lambda_n)\;|\;\mathcal{F}_{k^{(m)}}]&\to\E[\mu(\Lambda_n)\;|\;\mathcal{F}^+]
    \end{align*}
    where
    \begin{displaymath}
        \mathcal{F}^-=\sigma\Big(\bigcup_{m\in\N}\mathcal{F}_{j^{(m)}}\Big)\quad\text{and}\quad\mathcal{F}^+=\bigcap_{m\in\N}\mathcal{F}_{k^{(m)}}.
    \end{displaymath}
    Regularity at infinity then follows if we can show that the completions of $\mathcal{F}^+$ and $\mathcal{F}^-$ coincide. We note that $\mathcal{F}^+$ is generated by events in $\mathcal{F}^-$ together with those measurable with respect to a `tail' of independent variables. We can therefore argue by generalising the proof of Kolmogorov's zero-one law. Specifically for $A\in\mathcal{F}^+$, defining
    \begin{align*}
        \mathcal{G}_m:=\sigma(W_v\;|\;k^{(m)}\preceq v\preceq k^{(1)})\quad\text{and}\quad\mathcal{G}_\infty=\sigma\Big(\bigcup_m\mathcal{G}_m\Big)
    \end{align*}
    we may apply L\'evy's upward theorem once more to see that
    \begin{equation}\label{e:LevyUpward}
        \E[\ind_A\;|\;\sigma(\mathcal{F}^-,\mathcal{G}_m)]\to \E[\ind_A\;|\;\sigma(\mathcal{F}^-,\mathcal{G}_\infty)]=\ind_A
    \end{equation}
    where the final equality follows since $\sigma(\mathcal{F}^-,\mathcal{G}_\infty)\supseteq\mathcal{F}_{k^{(1)}}\supseteq\mathcal{F}^+$. However since $\mathcal{F}^+$ is independent of $\mathcal{G}_m$ we have
    \begin{displaymath}
        \E[\ind_A\;|\;\sigma(\mathcal{F}^-,\mathcal{G}_m)]=\E[\ind_A\;|\;\mathcal{F}^-].
    \end{displaymath}
    Combined with \eqref{e:LevyUpward} this implies that $A$ is measurable with respect to the completion of $\mathcal{F}^-$, as required.
\end{proof}

Next we record an elementary lemma which we will use repeatedly in calculations. Recall that $\dist(\cdot,\cdot)$ denotes the Euclidean distance between two sets (or a point and a set).
\begin{lemma}\label{l:ElementarySum}
    Let $\gamma>d$, then there exists a constant $c>0$ depending only on $d$ and $\gamma$ such that for $n\in\N$ and $r\geq 1$
    \begin{displaymath}
        \sum_{v\in\Z^d\;:\;\dist(v,\Lambda_n)>r}\dist(v,\Lambda_n)^{-\gamma}\leq cr^{-\gamma+1}\max\{r,n\}^{d-1}.
    \end{displaymath}
\end{lemma}
\begin{proof}
    For $v\in\Z^d$ let $C(v)$ be the point in $\Lambda_n\cap\Z^d$ nearest to $v$. The idea of the proof is to partition $\{v\;|\;\dist(v,\Lambda_n)>r)\}$ according to the point in $\Lambda_n\cap\Z^d$ which $v$ is nearest to.

    For $i=0,1,\dots,d$ let $\mathrm{Face}_i$ denote the points $x\in\Lambda_n\cap\Z^d$ such that $d+i$ of their nearest neighbours (in $\Z^d$) are also contained in $\Lambda_n$. So for example $\mathrm{Face}_0$ denotes the corners of the cube $\Lambda_n$ and $\mathrm{Face}_d$ denotes the points of $\Z^d$ in the interior of $\Lambda_n$. We note that the number of points in $\mathrm{Face}_i$ is at most $c_d n^i$, where $c_d>0$ is some constant depending only on $d$, by elementary geometric considerations.

    We define $S_x=\{v\in\Z^d\;|\;\dist(v,\Lambda_n)>r, C(v)=x\}$. Observe that $S_x=\emptyset$ whenever $x\in\mathrm{Face}_d$. Moreover when $x\in\mathrm{Face}_i$ and $v\in S_x$, $x-v$ must be orthogonal to each of the $i$ directions in which both neighbours of $x$ are contained in $\Lambda_n$. In other words $S_x-x$ is contained in a subspace of dimension $d-i$ and hence
    \begin{displaymath}
        \sum_{v\in S_x}\lvert v-x\rvert^{-\gamma}\leq \sum_{y\in\Z^{d-i}\;:\;\lvert y\rvert>r}\lvert y\rvert^{-\gamma}\leq c_\gamma r^{-\gamma+d-i}.
    \end{displaymath}
    We then conclude that
    \begin{align*}
        \sum_{v\in\Z^d\;:\;\dist(v,\Lambda_n)>r}\dist(v,\Lambda_n)^{-\gamma}=\sum_{i=0}^{d-1}\sum_{x\in\mathrm{Face}_i}\sum_{v\in S_x}\lvert v-x\rvert^{-\gamma}&\leq \sum_{i=0}^{d-1}c_d n^i c_\gamma r^{-\gamma+d-i}\leq cr^{-\gamma+1}\max\{r,n\}^{d-1}
    \end{align*}
    as required.
\end{proof}

Finally we state a version of the ergodic theorem which we will make use of in proving Theorem~\ref{t:GenCLT}:

\begin{theorem}[Multi-variate ergodic theorem {\cite[Theorem~25.12]{kal21}}]\label{t:ergodic}
    Let $\xi$ be a random element in some set $S$ with distribution $\nu$. Let $T_1,\dots,T_d$ be $\nu$-preserving transformations of $S$. Assume that the invariant $\sigma$-algebra of each $T_i$ is trivial and let $G\in L^p(\nu)$ for some $p>1$, then as $n\to\infty$
    \begin{displaymath}
        \frac{1}{n^d}\sum_{i=1}^d\sum_{k_i\leq n}G(T_1^{k_1}\dots T_d^{k_d}\xi)\to\E[G(\xi)]
    \end{displaymath}
    where convergence occurs almost surely and in $L^p$.
\end{theorem}

\begin{proof}[Proof of Theorem~\ref{t:GenCLT}]
By Lemma~\ref{l:Regular}, the variance convergence and CLT that we wish to prove will follow if we can verify the numbered conditions in Theorem~\ref{t:Lex_MCLT} (along the way we will also derive the claimed expression for the limiting variance $\sigma^2$).

First we observe, since $W$ is independent on disjoint domains, that
\begin{displaymath}
    U_{n,v}=(2n)^{-d/2}\E[\Delta_v(\Lambda_n)\;|\;\mathcal{F}_v]
\end{displaymath}
almost surely. Combining the bounded moments and moment decay assumptions with Lemma~\ref{l:ElementarySum} shows that
\begin{equation}\label{e:CLTProof0}
    \sum_{v\in\Z^d}\E\big[\lvert \Delta_v(\Lambda_n)\rvert^{2+\epsilon}\big]\leq cn^d +\sum_{v\in\Z^d\setminus \Lambda_{2n}}cn^\zeta\;\dist(v,\Lambda_n)^{-\gamma}\leq c^\prime n^d
\end{equation}
for some $\epsilon,c,c^\prime>0$ (since $\gamma>\max\{d,\zeta\}$). By interpolating the moment decay assumption (i.e.\ using Littlewood's $L^p$ inequality) we have
\begin{equation}\label{e:CLTProof0.5}
    \E\big[\lvert \Delta_v(\Lambda_n)\rvert^2\big]\leq c n^\zeta\dist(v,\Lambda_n)^{-\gamma}.
\end{equation}
Hence by Jensen's inequality and the bounded moments condition, we see that \eqref{e:CLTProof0} also holds for $\epsilon=0$ (with a different constant $c^\prime$). Applying Markov's inequality and the conditional form of Jensen's inequality, for any $\delta>0$
\begin{align*}
    \P\Big[\sup_{v\in\Z^d}\lvert U_{n,v}\rvert>\delta\Big]\;\delta^{2+\epsilon}\leq\sum_{v\in\Z^d}\E\big[\lvert U_{n,v}\rvert^{2+\epsilon}\big]\leq (2n)^{-d\frac{2+\epsilon}{2}}\sum_{v\in\Z^d}\E\big[\lvert \Delta_v(\Lambda_n)\rvert^{2+\epsilon}\big]\leq cn^{-d\frac{\epsilon}{2}}
\end{align*}
which tends to zero, verifying \eqref{e:MCLT1}. By identical reasoning in the case $\epsilon=0$, we have
\begin{align*}
    \E\Big[\sup_{v\in\Z^d}U_{n,v}^2\Big]\leq\E\Big[\sum_{v\in\Z^d}U_{n,v}^2\Big]\leq (2n)^{-d}\E\Big[\sum_{v\in\Z^d}\lvert\Delta_v(\Lambda_n)\rvert^2\Big]\leq c
\end{align*}
where $c$ is independent of $n$, verifying \eqref{e:MCLT2}. Applying the conditional and unconditional forms of Jensen's inequality, the bounded moment/moment decay assumptions and Lemma~\ref{l:ElementarySum} yields that for each $n\in\N$
\begin{align*}
    (2n)^{d/2}\E\Big[\sum_{v\in\Z^d}\lvert U_{n,v}\rvert\Big]&\leq \sum_{v\in\Z^d}\E\Big[\lvert \Delta_v(\Lambda_n)\rvert\Big]\\
    &\leq \sum_{\dist(v,\Lambda_n)\leq n}\E\Big[\lvert \Delta_v(\Lambda_n)\rvert^{2+\epsilon}\Big]^{\frac{1}{2+\epsilon}}+\sum_{\dist(v,\Lambda_n)>n}cn^\zeta\;\dist(v,\Lambda_n)^{-\gamma}<\infty
\end{align*}
verifying \eqref{e:MCLT4}.

It remains to verify \eqref{e:MCLT3}. For $v\in\Z^d$ let $\tau_v$ denote translation by $v$. Let $D_n=w_n+\Lambda_n$ where $w_1,w_2,\dots\in\Z^d$ then by the definitions of $\mu$ and $\Delta_v$, for any $v\in\Z^d$ the sequences of random variables $\Delta_v(D_n)$ and $\Delta_0(\tau_{-v}D_n)$ have the same distribution. Hence if $\liminf_nD_n=\R^d$, then by the stabilisation assumption there exists a random variable $\Delta_v$ such that $\Delta_v(D_n)\to\Delta_v$ in probability. We now simplify notation by defining
\begin{displaymath}
    X_v(\Lambda_n):=\E[\Delta_v(\Lambda_n)\;|\;\mathcal{F}_v]\qquad\text{and}\qquad X_v:=\E[\Delta_v\;|\;\mathcal{F}_v]
\end{displaymath}
so that the statement we need to prove is
\begin{equation}\label{e:CLTProof1}
    (2n)^{-d}\sum_{v\in\Z^d}X_v^2(\Lambda_n)\xrightarrow{L^1}\sigma^2,\qquad\text{as }n\to\infty.
\end{equation}
We do this in three steps: (i) we show that the contribution to this sum from terms $X_v^2(\Lambda_n)$ where $v$ is outside $\Lambda_n$ or near the boundary of $\Lambda_n$ is small, (ii) we show that the remaining terms are well approximated in $L^1$ by their limits $X_v^2$ and (iii) we apply the ergodic theorem to the sum of these limit terms.

If we choose $\lambda\in(0,1)$ sufficiently small, then by \eqref{e:CLTProof0.5} and the conditional Jensen inequality
\begin{equation}\label{e:CLTProof2}
    (2n)^{-d}\sum_{v\;:\;\dist(v,\Lambda_n)>n^{1-\lambda}}\E[X_v^2(\Lambda_n)]\leq (2n)^{-d}\sum_{v:\dist(v,\Lambda_n)>n^{1-\lambda}}cn^\zeta\;\dist(v,\Lambda_n)^{-\gamma}\to 0
\end{equation}
as $n\to\infty$ (using Lemma~\ref{l:ElementarySum}). By the bounded moments condition, for the same value of $\lambda>0$ we have
\begin{equation}\label{e:CLTProof3}
\begin{aligned}
    (2n)^{-d}\sum_{v\;:\;\dist(v,\partial\Lambda_n)\leq n^{1-\lambda}}\E[X_v^2(\Lambda_n)]&\leq (2n)^{-d}\sum_{v\;:\;\dist(v,\partial\Lambda_n)\leq n^{1-\lambda}}\E[X_v^{2+\epsilon}(\Lambda_n)]^\frac{1}{2+\epsilon}\leq c n^{-\lambda}\to 0  
\end{aligned}
\end{equation}
as $n\to\infty$.

Now let $n^-=n-n^{1-\lambda}$. We claim that
\begin{equation}\label{e:CLTProof4}
    \max_{v\in\Z^d\cap\Lambda_{n^{-}}}\E\Big[\big\lvert X_v^2(\Lambda_n)-X_v^2\big\rvert\Big]\to 0\quad\text{as }n\to\infty.
\end{equation}
Let $v_n\in\Z^d$ be the vertex which maximises the moment above for a given $n$ and let $D_n=-v_n+\Lambda_n$. Observe that $\liminf_nD_n=\R^d$ since $\dist(v_n,\partial\Lambda_n)\to\infty$. Therefore $\Delta_v(D_n)\to\Delta_v$ in probability. Since the $(2+\epsilon)$-moments of $\Delta_v(D_n)$ are bounded uniformly over $n$, Vitali's convergence theorem implies that $\Delta_v(D_n)\to\Delta_v$ in $L^2$. Then using the fact that conditional expectation is a contraction on $L^p$ we have
\begin{displaymath}
    X_v(D_n)=\E[\Delta_v(D_n)\;|\;\mathcal{F}_v]\xrightarrow{L^2}\E[\Delta_v\;|\;\mathcal{F}_v]=X_v
\end{displaymath}
and in particular \eqref{e:CLTProof4} holds.

Combining \eqref{e:CLTProof2}, \eqref{e:CLTProof3} and \eqref{e:CLTProof4} we see that
\begin{equation}\label{e:CLTProof5}
    (2n)^{-d}\sum_{v\in\Z^d}X_v^2(\Lambda_n)-(2n)^{-d}\sum_{v\in\Z^d\cap\Lambda_{n^-}}X_v^2\xrightarrow{L^1}0.
\end{equation}
It remains for us to apply the ergodic theorem to the finite sum above. Let $\xi=(W_v,W_v^\prime)_{v\in\Z^d}$ denote the pair of white noise processes used to define $\mu$ and $\Delta_v(\cdot)$. Let $T_1,\dots,T_d$ respectively denote translation by distance one in the positive direction of the $d$ axes of $\R^d$. Since $W$ and $W^\prime$ are stationary, each $T_i$ is a measure-preserving transformation of $\xi$. Since $(W_v,W^\prime_v)$ are independent for different $v\in\Z^d$, the $\sigma$-algebra of invariant events associated with each $T_i$ is trivial (by Kolmogorov's 01-law). We now augment our previous notation to include the dependency on $\xi$; so we write $X_v(D_n,\xi)=X_v(D_n)$ and $X_v(\xi)=X_v$ etc. We then claim that for any $v\in\Z^d$
\begin{equation}\label{e:CLTProof6}
    X_v(\xi)=X_0(\tau_{-v}\xi)
\end{equation}
where $\tau_{-v}\xi=(\tau_{-v}W_u,\tau_{-v}W^\prime_u)_{u\in\Z^d}=(W_{u+v},W^\prime_{u+v})_{u\in\Z^d}$. To prove the claim, we first note that for any $n$
\begin{displaymath}
    \Delta_v(\Lambda_n,\xi)=\Delta_0(-v+\Lambda_n,\tau_{-v}\xi)
\end{displaymath}
using the fact that $\mu$ is stationary (recall the definition before Theorem~\ref{t:GenCLT}). Then taking $n\to\infty$ and using the stabilisation assumption, we see that the term on the left converges in probability to $\Delta_v(\xi)$ while the term on the right converges to $\Delta_0(\tau_{-v}\xi)$ and so these limits must coincide almost surely.

Next we observe that by definition of the lexicographic ordering
\begin{equation}\label{e:lex_filtration}
    \mathcal{F}_0(\tau_{-v}W):=\sigma(\tau_{-v}W_u\;|\;u\preceq 0)=\sigma(W_{u+v}\;|\;u\preceq 0)=\sigma(W_u\;|\;u\preceq v)=:\mathcal{F}_v(W).
\end{equation}
Therefore
\begin{displaymath}
    X_v(\xi)=\E[\Delta_v(\xi)\;|\;\mathcal{F}_v(W)]=\E[\Delta_0(\tau_{-v}\xi)\;|\;\mathcal{F}_0(\tau_{-v}W)]=X_0(\tau_{-v}\xi)
\end{displaymath}
completing the proof of \eqref{e:CLTProof6}. Finally we note that by Fatou's lemma and the bounded moments assumption
\begin{displaymath}
    \E\left[\lvert X_0(\xi)\rvert^{2+\epsilon}\right]\leq \liminf_{n\to\infty}\E\left[\lvert X_0(\Lambda_n,\xi)\rvert^{2+\epsilon}\right]<\infty
\end{displaymath}
so $X_0^2\in L^{1+\epsilon/2}(\nu)$ where $\nu$ denotes the distribution of $\xi$. Then since $\tau_v=T_1^{v_1}\dots T_d^{v_d}$ for $v\in\Z^d$, we may apply the ergodic theorem (Theorem~\ref{t:ergodic}) to conclude that
\begin{displaymath}
    (2n)^{-d}\sum_{v\in\Z^d\cap\Lambda_n^-}X_v^2=(2n)^{-d}\sum_{v\in\Z^d\cap\Lambda_n^-}X_0^2(\tau_{-v}\xi)\xrightarrow{L^1}\E[X_0^2]=\E\big[\E[\Delta_0\;|\;\mathcal{F}_0]^2\big]=:\sigma^2.
\end{displaymath}
Combining this with \eqref{e:CLTProof5} proves \eqref{e:CLTProof1} and hence completes the proof of the theorem.
\end{proof}

\begin{remark}
    When applying the martingale CLT, the most difficult condition to verify is often the convergence of the sum of squared increments (i.e.\ the analogue of \eqref{e:MCLT3}). Penrose \cite{pen01} recognised that for stationary white-noise functionals, this convergence could be elegantly dealt with  using the ergodic theorem and this insight is crucial for the proof we have just given. It is also interesting to note the importance of the lexicographic filtration in this context: \eqref{e:lex_filtration} was essential in our application of the ergodic theorem and this property holds only for the standard lexicographic ordering of $\Z^d$ (up to reflections and reorderings of axes).
\end{remark}

If the white-noise functional in the statement of Theorem~\ref{t:GenCLT} is additive, which is the case for our three functionals of interest, it is natural to expect that one could find somewhat simpler versions of conditions (2)-(4) in terms of how the functional changes on \emph{unit cubes} when resampling the white-noise (i.e.\ conditions on $\Delta_v(B_w)$ rather than $\Delta_v(\Lambda_n)$). The following result shows that this is indeed true and will simplify the proof of the CLTs for $\mu_\Vol$, $\mu_\SA$ and $\mu_\EC$.

\begin{proposition}\label{p:Sufficient_Moments}
    Let $\mu$ be a stationary white-noise functional such that with probability one the following holds: for any distinct $w_1,\dots,w_n\in\Z^d$
    \begin{equation}\label{e:Additivity}
        \mu\left(\bigcup_{i=1}^nB_{w_i}\right)=\sum_{i=1}^n\mu(B_{w_i}).
    \end{equation}
    Suppose that there exists a family of random variables $Y_v(w)\geq 0$ for $v,w\in\Z^d$ such that
    \begin{equation}\label{e:Moment_Suff1}
        \left\lvert \Delta_v\left(B_{w}\right)\right\rvert \leq Y_v(w)
    \end{equation}
    and
    \begin{equation}\label{e:Moment_Suff2}
        \E[Y_v(w)]+\E\big[Y_v(w)^{2+\epsilon}\big]\leq c(1+\lvert v-w\rvert)^{-(3d+\delta)}
    \end{equation}
    for some $\epsilon,\delta,c>0$ which depend only on the distribution of $f$. Then $\mu$ satisfies conditions (2)-(4) of Theorem~\ref{t:GenCLT}.

    If we replace \eqref{e:Additivity} and \eqref{e:Moment_Suff1} with the weaker assumption that for any $D:=u+\Lambda_n$ (where $u\in\Z^d$)
    \begin{equation}\label{e:Moment_Suff3}
        \left\lvert \Delta_v\left(D\right)\right\rvert\leq\sum_{w\in \Z^d\cap D}Y_v(w),
    \end{equation}
    where $Y_v(w)$ again satisfies \eqref{e:Moment_Suff2}, then $\mu$ satisfies conditions (3)-(4) of Theorem~\ref{t:GenCLT}.
\end{proposition}

\begin{proof}
Suppose that \eqref{e:Moment_Suff3} and \eqref{e:Moment_Suff2} hold. Taking $v\in\Z^d$ and $n\in\N$ as given, we write $Y(w)=Y_v(w)$. By \eqref{e:Moment_Suff3} and the inequality $\left(\sum_{i=1}^my_i\right)^p\leq\sum_{i=1}^my_i^p$, which holds for any $p\in(0,1)$ and $y_1,\dots,y_m\geq0$, we have
\begin{displaymath}
    \E\left[\lvert\Delta_v(\Lambda_n)\rvert^{2+\epsilon}\right]\leq\E\left[\left(\sum_{w,x,y\in\Z^d\cap\Lambda_n}Y(w)Y(x)Y(y)\right)^{\frac{2+\epsilon}{3}}\right]
    \leq \sum_{w,x,y\in\Z^d\cap\Lambda_n}\E\left[\big(Y(w)Y(x)Y(y)\big)^\frac{2+\epsilon}{3}\right].
\end{displaymath}
Applying H\"older's inequality to the right hand side along with \eqref{e:Moment_Suff2}
\begin{equation}\label{e:Moment_Suff4}
\begin{aligned}
    \E\left[\lvert\Delta_v(\Lambda_n)\rvert^{2+\epsilon}\right]\leq\left(\sum_{w\in\Lambda_n}\E\left[Y(w)^{2+\epsilon}\right]^\frac{1}{3}\right)^3\leq \left(\sum_{w\in\Lambda_n}c(1+\lvert v-w\rvert)^\frac{-3d-\delta}{3}\right)^3.
\end{aligned}
\end{equation}
Hence
\begin{displaymath}
    \sup_{v\in\Z^d,n\in\N}\E\left[\lvert\Delta_v(\Lambda_n)\rvert^{2+\epsilon}\right]\leq \left(\sum_{w\in\Z^d}c(1+\lvert w\rvert)^{-d-\frac{\delta}{3}}\right)^3<\infty
\end{displaymath}
verifying the bounded moments condition.

Now suppose that $v\notin\Lambda_n$. Then by \eqref{e:Moment_Suff4}
\begin{align*}
    \E\left[\lvert\Delta_v(\Lambda_n)\rvert^{2+\epsilon}\right]&\leq\left(cn^d\;\dist(v,\Lambda_n)^{-d-\delta/3}\right)^3
\end{align*}
which verifies the moment decay condition for the $(2+\epsilon)$-moments. By \eqref{e:Moment_Suff3} and \eqref{e:Moment_Suff2}
\begin{align*}
    \E\left[\lvert\Delta_v(\Lambda_n)\rvert\right]\leq\sum_{w\in\Lambda_n}\E[Y_v(w)]&\leq cn^d\;\dist(v,\Lambda_n)^{-3d-\delta}
\end{align*}
which verifies the other part of the moment decay condition.

Now assume that \eqref{e:Additivity} holds and fix a sequence $D_n:=v_n+\Lambda_n$ such that $\liminf_nD_n=\R^d$. By \eqref{e:Moment_Suff2} and the Borel-Cantelli lemma
\begin{displaymath}
    \sum_{w\in\Z^d}Y_0(w)<\infty
\end{displaymath}
almost surely. Given $m\in\N$, we may choose $N\in\N$ large enough that $\cap_{n\geq N}D_n\supset\Lambda_m$. Then for any $n_1,n_2>N$ by \eqref{e:Additivity} and \eqref{e:Moment_Suff1}
\begin{align*}
    \lvert\Delta_0(D_{n_1})-\Delta_0(D_{n_2})\rvert\leq\lvert\Delta_0(D_{n_1}\setminus D_{n_2})\rvert+\lvert\Delta_0(D_{n_2}\setminus D_{n_1})\rvert\leq \sum_{v\in\Z^d\setminus\Lambda_m}Y_0(w).
\end{align*}
Since the right-hand side converges to zero as $m\to\infty$, we see that $\Delta_0(D_n)$ is almost surely Cauchy and hence convergent to some limit $\Delta_0$.

To see that the limit does not depend on the choice of domains $D_n$, we may take any other suitable sequence $D_n^\prime$ and consider $(D_1,D_1^\prime,D_2,D_2^\prime,\dots)$. By our above argument we see that $\Delta_0(D_n^\prime)$ and $\Delta_0(D_n)$ converge almost surely to some $\Delta_0^*$ but the latter convergence implies that $\Delta_0^*=\Delta_0$ almost surely.
\end{proof}

\section{Topological stability}\label{s:Topological}
As described in Section~\ref{ss:Outline} there are three different potential contributions to the change in each of our functionals (volume, surface area and Euler characteristic) under perturbation: local/non-local topological contributions and geometric contributions. In this subsection, we control the probability of topological contributions using concepts from (stratified) Morse theory.

Dealing first with the local contributions; if we consider the level sets $\{f+tp_w=\ell\}\cap B_v$ as $t$ varies in $[0,1]$ (recall that $p_w$, defined in \eqref{e:PerturbDef}, is the perturbation induced by resampling the white noise on $B_w$), it is intuitively clear that the level sets should deform continuously, preserving their topology, unless they pass through a critical point. Thinking more carefully, one might realise that we also need to control how the topology changes near the boundary of $B_v$, which can be done by considering `boundary critical points'. This is essentially the logic of the first fundamental result of stratified Morse theory which we now make rigorous.

In this and the remaining sections we will state and prove a number of results involving perturbations of smooth functions, for which we adopt the following notational conventions: $p$ will denote an arbitrary perturbation in the statement of deterministic results, $p_v$ is the random perturbation defined by resampling the white noise in $B_v$ and $\rho$ is a deterministic function that will eventually be chosen equal to $sq\ast\ind_{\Lambda_m}$ for some values of $s$ and $m$. Perturbing by the latter expression will be needed in the proof of positivity of the limiting variance (see the end of Section~\ref{ss:Outline}).

Each unit cube $B_v$ for $v\in\Z^d$ can be viewed as a \emph{stratified set} by partitioning it into the finite union of each of its open faces of dimension $0,1,\dots,d$ which we refer to as \emph{strata}. For example, if $d=2$ then the stratification of $B_v$ consists of the (two-dimensional) interior of the square, four one-dimensional edges and four zero-dimensional corners. Let $A$ be a stratum of $B_v$ and $x\in A$. We say that $x$ is a \emph{stratified critical point} of a function $g\in C^1(B_v)$ if $\nabla_A g(x)=0$ where $\nabla_A$ denotes the gradient restricted to the affine space $A$. If $A$ is a zero-dimensional stratum then $\nabla_Ag\equiv 0$ by convention (i.e.\ all corners of the cube are critical points). We say that the \emph{level} of the critical point is $g(x)$.

Given $g,p\in C^1(B_v)$, we say that $(g,p)$ is \emph{locally topologically stable} (for domain $B_v$ at level $\ell$) if for all $t\in[0,1]$, $g+tp$ has no stratified critical points in $B_v$ at level $\ell$. Note that local topological stability is equivalent for $(g,p)$ and $(g+p,-p)$.

We will often work with a stratified set $D$ which is either a finite union of unit cubes or closed faces of unit cubes. In both cases we equip $D$ with the stratification consisting of all open faces of unit cubes contained in the set. We define a \emph{stratified isotopy} of $D$ to be a continuous map $H:D\times[0,1]\to D$ such that for each $t\in[0,1]$
\begin{enumerate}
    \item $H(\cdot,t):D\to D$ is a homeomorphism, and
    \item for any stratum $A$ of $D$, $H(A,t)=A$.
\end{enumerate}
With these definitions we may state our first stability result:

\begin{lemma}\label{l: morse continuity}
For $v_1,\dots,v_n\in\Z^d$, let $D:=\cup_{i=1}^nB_{v_i}$ and let $g,p\in C^2(D)$. If $(g,p)$ is locally topologically stable on each $B_{v_i}$ at level $\ell$ then there exists a stratified isotopy $H$ of $D$ such that
\begin{enumerate}
    \item $H(\{g\geq\ell\},t)=\{g+tp\geq\ell\}$ for all $t\in[0,1]$,
    \item $H(\cdot,0):D\to D$ is the identity map.
\end{enumerate}
\end{lemma}
\begin{proof}
    The existence of a stratified isotopy satisfying the first property was proven for boxes in \cite[Lemma~12]{bmm24b} using Thom's isotopy lemma. The proof remains valid for $D$ since the latter is a Whitney stratified space. The second property is trivial: if $\widetilde{H}$ satisfies the first property then defining $h=\widetilde{H}(\cdot,0)$ we can set $H(x,t)=\widetilde{H}(h^{-1}(x),t)$.
\end{proof}

This result allows us to deduce some topological control over excursion sets from pointwise conditions on $g$, $p$ and their derivatives. We next give a probabilistic statement for stability of $f$ under the perturbations $p_v$ (defined in \eqref{e:PerturbDef}) and $p_v+\rho$. (We recall that our results will eventually be applied with $\rho=sq\ast\ind_{\Lambda_m}$ for some values of $s$ and $m$; see Section~\ref{ss:Outline}.) Given $C^1$ functions $g,p:\R^d\to\R$ we define the \textit{locally unstable set} as
\begin{displaymath}
    \mathcal{U}_\mathrm{Loc}(g,p)=\big\{w\in\Z^d\;\big|\;(g,p)\text{ is locally unstable for }B_w\text{ at level }\ell\big\}.
\end{displaymath}
See Figure~\ref{fig:unstable_set}. Note that this set depends on the level $\ell$ but we omit this from our notation as it will be fixed throughout our analysis.

\begin{figure}[ht]
    \centering
    \begin{tikzpicture}

\foreach \i in {0,2,4,6,8}
{\draw[dashed] (\i,0)--(\i,4);}
\foreach \i in {0,2,4}
{\draw[dashed] (0,\i)--(8,\i);}

\begin{scope}[scale=1.4, shift = {(-0.2,-0.7)}]
    \draw[thick,pattern=north west lines,pattern color=gray] plot [smooth cycle,tension=0.5] coordinates {(2,2.8)(1.9,3)(1.7,3.2)(1.5,3.2)(1.3,2.9)(1.1,2.9)(0.9,3.2)(0.7,3.2)(0.5,3)(0.4,2.8)(0.5,2.6)(0.7,2.4)(0.9,2.4)(1.1,2.7)(1.3,2.7)(1.5,2.4)(1.7,2.4)(1.9,2.6)};
\end{scope}

\draw[thick,pattern=north west lines,pattern color=gray] (1.7,1.5)circle (0.1);

\begin{scope}[shift = {(4.56,0.5)},scale = 1.1]
\draw[thick,pattern=north west lines,pattern color=gray] plot [smooth cycle,tension=0.5] coordinates {(0,0)(0.2,0.5)(0.7,1)(1.2,1.1)(1.4,0.9)(1.3,0.7)(1.0,0.4)(0.5,0)};
\end{scope}
\node[above] at (5,2.8) {$\{g\geq\ell\}$};
\draw (5,2.8)--(5,1.1);
\draw (5,2.8)--(2.2,2.8);
\draw (5,2.8)--(1.7,1.5);

\draw [fill] (0,0) circle (2pt);
\node[below left] at (0,0) {$v_2$};
\draw [fill] (0,2) circle (2pt);
\node[below left] at (0,2) {$v_1$};
\draw [fill] (4,0) circle (2pt);
\node[below left] at (4,0) {$v_3$};
\draw [fill] (6,0) circle (2pt);
\node[below left] at (6,0) {$v_4$};

\end{tikzpicture}

    \caption{For the displayed excursion set, a small perturbation could create critical points in $B_{v_1}$ and $B_{v_2}$ (a saddle point and local maximum respectively) and a stratified critical point in $B_{v_3}\cap B_{v_4}$ (as the excursion component `shrinks' out of $B_{v_4})$. For a function $p$ which is not too small (and not too large), $\mathcal{U}_{\mathrm{Loc}}(g,p)$ would consist of $\{v_1,\dots,v_4\}$.}
    \label{fig:unstable_set}
\end{figure}

For an open set $U\subset\R^d$, a compact set $F\subset U$, $k\in\N\cup\{0\}$ and $g\in C^k(U)$ we recall that
\begin{displaymath}
    \|g\|_{C^k(F)}:=\sup_{x\in F}\sup_{\lvert\alpha\rvert\leq k}\lvert\partial^\alpha g(x)\rvert.
\end{displaymath}

\begin{lemma}\label{l:Prob_unstable_decay}
Let $f$ satisfy Assumption~\ref{a:clt} and let $\rho:\R^d\to\R$ be a $C^1$ deterministic function. Then for each $\delta>0$ there exists $c>0$ (which is independent of $\rho$ but may depend on the distribution of $f$), such that
\begin{displaymath}
    \P(w\in\mathcal{U}_\mathrm{Loc}(f,\rho))\leq c\|\rho\|^{1-\delta}_{C^1(w+\Lambda_2)}\qquad\text{for all }w\in\Z^d
\end{displaymath}
and
\begin{displaymath}
\P\left(w\in \mathcal{U}_\mathrm{Loc}(f,p_v+\rho)\right)\leq c\|\rho\|_{C^1(w+\Lambda_2)}^{1-\delta}+ c( 1+ \lvert v-w\rvert)^{-\beta+\delta}\qquad\text{for all }v,w\in\Z^d
\end{displaymath}
where $\beta$ is the decay exponent given in Assumption~\ref{a:clt}.
\end{lemma}
\begin{proof}
    The first statement follows from \cite[Lemma~2.4]{bmm24}. The same lemma also states that for any $\epsilon>0$, the probability that $w\in\mathcal{U}_\mathrm{Loc}(f,p_v+\rho)$ is at most
    \begin{equation}\label{e:LocStabBound}
        c_\epsilon\inf_{\tau>M_1+c\sqrt{M_2}}\left(\tau^{1-\epsilon}+e^{-\frac{(\tau-M_1-c\sqrt{M_2})^2}{2M_2}}\right)
    \end{equation}
    where
    \begin{displaymath}
        M_1:=\|\rho\|_{C^1(w+\Lambda_2)},\qquad M_2:=\sup_{x,y\in w+\Lambda_2}\sup_{\lvert\alpha\rvert,\lvert\gamma\rvert\leq 2}\left\lvert\partial^\alpha_x\partial^\gamma_y\Cov[p_v(x),p_v(y)]\right\rvert
    \end{displaymath}
    and $c>0$ is an absolute constant. (The proof of the cited lemma uses standard estimates for the supremum of a Gaussian field, given in \cite[Section~3]{bmm24}, and a bound on the probability of $f$ and $\nabla f$ being simultaneously small, from \cite{ns16}.) From the white noise representation for $p_v$ (in \eqref{e:PerturbDef})
    \begin{displaymath}
        \lvert\partial_x^\alpha\partial_y^\gamma\Cov[p_v(x),p_v(y)]\rvert=2\left\lvert\int_{B_v}\partial_x^\alpha q(x-u)\partial_y^\gamma q(y-u)\;du\right\rvert
    \end{displaymath}
    (where taking the derivative inside the integral follows from dominated convergence since $q\in C^3(\R^d)$ and $B_v$ is compact). Therefore by Assumption~\ref{a:clt}, $M_2\leq C(1+\lvert v-w\rvert)^{-2\beta}$ for some $C>0$ independent of $v$ and $w$. Using this observation and setting $\tau=2(M_1+c\sqrt{M_2})^{1-\epsilon}$ in \eqref{e:LocStabBound} we have
    \begin{displaymath}
        \P\left(w\in \mathcal{U}_\mathrm{Loc}(f,p_v+\rho)\right)\leq c_1\Big(\|\rho\|_{C^1(w+\Lambda_2)}^{1-2\epsilon}+(1+ \lvert v-w\rvert)^{-\beta(1-2\epsilon)}+e^{-c_2(1+\lvert v-w\rvert)^{2\epsilon\beta}}\Big)
    \end{displaymath}
    for some $c_1,c_2>0$ which depend on $\epsilon$. Choosing $\epsilon>0$ sufficiently small (depending on $\beta$ and $\delta$) then proves the bound in the statement of the lemma.
\end{proof}

Next we turn to non-local topological contributions. Recall that such contributions occur whenever the topology of the level sets within a unit cube $B_v$ does not change but topological changes in another cube connect/disconnect some components in $B_v$ to/from infinity. Our next result shows that, for such a change to occur, $B_v$ must be connected to a locally unstable cube via a bounded component.

Recall that for three sets $A,B,C\subset\R^d$ we write $A\overset{B}{\longleftrightarrow}C$ if there exists a continuous curve in $B$ joining a point in $A$ to a point in $C$. For a $C^2$ function $g:\R^d\to\R$, let $\{g\geq\ell\}_\infty$ denote the union of any unbounded components of $\{g\geq\ell\}$ and let $\{g\geq\ell\}_{<\infty}:=\{g\geq\ell\}\setminus\{g\geq\ell\}_\infty$ be the union of the bounded components. Then for a pair of $C^2$ functions $(g,p)$ we define the \emph{finite-range unstable set} as
\begin{equation}\label{e:finite_range_unstable}
    \mathcal{U}_{\mathrm{FR}}(g,p)=\mathcal{U}_\mathrm{Loc}(g,p)\cup\bigcup_{v\in\mathcal{U}_\mathrm{Loc}(g,p)}\Big\{w\in\Z^d\;\Big|\;B_w\overset{\{g\geq\ell\}_{<\infty}}{\longleftrightarrow}B_v\text{ or }B_w\overset{\{g+p\geq\ell\}_{<\infty}}{\longleftrightarrow}B_v\Big\}.
\end{equation}
Note that this set depends on the level $\ell$. We say that $B_w$ is finite-range stable if $w\notin\mathcal{U}_\mathrm{FR}(g,p)$.

\begin{lemma}\label{l:TopStab}
    Let $D:=\cup_{i=1}^nB_{v_i}$ be equipped with the stratification consisting of all open faces of $B_{v_1},\dots,B_{v_n}$ and let $g,p:\R^d\to\R$ be $C^2$. Assume that $\{g\geq\ell\}\cap D$ and $\{g+p\geq\ell\}\cap D$ have at most a finite number of components. If $B_{v_i}$ is finite-range stable for each $i$, then there exists a stratified isotopy $H$ of $D$ such that
    \begin{enumerate}
        \item $H(\{g\geq\ell\}\cap D,t)=\{g+tp\geq\ell\}\cap D$ for all $t\in[0,1]$,
        \item $H(\cdot,0):D\to D$ is the identity map,
        \item $H(\{g\geq\ell\}_{\infty}\cap D,1)=\{g+p\geq\ell\}_\infty\cap D$.
    \end{enumerate}
\end{lemma}
\begin{proof}
    Let $D^+$ be the union of all unit cubes which are connected to $D$ by bounded excursion components of $g$ or $g+p$, that is
    \begin{displaymath}
        D^+:=D\cup\bigcup_{i=1}^n\bigcup\Big\{B_v\;\Big|\;B_v\overset{\{g\geq\ell\}_{<\infty}}{\longleftrightarrow}B_{v_i}\;\text{or}\;B_v\overset{\{g+p\geq\ell\}_{<\infty}}{\longleftrightarrow}B_{v_i}\Big\}.
    \end{displaymath}
    Since $\{g\geq\ell\}\cap D$ and $\{g+p\geq\ell\}\cap D$ have finitely many components, we see that $D^+$ is the union of finitely many unit cubes. Since each cube in $D$ is finite-range stable, each cube in $D^+$ is locally topologically stable, so we may apply Lemma~\ref{l: morse continuity} to obtain a stratified isotopy $H:D^+\times[0,1]\to D^+$ such that $H(\{g\geq\ell\}\cap D^+,t)=\{g+tp\geq\ell\}\cap D^+$.
    
    Let $A$ be a component of $\{g\geq\ell\}_{<\infty}$ which intersects $D$. Observe that $A$ cannot intersect the boundary of $D^+$ (if it did, then some cube not contained in $D^+$ would be connected to $D$ by a bounded excursion component). Hence, since $H$ preserves strata, $H(A,1)$ cannot intersect the boundary of $D^+$. In particular $H(A,1)$ is a bounded component of $\{g+p\geq\ell\}$. This is true for any such component $A$, and so
    \begin{displaymath}
        H(\{g\geq\ell\}_{<\infty}\cap D,1)\subseteq\{g+p\geq\ell\}_{<\infty}.
    \end{displaymath}
    By the same reasoning, any component of $\{g+p\geq\ell\}_{<\infty}$ which intersects $D$ must be the image under $H(\cdot,1)$ of a component of $\{g\geq\ell\}_{<\infty}$. Therefore, recalling that $H(D,1)=D$, we have
    \begin{displaymath}
        H(\{g\geq\ell\}_{<\infty}\cap D,1)=\{g+p\geq\ell\}_{<\infty}\cap D.
    \end{displaymath}
    The restriction of $H$ to $D$ then satisfies the conditions in the statement of the lemma.
\end{proof}

Our next objective is to obtain probabilistic bounds for finite-range stability. As an input to such a bound, we will require the following generalisation of Proposition~\ref{p:BCDecay} which allows the field $f$ to be perturbed by a deterministic, non-negative function.

\begin{proposition}\label{p:UnifTruncConnecDecay}
    Let $f$ satisfy Assumption~\ref{a:clt} and let $\rho:\R^d\to\R$ be non-negative, continuously differentiable and deterministic. For each $\ell<-\ell_c$ there exists $C,c>0$ (independent of $\rho$) such that for all $n\in\N$ and $v\in\Z^d$
    \begin{displaymath}
        \P\left(B_v\overset{\{f+\rho\geq\ell\}_{<\infty}}{\longleftrightarrow}v+\partial\Lambda_n\right)\leq Ce^{-cn}.
    \end{displaymath}
\end{proposition}
\begin{proof}[Proof of Proposition~\ref{p:BCDecay}]
    Set $\rho\equiv 0$ and apply Proposition~\ref{p:UnifTruncConnecDecay}.
\end{proof}

\begin{proof}[Proof of Proposition~\ref{p:UnifTruncConnecDecay}]
Since the distribution of $f$ is translation invariant (as are our assumptions on the function $\rho$) it is enough to prove the bound only for $v=0$. Suppose that $B_0$ is connected to $\partial\Lambda_n$ by some bounded component of $\{f+\rho\geq\ell\}$. Let $A$ denote this component and let $N$ be the largest integer such that $A$ intersects $\partial\Lambda_{N+1}$ (see Figure~\ref{fig:arm_decay}). So the diameter of $A$ is at least $N$ and $N\geq n-1$. By Bulinskaya's lemma \cite[Lemma~11.2.10]{at07} $f+\rho$ almost surely has no critical points at level $\ell$. Hence the boundary of $A$ is $C^1$-smooth and so there is a component $E$ of $\{f+\rho\leq\ell\}$ which `surrounds' $A$. More precisely, $E$ is the component of $\{f+\rho\leq\ell\}$ which intersects the outer boundary of $A$. Then by definition of $N$, there exists some $x\in\Z^d\cap\Lambda_{N+2}\setminus\Lambda_{N}$ such that $E$ intersects $B_x$. Moreover since the diameter of $E$ is at least as large as the diameter of $A$ we see that $E$ connects $B_x$ to $x+\partial\Lambda_{N/\sqrt{d}-1}$. By non-negativity of $\rho$ and symmetry of the normal distribution
\begin{displaymath}
    \P\left(B_x\overset{\{f+\rho\leq\ell\}}{\longleftrightarrow}x+\partial\Lambda_{\frac{N}{\sqrt{d}}-1}\right)\leq\P\left(B_x\overset{\{f\leq\ell\}}{\longleftrightarrow}x+\partial\Lambda_{\frac{N}{\sqrt{d}}-1}\right)=\P\left(B_x\overset{\{f\geq-\ell\}}{\longleftrightarrow}x+\partial\Lambda_{\frac{N}{\sqrt{d}}-1}\right).
\end{displaymath}
Since $-\ell>\ell_c$, using sharpness of the phase transition (i.e.\ Theorem~\ref{t:PhaseTransition}) the latter probability is bounded above by $Ce^{-cN}$ where $C,c>0$ depend only on $\ell$ and the distribution of $f$. Therefore by the union bound
\begin{equation*}
    \P\Big(\Lambda_1\overset{\{f\geq\ell\}_{<\infty}}{\longleftrightarrow}\partial\Lambda_n\Big)\leq\sum_{N\geq n-1}\sum_{x\in\Lambda_{N+2}\setminus\Lambda_{N}}Ce^{-cN}\leq \sum_{N\geq n-1} C^\prime N^{d-1}e^{-cN}\leq C^{\prime\prime}e^{-cn}
\end{equation*}
for constants $C,C^\prime,C^{\prime\prime},c>0$, as required.
\end{proof}

\begin{figure}[ht]
    \centering
    \begin{tikzpicture}[scale=0.6]

\draw[dashed] (-5,-5) rectangle (5,5);
\draw[dashed] (-3,-3) rectangle (3,3);
\node[left] at (-5,-5) {$\partial\Lambda_{N+1}$};
\node[left] at (-3,-3) {$\partial\Lambda_{n}$};
\draw [fill] (0,0) circle (2pt);
\node[below left] at (0,0) {$0$};
\draw[thick,pattern=north west lines,pattern color=gray] plot [smooth cycle,tension=0.5] coordinates {(-1,-1)(-1,0)(-0.5,0.5)(1,0.5)(2.5,0.5)(3.5,0.8)(4.5,1.2)(5.2,1.1)(5.2,0.9)(4.5,0.5)(3.5,0)};
\node[right] at (6,0.2) {$A$};
\draw (6,0.2)--(4.3,0.7);
\node[above] at (-3,3.5) {$\{f+\rho\geq\ell\}$};
\draw (-3,3.5)--(-2.5,1.5);
\draw (-3,3.5)--(-0.5,0);

\draw (5,1.6) rectangle (5.4,2);
\node [right] at (6,2.5) {$B_x$};
\draw(6,2.5)--(5.2,1.8);

\clip (-6,-6) rectangle (6,6);
\draw[thick,pattern=north west lines,pattern color=gray] plot [smooth cycle,tension=0.5] coordinates {(-2.5,0)(-1.5,1)(0,2)(2,2.5)(4,2.5)(6.2,3.5)(6.5,4.5)(5.5,5)(3.5,4.5)(2.5,3.5)(0.5,3.5)(-1.5,2.5)(-3.5,1.5)(-5.5,1)(-6.5,0.5)(-6.5,-0.5)(-4.5,0)(-3,-1.5)(-1.5,-3)(1,-3.5)(3,-3.5)(5,-3)(6.5,-2.5)(6.5,-1.5)(4.5,-2)(2.5,-2.5)(0.5,-2.5)(-1.5,-1.5)};

\end{tikzpicture}

    \caption{A possible excursion set in which the origin is connected to $\partial\Lambda_n$ by a bounded component. The set $E$, defined in the proof of Proposition~\ref{p:UnifTruncConnecDecay}, is the unshaded region surrounding $A$.}
    \label{fig:arm_decay}
\end{figure}

We now use this result to get a probabilistic bound for finite-range stability.

\begin{lemma}\label{l:ProbExtUnst}
Let $f$ satisfy Assumptions~\ref{a:clt} and \ref{a:Decay}, then for all $\delta>0$ there exists $c>0$ such that
\begin{displaymath}
    \P(w\in\mathcal{U}_\mathrm{FR}(f,p_v))\leq c(1+\lvert v-w\rvert)^{-\beta+\delta}\qquad\text{for }v,w\in\Z^d.
\end{displaymath}
In particular, by the Borel-Cantelli lemma, $\mathcal{U}_\mathrm{FR}(f,p_0)$ is almost surely finite. 

Let $\rho:\R^d\to\R$ be a deterministic function such that
\begin{displaymath}
    \max_{\lvert\alpha\rvert\leq 2}\lvert\partial^\alpha\rho(x)\rvert\leq C(1+\lvert x\rvert)^{-\beta}
\end{displaymath}
for some $C>0$ and all $x\in\R^d$. Suppose in addition that $\ell<-\ell_c$, then for all $\delta>0$ there exists $c>0$ such that
\begin{displaymath}
    \P(w\in\mathcal{U}_\mathrm{FR}(f,p_0+\rho))\leq c(1+\lvert w\rvert)^{-\beta+\delta}\qquad\text{for }w\in\Z^d.
\end{displaymath}
\end{lemma}
\begin{proof}
    We first argue for $p_0+\rho$. By definition of the finite-range unstable set and the union bound
    \begin{equation}\label{e:FRBound}
    \begin{aligned}
        \P(w\in\mathcal{U}_\mathrm{FR}(f,p_0+\rho))\leq\; &\P(w\in\mathcal{U}_\mathrm{Loc}(f,p_0+\rho))\\
        &+\sum_{u\in\Z^d}\P\Big(\{u\in\mathcal{U}_\mathrm{Loc}(f,p_0+\rho)\}\cap\big\{B_w\overset{\{f\geq\ell\}_{<\infty}}{\longleftrightarrow}B_u\big\}\Big)\\
        &\quad+\sum_{u\in\Z^d}\P\Big(\{u\in\mathcal{U}_\mathrm{Loc}(f,p_0+\rho)\}\cap\big\{B_w\overset{\{\widetilde{f}_0+\rho\geq\ell\}_{<\infty}}{\longleftrightarrow}B_u\big\}\Big).
    \end{aligned}
    \end{equation}
By Lemma~\ref{l:Prob_unstable_decay} the first term on the right hand side is bounded by $c(1+\lvert w\rvert)^{-\beta+\delta}$. For any $\epsilon\in(0,\delta/2d)$, the second term can be bounded by
\begin{equation}\label{e:FRunstable}
    \sum_{\lvert u-w\rvert\leq\lvert w\rvert^\epsilon}\P(u\in\mathcal{U}_\mathrm{Loc}(f,p_0+\rho))+\sum_{\lvert u-w\rvert>\lvert w\rvert^\epsilon}\P\Big(B_w\overset{\{f\geq\ell\}_{<\infty}}{\longleftrightarrow}B_u\Big).
\end{equation}
By Lemma~\ref{l:Prob_unstable_decay} the first sum in \eqref{e:FRunstable} is at most
\begin{align*}
    \sum_{\lvert u-w\rvert\leq\lvert w\rvert^\epsilon}c(1+\lvert u\rvert)^{-\beta+\delta/2}\leq c^\prime\lvert w\rvert^{d\epsilon}(1+\lvert w\rvert)^{-\beta+\delta/2}\leq c^\prime(1+\lvert w\rvert)^{-\beta+\delta}
\end{align*}
for constants $c,c^\prime>0$. By Assumption~\ref{a:Decay} and Lemma~\ref{l:ElementarySum}, the second sum in \eqref{e:FRunstable} decays super-polynomially in $\lvert w\rvert$. If $\rho\equiv 0$, then the third term of \eqref{e:FRBound} is identical to the second term (since $\widetilde{f}_0$ has the same distribution as $f$). Hence we conclude that
\begin{displaymath}
    \P(w\in\mathcal{U}_\mathrm{FR}(f,p_0))\leq c(1+\lvert w\rvert)^{-\beta+\delta},
\end{displaymath}
and the first statement of the lemma follows by stationarity.

If $\rho$ is not identically zero, then we can bound the third term of \eqref{e:FRBound} using the same argument as for the second term, except that we use Proposition~\ref{p:UnifTruncConnecDecay} (and the assumption that $\ell<-\ell_c$) to control
\begin{displaymath}
    \sum_{\lvert u-w\rvert>\lvert w\rvert^\epsilon}\P\Big(B_w\overset{\{\widetilde{f}_0+\rho\geq\ell\}_{<\infty}}{\longleftrightarrow}B_u\Big)
\end{displaymath}
rather than Assumption~\ref{a:Decay}.
\end{proof}

\section{Volume, surface area and Euler characteristic of the unbounded component}\label{s:Pert}
To summarise our progress so far: by Theorem~\ref{t:GenCLT} and Proposition~\ref{p:Sufficient_Moments} our functionals $\mu_\star$ will each satisfy a CLT if we can obtain appropriate moment bounds for the change in the value of the functional under perturbation. In the previous section we obtained bounds on the probability of topological contributions to this change. It remains to control the magnitude of such contributions as well as the geometric contributions. In the next three subsections, which can be read independently, we do this using arguments specific to each of the three functionals.

\subsection{Volume}\label{ss:Volume}
Controlling the magnitude of topological contributions to $\Delta_v$ is trivial for the volume functional since the volume of any unit cube is bounded deterministically. It is also straightforward to bound the geometric contributions to $\Delta_v$ in this case, as the following lemma shows.

For a $C^2$ function $g:\R^d\to\R$ let $\Vol_\infty[D,g]=\Vol[\{g\geq\ell\}_\infty\cap D]$.

\begin{lemma}\label{l:ChangeDecomp}
Let $g,p:\R^d\to\R$ be $C^2$.
\begin{enumerate}
\item For any distinct $v_1,\dots,v_n\in\Z^d$
\begin{displaymath}
    \Vol_\infty\left[\bigcup_{i=1}^nB_{v_i},g\right]=\sum_{i=1}^n\Vol_\infty[B_{v_i},g].
\end{displaymath}
    \item For any $v\in\Z^d$ 
\begin{displaymath}
    \lvert\Vol_\infty[B_v,g]-\Vol_\infty[B_v,g+p]\rvert\leq 1.
\end{displaymath}
    \item For $v\notin\mathcal{U}_\mathrm{FR}(g,p)$, suppose that $g$ and $g+p$ have at most a finite number of stratified critical points on $B_v$, then
\begin{displaymath}
    \lvert\Vol_\infty[B_v,g]-\Vol_\infty[B_v,g+p]\rvert\leq\Vol\big[B_v\cap\{\lvert g-\ell\rvert\leq\lvert p\rvert\}\big].
\end{displaymath}
\end{enumerate}
\end{lemma}

\begin{proof}
The first statement follows immediately from the fact that $\Vol_\infty[\cdot,g]$ is a measure such that $\Vol_\infty[\partial B_v,g]=0$ for any $v\in\Z^d$. The second statement is trivial since $0\leq\Vol_\infty[B_v,g],\Vol_\infty[B_v,g+p]\leq 1$.

Now assume $v\notin\mathcal{U}_\mathrm{FR}(g,p)$. If $g$ and $g+p$ have a finite number of stratified critical points, then $\{g\geq\ell\}\cap B_v$ and $\{g+p\geq\ell\}\cap B_v$ have at most finitely many components (since each component has a maximum which is a stratified critical point). Hence we can apply Lemma~\ref{l:TopStab} and let $H$ denote the corresponding stratified isotopy on $B_v$. Let $S=\{\lvert g-\ell\rvert>\lvert p\rvert\}\cap B_v$ and $S^c=B_v\setminus S$. We claim that if $A$ is a component of $\{g\geq\ell\}\cap B_v$ then
\begin{displaymath}
    A\cap S=H(A,1)\cap S.
\end{displaymath}
Applying the claim to any component $A$ of $\{g\geq\ell\}_\infty\cap B_v$ and using the third point of Lemma~\ref{l:TopStab}, we see that
\begin{displaymath}
    \Vol[A]-\Vol[H(A,1)]=\Vol[A\cap S^c]-\Vol[H(A,1)\cap S^c].
\end{displaymath}
Summing over the components $A$ we have
\begin{displaymath}
    \lvert\Vol_\infty[B_v,g]-\Vol_\infty[B_v,g+p]\rvert=\lvert\Vol[\{g\geq\ell\}_\infty\cap S^c]-\Vol[\{g+p\geq\ell\}_\infty\cap S^c]\rvert\leq \Vol\big[S^c\big]
\end{displaymath}
which verifies the third statement of the lemma.

It remains to prove the claim. We define
\begin{displaymath}
    \epsilon=\min_{A_1\neq A_2}\min_{t\in[0,1]}\dist(H(A_1,t),H(A_2,t))
\end{displaymath}
where $A_1$ and $A_2$ are components of $\{g\geq\ell\}\cap B_v$. In other words $\epsilon$ is the minimum distance between the images under $H$ of any two distinct excursion components. Since $H$ is continuous and $H(\cdot,t)$ is a homeomorphism for each $t$, we see that $\epsilon>0$.

Let $A$ be a component of $\{g\geq\ell\}\cap B_v$ and let $x$ be a point in $A\cap S$ or $H(A,1)\cap S$. In either case, $x\in\{g+tp\geq\ell\}=H(\{g\geq\ell\},t)$ for all $t\in[0,1]$. Suppose that $x\in H(A^\prime,t^\prime)$ for some excursion component $A^\prime\neq A$ and $t^\prime\in(0,1]$. Then $\dist(x,H(A,t^\prime))\geq\epsilon$ and so by the intermediate value theorem there exists $t^{\prime\prime}\in(0,t^\prime)$ such that $\dist(x,H(A,t^{\prime\prime}))=\epsilon/2>0$. But then $x$ is not contained in $H(\{g\geq\ell\},t^{\prime\prime})$ which yields a contradiction. We conclude that $x\in H(A,t)$ for all $t\in[0,1]$, proving the claim.
\end{proof}

Armed with this lemma, we can easily obtain moment bounds on the geometric contributions to $\Delta_v$ using Gaussian tail inequalities. We remind the reader that $\rho$ is an additional deterministic perturbation that will be required to prove positivity of the limiting variance in Section~\ref{s:positivity}.

\begin{lemma}\label{l:VolMoments}
Let $f$ satisfy Assumption~\ref{a:clt} and let $\rho:\R^d\to\R$ be a deterministic $C^2$ function. Then there exists an absolute constant $c>0$ such that for $w\in\Z^d$
\begin{displaymath}
    \E\Big[\Vol[B_w\cap\{\lvert f-\ell\rvert\leq\lvert\rho\rvert\}]\Big]\leq c\|\rho\|_{C(B_w)}.
\end{displaymath}
If $\rho$ also satisfies the conditions in Lemma~\ref{l:ProbExtUnst}, then for each $\delta>0$ there exists $c>0$ such that for $w\in\Z^d$
\begin{displaymath}
    \E\Big[\Vol[B_w\cap\{\lvert f-\ell\rvert\leq\lvert p_0+\rho\rvert\}]\Big]\leq c(1+\lvert w\rvert)^{-\beta+\delta}.
\end{displaymath}
In particular, by stationarity, for any $v\in\Z^d$,
\begin{displaymath}
\E\Big[\Vol[B_w\cap\{\lvert f-\ell\rvert\leq\lvert p_v\rvert\}]\Big]\leq c(1+\lvert v-w\rvert)^{-\beta+\delta}.
\end{displaymath}
\end{lemma}

\begin{proof}
By Fubini's theorem
\begin{align*}
\E\Big[\Vol[B_w\cap\{\lvert f-\ell\rvert\leq\lvert \rho\rvert\}]\Big]=\int_{B_w}\P(\lvert f(x)-\ell\rvert\leq\lvert \rho(x)\rvert)\;dx&\leq \int_{B_w}\sqrt{\frac{2}{\pi}}\lvert \rho(x)\rvert\;dx
\end{align*}
using the (standard Gaussian) density of $f(x)$. The right hand side is at most $\sqrt{2/\pi}\|p\|_{C(B_w)}$, which proves the first statement of the lemma.

Applying Fubini's theorem once more yields
\begin{equation}\label{e:Volbound1}
\E\Big[\Vol[B_w\cap\{\lvert f-\ell\rvert\leq\lvert p_0+\rho\rvert\}]\Big]=\int_{B_w}\P(\lvert f(x)-\ell\rvert\leq\lvert p_0(x)+\rho(x)\rvert)\;dx.
\end{equation}
By the union bound, for any $\eta>0$
\begin{equation}\label{e:Volbound2}
\begin{aligned}
\P(\lvert f(x)-\ell\rvert\leq\lvert p_0(x)+\rho(x)\rvert)&\leq \P(\lvert f(x)-\ell\rvert<\eta)+\P(\lvert p_0(x)+\rho(x)\rvert>\eta)\\
&\leq \sqrt{2/\pi}\eta+\P(\lvert p_0(x)\rvert>\eta-\lvert \rho(x)\rvert).
\end{aligned}
\end{equation}
From the definition of $p_0$ (in \eqref{e:PerturbDef}) and the decay of $q$ (i.e.\ Assumption~\ref{a:clt}), $p_0(x)$ is normally distributed with mean zero and variance
\begin{displaymath}
    2\int_{B_0}q^2(x-u)\;du\leq c(1+\lvert x\rvert)^{-2\beta}
\end{displaymath}
for some $c>0$. Setting $\eta=c(1+\lvert w\rvert)^{-\beta+\delta}$ (and recalling that $\lvert\rho(x)\rvert<c(1+\lvert x\rvert)^{-\beta}$) we obtain for $x\in B_w$\begin{equation}\label{e:Volbound3}
\P(\lvert p_0(x)\rvert>\eta-\lvert \rho(x)\rvert)\leq c_1\exp\left(-c_2(1+\lvert w\rvert)^{2\delta}\right)
\end{equation}
for constants $c_1,c_2>0$. Combining \eqref{e:Volbound1}-\eqref{e:Volbound3} proves the second statement of the lemma.
\end{proof}

\begin{proof}[Proof of Theorem~\ref{t:clt} (for $\star=\Vol$)]
    Taking $\mu:=\mu_\Vol$ it is enough to verify the four conditions of Theorem~\ref{t:GenCLT}. It is trivial that $\mu_\Vol$ has finite second moments since $\mu_\Vol(v+\Lambda_n)\leq (2n)^d$ and the remaining conditions will follow if we can justify an application of Proposition~\ref{p:Sufficient_Moments}.

    Condition \eqref{e:Additivity} follows immediately from Lemma~\ref{l:ChangeDecomp}. Given $v,w\in\Z^d$ we define
    \begin{displaymath}
        Y_v(w)=\Vol\big[B_w\cap\{\lvert f-\ell\rvert\leq\lvert p_v\rvert\}\big]+\ind_{w\in\mathcal{U}_\mathrm{FR}(f,p_v)}
    \end{displaymath}
    then by Lemma~\ref{l:ChangeDecomp} we have $\lvert\Delta_v(B_w)\rvert\leq Y_v(w)$, verifying \eqref{e:Moment_Suff1}. By Lemma~\ref{l:VolMoments}
    \begin{displaymath}
        \E\left[\Vol\big[B_w\cap\{\lvert f-\ell\rvert\leq\lvert p_v\rvert\}\big]^{2+\epsilon}\right]\leq\E\left[\Vol\big[B_w\cap\{\lvert f-\ell\rvert\leq\lvert p_v\rvert\}\big]\right]\leq c(1+\lvert v-w\rvert)^{-\beta+\delta}
    \end{displaymath}
    for any $\delta>0$ and some $c$ depending on $\delta$. Combined with Lemma~\ref{l:ProbExtUnst} we see that
    \begin{displaymath}
        \E[Y_v(w)+Y_v(w)^{2+\epsilon}]\leq c(1+\lvert v-w\rvert)^{-\beta+\delta}.
    \end{displaymath}
    Since $\beta>3d$, by taking $\delta>0$ sufficiently small we verify \eqref{e:Moment_Suff2} and so we may apply Proposition~\ref{p:Sufficient_Moments} which completes the proof.
\end{proof}

\subsection{Surface area}\label{ss:SurfaceArea}
We now move on to consider the surface area functional. Our first order of business is to control the magnitude of topological and geometric contributions to $\Delta_v(B_w)$. As in the previous subsection, we begin with some deterministic arguments.

Let $g,p:\R^d\to\R$ be $C^2$ such that $g$ and $g+p$ have no stratified critical points at level $\ell$ on any unit cube $B_v$ for $v\in\Z^d$. For any finite union of unit cubes $D$, we define
\begin{displaymath}
    \SA_\infty[g,D]=\mathcal{H}^{d-1}[D\cap\partial(\{g\geq\ell\}_\infty)]
\end{displaymath}
where $\mathcal{H}^{d-1}$ denotes $(d-1)$-dimensional Hausdorff measure (and we define $\SA_\infty[g+p,D]$ analogously). We assume that the level sets of $g$ and $g+p$ restricted to the boundaries of unit cubes have zero surface area, that is
\begin{displaymath}
    \mathcal{H}^{d-1}[\{g=\ell\}\cap\partial B_v]=\mathcal{H}^{d-1}[\{g+p=\ell\}\cap\partial B_v]=0\qquad\text{for all }v\in\Z^d.
\end{displaymath}
This assumption will allow us to ignore such boundary components when using additivity of $\SA_\infty$ over unions of cubes below and will of course be satisfied with probability one for our intended application to the Gaussian fields $f$ and $\widetilde{f}_v$.

The topological contribution to $\Delta_v(B_w)$ is simple to control, as we can bound the change in surface area for unstable unit cubes by the total surface area for the original and perturbed functions:

\begin{lemma}\label{l:SATrivialBound}
    Let $g,p$ be as above and $w\in\Z^d$, then
    \begin{align*}
        \lvert \SA_\infty[g,B_w]-\SA_\infty[g+p,B_w]\rvert\leq\mathcal{H}^{d-1}[\{g=\ell\}\cap B_w]+\mathcal{H}^{d-1}[\{g+p=\ell\}\cap B_w].
    \end{align*}
\end{lemma}
\begin{proof}
    This follows from the triangle inequality along with the fact that the boundaries of $\{g\geq\ell\}_\infty$ and $\{g+p\geq\ell\}_\infty$ are subsets of $\{g=\ell\}$ and $\{g+p=\ell\}$ respectively.
\end{proof}

It remains to control the geometric contribution to changes in the surface area. Since we are only interested in the boundary of $\{g\geq\ell\}_\infty$ we cannot simply use the total change in area of level sets. Instead we need to compare components of $\{g=\ell\}$ and $\{g+p=\ell\}$ which are both contained in unbounded excursion sets. The isotopy constructed in Lemma~\ref{l:TopStab} allows us to do this.

Consider a unit cube $B_w$ which is finite-range stable for $(g,p)$ (i.e.\ $w\notin\mathcal{U}_\mathrm{FR}(g,p)$). Let $\mathrm{Comp}$ denote the set of components of $\{g=\ell\}\cap B_w$ and let $H$ be the stratified isotopy of $B_w$ specified in Lemma~\ref{l:TopStab}.

\begin{lemma}\label{l:SAComponentDecomp}
    Let $g,p$ be as above and $w\notin\mathcal{U}_\mathrm{FR}(g,p)$, then
    \begin{displaymath}
        \lvert\SA_\infty[g,B_w]-\SA_\infty[g+p,B_w]\rvert\leq\sum_{L\in\mathrm{Comp}}\lvert\mathcal{H}^{d-1}[H(L,1)]-\mathcal{H}^{d-1}[L]\rvert.
    \end{displaymath}
\end{lemma}
\begin{proof}
By Lemma~\ref{l:TopStab} we know that $L\in\mathrm{Comp}$ is a subset of $\{g\geq\ell\}_\infty$ if and only if $H(L,1)$ is a subset of $\{g+p\geq\ell\}_\infty$. Therefore
\begin{displaymath}
    \SA_\infty[g,B_w]-\SA_\infty[g+p,B_w]=\sum_{L\in\mathrm{Comp}\;:\; L\subseteq \{g\geq\ell\}_\infty}\mathcal{H}^{d-1}[L]-\mathcal{H}^{d-1}[H(L,1)].
\end{displaymath}
The statement of the lemma then follows from the triangle inequality and adding the non-negative terms which correspond to $L\subseteq\{g\geq\ell\}_{<\infty}$.
\end{proof}

Hence the geometric contribution to $\Delta_v(B_w)$ is bounded by the change in surface area for each component of the level set. Our strategy for controlling the latter is the following: assuming that the level set $\{g=\ell\}\cap B$ is somewhat stable compared to the perturbation $p$, we show that $\{g=\ell\}\cap B$ can be projected normally onto $\{g+p=\ell\}\cap B$ to yield a bijection (up to boundary effects) that identifies components which correspond under $H$ (i.e.\ the bijection maps $L$ to $H(L,1)$). The area formula and the implicit function theorem then allow us to bound the change in surface area for each component.

The stability assumptions we need are the following:
\begin{samepage}
\begin{assumption}\label{a:SA}
    Let $w\in\Z^d$. We assume that $g,p:\R^d\to\R$ are $C^2$ and satisfy
\begin{enumerate}
    \item $(g,p)$ is finite-range stable on $B_w$ at level $\ell$,
    \item for $x\in B_w$, if $g(x)=\ell$ then $\lvert\nabla g(x)\rvert> A_1$,
    \item $\lvert\nabla^2 g(x)\rvert< A_2$ for all $x\in B_w$,
    \item $\|p\|_{C^1(B_w)}<A_0$,
\end{enumerate}
for $A_1\in(0,1)$, $A_2\geq 1$ and $A_0< A_1^2/(C_dA_2)$ where $C_d>20$ depends only on the dimension $d$ and will be specified in the proof of Lemma~\ref{l:SA_AreaFormula}.
\end{assumption}
\end{samepage}

Given functions $g,p$ and constants $A_0,A_1,A_2$ as above, we define
\begin{displaymath}
    \mathcal{U}_\SA(g,p):=\left\{w\in\Z^d\;\middle|\;\inf_{x\in B_w:g(x)=\ell}\lvert\nabla g(x)\rvert<A_1\text{ or }\sup_{B_w}\lvert\nabla^2 g\rvert>A_2\text{ or }\|p\|_{C^1(B_w)}>A_0\right\}.
\end{displaymath}

The geometric contribution to the change in the surface area functional can then be bounded (using the argument described above) as follows:

\begin{lemma}\label{l:SA_Stable}
    Let $g,p:\R^d\to\R$ be $C^2$ functions such that $g$ and $g+p$ have no stratified critical points at level $\ell$ on any unit cube $B_v$ for $v\in\Z^d$. If $w\notin\mathcal{U}_\mathrm{FR}(g,p)\cup\mathcal{U}_\SA(g,p)$, then
    \begin{align*}
        \lvert\SA_\infty[g,B]-\SA_\infty[g+p,B]\rvert\leq C_d^\prime\frac{A_2A_0}{A_1^2}\mathcal{H}^{d-1}[\{g=\ell\}\cap B]&+2\mathcal{H}^{d-1}\Big[\{g=\ell\}\cap (\partial B)_{+\frac{7A_0}{A_1}}\Big]\\
        &+2\mathcal{H}^{d-1}\Big[\{g+p=\ell\}\cap (\partial B)_{+\frac{7A_0}{A_1}}\Big]
    \end{align*}
    where $B:=B_w$, $(\partial B)_{+a}:=\{x\in B\;|\;\dist(x,\partial B)\leq a\}$ and $C_d^\prime>0$ depends only on $d$.
\end{lemma}

The proof of this lemma is given in Section~\ref{s:SA_perturb}. 

We will require one more ingredient to control the magnitude of topological/geometric contributions: a moment bound for (total) surface area.

\begin{proposition}\label{p:SAMoments}
    Let $f:\R^d\to\R$ be a stationary $C^k$ Gaussian field ($k\geq 2$) with spectral density and let $\rho:\R^d\to\R$ be a deterministic $C^k$ function, then
    \begin{displaymath}
        \E\Big[\mathcal{H}^{d-1}[B_0\cap\{f+\rho=\ell\}]^{k^\prime}\Big]\leq C<\infty
    \end{displaymath}
    where $k^\prime=\frac{k(k+1)}{2}-2$ and $C\in(0,\infty)$ depends only on $\|\rho\|_{C^k(B_0)}$ and the distribution of $f$.

    Furthermore for any compact set $E$
    \begin{displaymath}
        \E\Big[\mathcal{H}^{d-1}[E\cap\{f+\rho=\ell\}]\Big]\leq C(1+\|\rho\|_{C^1(E)})\Vol[E]
    \end{displaymath}
    where $C\in(0,\infty)$ depends only on the distribution of $f$.
\end{proposition}
\begin{proof}
    The first statement is a direct application of \cite[Theorem~5.2]{aadlm23}.

    By the Kac-Rice theorem \cite[Theorem~6.8 and Proposition~6.12]{aw09}
    \begin{equation}\label{e:kac_rice1}
        \E\Big[\mathcal{H}^{d-1}[E\cap\{f+\rho=\ell\}]\Big]=\int_E \E\Big[\lvert\nabla(f+\rho)(x)\rvert\;\Big|\;(f+\rho)(x)=\ell \Big]p_{(f+\rho)(x)}(\ell)\;dx
    \end{equation}
    where $p_{(f+\rho)(x)}$ denotes the density of the Gaussian random variable $(f+\rho)(x)$. Since $f$ is stationary and has unit variance, this density is bounded above uniformly in $x$ by an absolute constant. By the triangle inequality and the fact that $\rho$ is deterministic
    \begin{equation}\label{e:kac_rice2}
        \E\Big[\lvert\nabla(f+\rho)(x)\rvert\;\Big|\;(f+\rho)(x)=\ell \Big]\leq \|\rho\|_{C^1(E)} +\E\Big[\lvert\nabla f(x)\rvert\;\Big|\;f(x)=\ell-\rho(x) \Big].
    \end{equation}
    For any $C^1$ Gaussian field with constant variance, the value of the field at a point is independent of its gradient at the same point (this is a standard fact in the literature, see \cite[Section~5.6]{at07}). Therefore
    \begin{equation}\label{e:kac_rice3}
        \E\Big[\lvert\nabla f(x)\rvert\;\Big|\;f(x)=\ell-\rho(x) \Big]=\E\Big[\lvert\nabla f(x)\rvert\Big]=\E\Big[\lvert\nabla f(0)\rvert\Big].
    \end{equation}
    Combining \eqref{e:kac_rice1}-\eqref{e:kac_rice3} proves the second statement of the proposition.
\end{proof}

We can now combine all of these estimates to verify the conditions of Proposition~\ref{p:Sufficient_Moments} for $\mu=\mu_\SA$. In the following two lemmas we will give slightly more general statements which will be useful later when proving positivity of the limiting variance (Theorem~\ref{t:posvar}).

\begin{lemma}\label{l:additivity_SA}
    Let $f$ satisfy Assumption~\ref{a:clt} and $\rho:\R^d\to\R$ be a deterministic $C^2$ function. The following holds with probability one: for any distinct $v_1,\dots,v_n\in\Z^d$
    \begin{displaymath}
        \SA_\infty\left[f+\rho,\bigcup_{i=1}^nB_{v_i}\right]=\sum_{i=1}^n\SA_\infty\left[f+\rho,B_{v_i}\right].
    \end{displaymath}
\end{lemma}
\begin{proof}
    Since $\SA_\infty[f+\rho,\cdot]$ is a measure, it is enough to show that $\SA_\infty[f+\rho,\partial B_v]=0$ almost surely for any $v\in\Z^d$. Fixing such a $v$, consider a stratum $S\subset B_v$ of dimension $d^\prime\leq d-1$. By the Kac-Rice theorem (\cite[Theorem~6.8 and Proposition~6.12]{aw09})
    \begin{displaymath}
        \mathcal{H}^{d^\prime-1}[\{f+\rho=\ell\}\cap S]<\infty
    \end{displaymath}
    almost surely, and hence
    \begin{displaymath}
        \SA_\infty[f+\rho,S]\leq\mathcal{H}^{d-1}[\{f+\rho=\ell\}\cap S]=0.
    \end{displaymath}
    Summing over the different strata of $\partial B_v$ completes the proof.
\end{proof}

\begin{lemma}\label{l:sa_perturbation_moments}
    Let $f$ satisfy Assumptions~\ref{a:clt} and~\ref{a:Decay} and let $p:\R^d\to\R$ be a (possibly degenerate) $C^{k}$ Gaussian field such that $f+p-\E[p]$ is stationary with a spectral density. For $w\in\Z^d$ let
    \begin{displaymath}
        M_2(w)=\sup_{\lvert\alpha\rvert,\lvert\gamma\rvert\leq 2}\sup_{x,y\in (w+\Lambda_2)}\left\lvert\partial_x^\alpha\partial_y^\gamma\Cov[p(x),p(y)]\right\rvert.
    \end{displaymath}
    We assume that $M_2(w)$ and $\|\E[p]\|_{C^1(B_w)}$ are bounded uniformly over $w$. Then there exists a collection of random variables $Y(w)$ such that
    \begin{equation}\label{e:sa_perturb_moments1}
        \lvert\mu_\SA(B_w,f,\ell)-\mu_\SA(B_w,f+p,\ell)\rvert\leq Y(w)
    \end{equation}
    almost surely and for any $\delta>0$ there exists $C_\delta>0$ (independent of $w$) such that
    \begin{equation}\label{e:sa_perturb_moments2}
        \E[Y(w)]\leq C_\delta\P\left(w\in\mathcal{U}_\mathrm{FR}(f,p)\right)^\frac{k^\prime-1}{k^\prime}+C_\delta\left( \sqrt{M_2(w)}+\|\E[p]\|_{C^1(B_w)}\right)^{\frac{k^\prime-1}{3k^\prime-1}-\delta}.
    \end{equation}
    Moreover if $p=p_v$ and $\beta>\beta_\SA(k)$ then there exists $\epsilon,\delta,C>0$ (independent of $w$) such that
    \begin{equation}\label{e:sa_perturb_moments3}
        \E[Y(w)]+\E[Y(w)^{2+\epsilon}]\leq C(1+\lvert v-w\rvert)^{-3d-\delta}
    \end{equation}
\end{lemma}
\begin{remark}
    The assumption that $f+p-\E[p]$ be stationary with a spectral density is much stronger than what is required for the conclusion of this lemma to hold. We use this assumption only in order to apply Proposition~\ref{p:SAMoments} but \cite[Theorem~5.2]{aadlm23} gives much weaker (albeit less concise) sufficient conditions for the same result. We apply this lemma with $p$ equal to $p_v$, $p_v+\rho$ or $\rho$ where $\rho$ is some deterministic function. Therefore the current statement is most convenient for our purposes.
\end{remark}
\begin{proof}[Proof of Lemma~\ref{l:sa_perturbation_moments}]
    Given $w\in\Z^d$, we define
    \begin{displaymath}
        \sigma^2_{w,p}:=\sup_{\lvert\alpha\rvert\leq 1}\sup_{x\in B_w}\Var[\partial^\alpha p(x)]\leq M_2(w)
    \end{displaymath}
    and in Assumption~\ref{a:SA} set
    \begin{equation}
    \begin{aligned}
        A_0=\|\E[p]\|_{C^1(B_w)}+\E\left[\|p-\E[p]\|_{C^1(B_w)}\right]+\sigma_{w,p}^{1-\delta},\\
        A_1=A_0^{\frac{k^\prime}{3k^\prime-1}-\delta}\qquad\text{and}\qquad A_2=c_d^\prime A_0^{-2\delta}
    \end{aligned}
    \end{equation}
    where $\delta,c_d^\prime>0$ will be specified later. By Kolmogorov's theorem \cite[Appendix~A.9]{ns16} there exists an absolute constant $C>0$ such that
    \begin{equation*}
        \E\left[\max_{\lvert\alpha\rvert\leq 1}\sup_{x\in B_w}\lvert\partial^\alpha p(x)-\E[\partial^\alpha p(x)]\rvert\right]\leq C \sqrt{M_2(w)}.
    \end{equation*}
    By assumption, $M_2(w)$ and $\|\E[p]\|_{C^1(B_w)}$ are bounded uniformly over $w$ by some constant $\widetilde{C}$, hence there exists $C^\prime>0$ depending only on $\widetilde{C}$ such that
    \begin{equation}\label{e:sa_perturb_moments4}
        A_0\leq \|\E[p]\|_{C^1(B_w)}+C^\prime M_2(w)^\frac{1-\delta}{2}.
    \end{equation}
    In particular $A_0$ is bounded uniformly over $w$ by some constant (depending only on $\widetilde{C}$). We note that $A_1^2/(C_dA_2)=A_0^{2k^\prime/(3k^\prime-1)}/(C_dc_d^\prime)$ and therefore, since $A_0$ is bounded, by choosing $c_d^\prime$ sufficiently small we can ensure that this expression is greater than $A_0$ for all $w$ (as required in Assumption~\ref{a:SA}).

    We next define
    \begin{equation*}
        Y(w)=\mathcal{L}(w)\ind_{w\in\mathcal{U}_\mathrm{FR}(f,p)\cup\mathcal{U}_\SA(f,p)}+\mathcal{L}_\partial(w)
    \end{equation*}
    where
    \begin{align*}
        \mathcal{L}(w)&:=\mathcal{H}^{d-1}[B_w\cap\{f=\ell\}]+\mathcal{H}^{d-1}[B_w\cap\{f+p=\ell\}]\\
        \mathcal{L}_\partial(w)&:=s_w\mathcal{H}^{d-1}[B_w\cap\{f=\ell\}]+2\mathcal{H}^{d-1}[(\partial B_w)_{+t_w}\cap\{f=\ell\}]+2\mathcal{H}^{d-1}[(\partial B_w)_{+t_w}\cap\{f+p=\ell\}]
    \end{align*}
    and
    \begin{displaymath}
        s_w:=C_d\frac{A_2A_0}{A_1^2}=C_dc_d^\prime A_0^\frac{k^\prime-1}{3k^\prime-1}\qquad\text{and}\qquad t_w:=\frac{7A_0}{A_1}=7A_0^{\frac{2k^\prime-1}{3k^\prime-1}+\delta}.
    \end{displaymath}
    Combining Lemmas~\ref{l:SATrivialBound} and~\ref{l:SA_Stable} shows that \eqref{e:sa_perturb_moments1} holds for our definition of $Y(w)$. It remains to prove the moment bounds \eqref{e:sa_perturb_moments2} and \eqref{e:sa_perturb_moments3}

    First we bound the probability that $w\in\mathcal{U}_\SA(f,p)$. The Borell-TIS inequality \cite[Theorem~2.1.1]{at07} states that for a continuous centred Gaussian field $h(t)$ defined on a compact set $T$
    \begin{displaymath}
        \P\left(\sup_{t\in T}h(t)-\E\Big[\sup_{t\in T}h(t)\Big]>u\right)\leq e^{-\frac{u^2}{2\sigma_T^2}}
    \end{displaymath}
    for all $u>0$, where $\sigma_T^2:=\sup_{t\in T}\Var[h(t)]$. Applying this to $\pm\partial^\alpha p$ for each $\lvert\alpha\rvert\leq 1$ and using the union bound, we have
    \begin{equation*}
    \begin{aligned}
        \P(\|p\|_{C^1(B_w)}>A_0)&\leq \P\left(\|p-\E[p]\|_{C^1(B_w)}>\E\left[\|p-\E[p]\|_{C^1(B_w)}\right]+\sigma_{w,p}^{1-\delta}\right)\\
        &\leq 2(d+1)e^{-\sigma_{w,p}^{-2\delta}/2}\leq 2(d+1)e^{-A_0^{-2\delta/(1-\delta)}/2}.
    \end{aligned}
    \end{equation*}
    A similar application of the Borell-TIS inequality to the second derivatives of $f$ yields that
    \begin{displaymath}
        \P\left(\sup_{x\in B_w}\lvert\nabla^2 f(x)\rvert>A_2\right)\leq C_1e^{-c(A_0^{-\delta}-C_2)^2}
    \end{displaymath}
    provided that $A_0^\delta>C_2$, where $c,C_1,C_2>0$ are constants depending only on the distribution of $f$. By increasing $C_1$ if necessary, we can relax the requirement that $A_0^\delta>C_2$.

    Next we note that by Lemma 7 of \cite{ns16} for any $\delta^\prime>0$ there exists $C_{\delta^\prime}>0$ such that
    \begin{displaymath}
        \P\left(\inf_{x\in B_w}\max\{\lvert f(x)-\ell\rvert,\lvert\nabla f(x)\rvert\}<\tau\right)<C_{\delta^\prime}\tau^{1-\delta^\prime}
    \end{displaymath}
    for all $\tau>0$. Applying this with $\tau=A_1$ and $\delta^\prime$ sufficiently small ensures that the right-hand side is at most $C_{\delta^\prime} A_1^{1-\delta^\prime}\leq C^\prime_\delta A_0^{k^\prime/(3k^\prime-1)-2\delta}$. Combining the three bounds above with the definition of $\mathcal{U}_\SA$, we have that for all $w\in\Z^d$
    \begin{equation}\label{e:sa_perturb_moments5}
        \P(w\in\mathcal{U}_\SA(f,p))\leq C_\delta A_0^{\frac{k^\prime}{3k^\prime-1}-2\delta}.
    \end{equation}

    We now work towards \eqref{e:sa_perturb_moments2}. By H\"older's inequality, Proposition~\ref{p:SAMoments} and \eqref{e:sa_perturb_moments5}
    \begin{align*}
        \E\left[\mathcal{L}(w)\ind_{w\in\mathcal{U}_\mathrm{FR}\cup\mathcal{U}_\SA}\right]&\leq\E\left[\mathcal{L}(w)^{k^\prime}\right]^{\frac{1}{k^\prime}}\left(\P(w\in\mathcal{U}_\mathrm{FR})+\P(w\in\mathcal{U}_\SA)\right)^\frac{k^\prime-1}{k^\prime}\leq C\P(w\in\mathcal{U}_\mathrm{FR})^\frac{k^\prime-1}{k^\prime}+C_\delta A_0^{\frac{k^\prime-1}{3k^\prime-1}-2\delta}
    \end{align*}
    where we have abbreviated $\mathcal{U}_\mathrm{FR}:=\mathcal{U}_\mathrm{FR}(f,p)$ and $\mathcal{U}_\SA:=\mathcal{U}_\SA(f,p)$. By the second statement of Proposition~\ref{p:SAMoments}
    \begin{equation*}
        \E[\mathcal{L}_\partial(w)]\leq C\max\{s_w,t_w\}\leq C^\prime A_0^\frac{k^\prime-1}{3k^\prime-1}.
    \end{equation*}
    Since $\delta>0$ can be chosen arbitrarily small, combining the last two equations with \eqref{e:sa_perturb_moments4} (and the definition of $Y(w)$) proves \eqref{e:sa_perturb_moments2}.

    Let us now assume that $p=p_v$ and $\beta>\beta_\SA(k)$. By definition of $p_v$
    \begin{displaymath}
        \Cov[p_v(x),p_v(y)]=2\int_{B_v}q(x-z)q(y-z)\;dz,
    \end{displaymath}
    so by Assumption~\ref{a:clt} and \eqref{e:sa_perturb_moments4}
    \begin{displaymath}
        A_0\leq C^\prime M_2(w)^{\frac{1-\delta}{2}}\leq C^{\prime\prime}(1+\lvert v-w\rvert)^{-\beta(1-\delta)}
    \end{displaymath}
    for all $w\in\Z^d$. By Lemma~\ref{l:ProbExtUnst}
    \begin{displaymath}
        \P(w\in\mathcal{U}_\mathrm{FR}(f,p_v))\leq c_\delta(1+\lvert v-w\rvert)^{-\beta+\delta}
    \end{displaymath}
    for any $\delta>0$. Hence taking $\delta>0$ sufficiently small and using \eqref{e:sa_perturb_moments2} we have
    \begin{align*}
        \E[Y(w)]&\leq C_\delta (1+\lvert v-w\rvert)^{-\beta\left(\frac{k^\prime-1}{3k^\prime-1}-\delta\right)}\leq C_\delta (1+\lvert v-w\rvert)^{-3d-\delta^\prime}
    \end{align*}
    where the final inequality follows upon taking $\delta,\delta^\prime>0$ sufficiently small since $\beta>\beta_\SA(k)=3d\frac{3k^\prime-1}{k^\prime-2}$. This verifies the first part of \eqref{e:sa_perturb_moments3}.

    Bounding the expectation of $Y(w)^{2+\epsilon}$ uses very similar arguments: by H\"older's inequality, Proposition~\ref{p:SAMoments}, Lemma~\ref{l:ProbExtUnst} and \eqref{e:sa_perturb_moments5}
    \begin{equation}\label{e:sa_perturb_moments6}
    \begin{aligned}
        \E[\mathcal{L}(w)^{2+\epsilon}\ind_{w\in\mathcal{U}_\SA\cup\mathcal{U}_\mathrm{FR}}]&\leq\E[\mathcal{L}(w)^{k^\prime}]^{\frac{2+\epsilon}{k^\prime}}\left(\P(w\in\mathcal{U}_\mathrm{FR})+\P(w\in\mathcal{U}_\SA)\right)^\frac{k^\prime-2-\epsilon}{k^\prime}\\
        &\leq C_\delta(1+\lvert v-w\rvert)^{-(\beta-\delta)\frac{k^\prime-2-\epsilon}{k^\prime}}+C_\delta A_0^{\frac{k^\prime-2-\epsilon}{3k^\prime-1}-2\delta}\\
        &\leq C_\delta^\prime (1+\lvert v-w\rvert)^{-\beta\left(\frac{k^\prime-2-\epsilon}{3k^\prime-1}-2\delta\right)(1-\delta)}.
    \end{aligned}
    \end{equation}
    Once again, since $\beta>\beta_\SA(k)$ we can ensure that the above exponent is strictly less than $-3d$ by taking $\epsilon$ and $\delta$ sufficiently small. By Proposition~\ref{p:SAMoments}
    \begin{equation}\label{e:sa_perturb_moments7}
    \begin{aligned}
        \E[(s_w\mathcal{H}^{d-1}[B_w\cap\{f=\ell\}])^{2+\epsilon}]\leq Cs_w^{2+\epsilon}&\leq C^{\prime}(1+\lvert v-w\rvert)^{-\beta(1-\delta)(2+\epsilon)\frac{k^\prime-1}{3k^\prime-1}}\\
        &\leq C^{\prime\prime}(1+\lvert v-w\rvert)^{-4d}.
    \end{aligned}
    \end{equation}
    We temporarily denote $X:=\mathcal{H}^{d-1}[(\partial B_w)_{+t_w}\cap \{f=\ell\}]$ and define $\theta=\frac{k^\prime-2-\epsilon}{(k^\prime-1)(2+\epsilon)}$. Then by Littlewood's inequality for $L^p$ spaces and Proposition~\ref{p:SAMoments}
    \begin{equation}\label{e:sa_perturb_moments8}
        \E\left[X^{2+\epsilon}\right]\leq\E\left[X^{k^\prime}\right]^{\frac{(1-\theta)(2+\epsilon)}{k^\prime}}\E\left[X\right]^{\theta(2+\epsilon)}\leq C t_w^{\theta(2+\epsilon)}\leq C^\prime (1+\lvert v-w\rvert)^{-\beta\left(\frac{k^\prime-2-\epsilon}{3k^\prime-1}-2\delta\right)}.
    \end{equation}
    As before we can ensure that the exponent of the right-most term is strictly less than $-3d$ by taking $\epsilon,\delta>0$ sufficiently small. Combining \eqref{e:sa_perturb_moments6}-\eqref{e:sa_perturb_moments8} shows that
    \begin{displaymath}
        \E[Y(w)^{2+\epsilon}]\leq C(1+\lvert v-w\rvert)^{-3d-\delta}
    \end{displaymath}
    for some $\epsilon,\delta>0$ and all $w\in\Z^d$, which verifies the second part of \eqref{e:sa_perturb_moments3} and hence completes the proof of the lemma.    
\end{proof}

\begin{proof}[Proof of Theorem~\ref{t:clt} (for $\star=\SA$)]
    The surface area functional $\mu_\SA$ is dominated by the total area of the level set and so has finite second moments courtesy of Proposition~\ref{p:SAMoments}. Lemmas~\ref{l:additivity_SA} and~\ref{l:sa_perturbation_moments} verify the assumptions of Proposition~\ref{p:Sufficient_Moments} so we see that the stabilisation, bounded moment and moment decay conditions hold. Hence an application of Theorem~\ref{t:GenCLT} yields the desired result.
\end{proof}
\subsection{Euler characteristic}
Turning to the Euler characteristic, we start by describing some fundamental properties of this functional which we make use of below. Readers who are unfamiliar with the Euler characteristic may find it helpful to consult \cite[Chapter~6]{at07} for an overview (in the setting of smooth Gaussian fields).

The Euler characteristic is well defined for a class of sets known as `basic complexes'. The precise definition is given in \cite[Chapter~6]{at07}, however for our purposes it will be sufficient to know that this class includes the excursion sets of functions satisfying certain regularity conditions:

\begin{lemma}
    Let $D\subset\R^d$ be the union of a finite number of closed faces (of any dimension) of unit cubes. If $g:D\to\R$ is suitably regular (as defined in \cite[Chapter~6]{at07}) on $D$ at level $\ell$ then any union of components of $\{g\geq\ell\}\cap D$ is a basic complex.
\end{lemma}
\begin{proof}
    This result is proven for the closed face of a cube in \cite[Theorem~6.2.2]{at07}, the proof generalises trivially to a union of such faces.
\end{proof}

\begin{lemma}
    Let $f$ satisfy Assumption~\ref{a:clt} then with probability one $f$ is suitably regular on any union of closed faces of unit cubes at level $\ell$.
\end{lemma}
\begin{proof}
    This is given for a single face by \cite[Theorem~11.3.3]{at07}. The lemma holds since the number of faces of unit cubes is countable.
\end{proof}

The Euler characteristic is well known to be additive and topologically invariant:

\begin{lemma}\label{l:ECbasics}
Let $A_1$ and $A_2$ be basic complexes.
\begin{enumerate}
    \item If $A_1\cup A_2$ and $A_1\cap A_2$ are also basic complexes then
    \begin{displaymath}
    \EC[A_1\cup A_2]=\EC[A_1]+\EC[A_2]-\EC[A_1\cap A_2]
    \end{displaymath}
    \item If $A_1$ is homotopy equivalent to $A_2$ then $\EC[A_1]=\EC[A_2]$.
\end{enumerate}
\end{lemma}
\begin{proof}
The first property is proven in \cite[Theorem~6.1.1]{at07} the second property in \cite[Theorem~2.44]{hat02}.
\end{proof}

We will also make use of the fact that the Euler characteristic of an excursion set can be bounded in terms of the number of critical points of the underlying function:

\begin{lemma}\label{l:ECBound}
    Let $D$ be a union of closed faces of the unit cubes $B_{v_1},\dots,B_{v_n}$ and $g:D\to\R$ be suitably regular on $D$ at level $\ell$. Let $U$ be a union of connected components of $\{g\geq\ell\}\cap D$, then there exists $c_d>0$ depending only on $d$ such that
    \begin{displaymath}
        \lvert\EC[U]\rvert\leq c_d\sum_{i=1}^n\overline{N}_\mathrm{Crit}(B_{v_i},g)
    \end{displaymath}
    where $\overline{N}_\mathrm{Crit}(A,g)$ denotes the number of stratified critical points of $g$ in $A$.
\end{lemma}
\begin{proof}
    Let $F$ be one of the faces which makes up $D$, then \cite[Section~9.4]{at07} gives an expression for $\EC[\{g\geq\ell\}\cap F]$ as an alternating sum of stratified critical points of $g$ contained in $\{g\geq\ell\}\cap F$. In particular taking this sum over the components of $\{g\geq\ell\}\cap F$ in $U$, this implies that $\lvert \EC[U\cap F]\rvert\leq \overline{N}_\mathrm{Crit}(F,g)$.

    Given a closed face $F$ of a unit cube we let $\Tilde{\partial}F$ denote the union of all closed faces of unit cubes that are strictly contained in $F$. So if $F$ is a two-dimensional face then $\Tilde{\partial}F$ is the union of four line segments and if $F$ is a singleton (i.e.\ vertex) then $\Tilde{\partial}F=\emptyset$. By additivity of the Euler characteristic (Lemma~\ref{l:ECbasics})
    \begin{displaymath}
        \EC[U]=\sum_{F\in\mathrm{Face}(D)}\EC[U\cap F]-\EC[U\cap\Tilde{\partial}F]
    \end{displaymath}
    where $\mathrm{Face}(D)$ denotes the set of all faces of unit cubes in $D$. Then by the conclusion of the previous paragraph
    \begin{align*}
        \lvert\EC[U]\rvert&\leq \sum_{F\in\mathrm{Face}(D)}2\overline{N}_\mathrm{Crit}(F,g)\leq\sum_{i=1}^n\sum_{F\in\mathrm{Face}(B_{v_i})}2\overline{N}_\mathrm{Crit}(F,g)
    \end{align*}
    which implies the statement of the lemma since the number of faces of each unit cube is bounded by a constant depending only on the dimension $d$.
\end{proof}

With these preliminary results we can now decompose the change in the Euler characteristic under perturbation. Since the Euler characteristic of a set depends only on its topology (Lemma~\ref{l:ECbasics}) there will be no geometric contributions to $\Delta_v(\cdot)$, which simplifies our arguments somewhat compared to the previous two subsections. Given a (suitably regular) $C^2$ function $g:\R^d\to\R$ and $D\subset\R^d$ we define $\EC_\infty[D,g]=\EC[\{g\geq\ell\}_\infty\cap D]$.

\begin{lemma}\label{l:ChangeDecompEC}
If $g$ and $g+p$ are suitably regular on $\Lambda_n$ then
\begin{displaymath}
    \lvert\EC_\infty[\Lambda_n,g]-\EC_\infty[\Lambda_n,g+p]\rvert\leq c_d\sum_{w\in\Z^d\cap\Lambda_n}\big(\overline{N}_\mathrm{Crit}(B_{w},g)+\overline{N}_\mathrm{Crit}(B_{w},g+p)\big)\ind_{w\in\mathcal{U}_{\mathrm{FR}}(g,p)}.
\end{displaymath}
\end{lemma}
\begin{proof}
    We define the unstable subset of $\Lambda_n$ as
    \begin{displaymath}
        Q:=\Lambda_n\cap\bigcup_{w\in\mathcal{U}_\mathrm{FR}(g,p)}B_w
    \end{displaymath}
    then by the first point of Lemma~\ref{l:ECbasics}, for $h=g,g+p$
    \begin{equation}\label{e:ECPerturb1}
        \EC_\infty[\Lambda_n,h]=\EC_\infty[\overline{\Lambda_n\setminus Q},h]+\EC_\infty[Q,h]-\EC_\infty[\overline{\Lambda_n\setminus Q}\cap Q,h].
    \end{equation}
    By Lemma~\ref{l:TopStab}, $\{g\geq\ell\}_\infty\cap \overline{\Lambda_n\setminus Q}$ is homeomorphic to $\{g+p\geq\ell\}_\infty\cap \overline{\Lambda_n\setminus Q}$ and so by Lemma~\ref{l:ECbasics}
    \begin{equation}\label{e:ECPerturb2}
        \EC_\infty[\overline{\Lambda_n\setminus Q},g]=\EC_\infty[\overline{\Lambda_n\setminus Q},g+p].
    \end{equation}
    By Lemma~\ref{l:ECBound}, for $h=g,g+p$
    \begin{displaymath}
        \lvert\EC_\infty[Q,h]\rvert,\lvert\EC_\infty[\overline{\Lambda_n\setminus Q}\cap Q,h]\rvert\leq c_d\sum_{w\in \Lambda_n\cap\mathcal{U}_\mathrm{FR}(g,p)}\overline{N}_\mathrm{Crit}(B_w,h).
    \end{displaymath}
    Combining this with \eqref{e:ECPerturb1} and \eqref{e:ECPerturb2} proves the lemma.
\end{proof}

To obtain a probabilistic bound on the change in the Euler characteristic under perturbation, we require a moment bound on the number of critical points:

\begin{proposition}\label{p:TMB}
Let $f:\R^d\to\R$ be a stationary $C^{k}$ Gaussian field ($k\geq 2$) such that the support of its spectral measure contains an open set. Let $\rho\in C^k(\R^d)$ be deterministic such that $\|\rho\|_{C^k(\R^d)}\leq C_0$. There exists $C^\prime<\infty$ depending only on $C_0$ and the distribution of $f$ such that for all $v\in\Z^d$
\begin{displaymath}
    \E\left[N_\mathrm{Crit}(B_v,f+\rho)^{k-1}\right]\leq C^\prime
\end{displaymath}
where $N_\mathrm{Crit}(B_v,f+\rho)$ denotes the number of critical points of $f+\rho$ in $B_v$.
\end{proposition}
\begin{proof}
    This is a special case of \cite[Theorem~1.2, Remark~1.3]{gs23}.
\end{proof}
Observe that the conditions of this result hold for any field satisfying Assumption~\ref{a:clt} with the same value of $k$ since $K\in C^{2k+\epsilon}$ ensures that $f\in C^{k}$ almost surely.

With these tools we can now verify the conditions for the CLT when $\star=\EC$.

\begin{lemma}\label{l:StabilisationEC}
Let $f$ satisfy Assumption~\ref{a:clt} for $\beta>\beta_\EC(k)$ and Assumption~\ref{a:Decay}, then $\mu_\EC$ satisfies the stabilisation condition in Theorem~\ref{t:GenCLT}.

If, in addition, $\ell<-\ell_c$ and $\rho:\R^d\to\R$ satisfies the assumptions of Lemma~\ref{l:ProbExtUnst} then there exists a random variable $\Delta_0(\rho)$ such that for any sequence of cubes $D_n:=v_n+\Lambda_n$ satisfying $\liminf_nD_n=\R^d$, we have
\begin{equation}\label{e:EC_stab_gen}
    \Delta_0(D_n,\rho):=\mu_\EC(D_n,f)-\mu_\EC(D_n,\widetilde{f}_0+\rho)\to\Delta_0(\rho)
\end{equation}
almost surely as $n\to\infty$. Furthermore
\begin{displaymath}
    \E[\Delta_0(\rho)]=\lim_{n\to\infty}\E[\mu_\EC(D_n,f)-\mu_\EC(D_n,f+\rho)].
\end{displaymath}
\end{lemma}
\begin{proof}
We will prove only the statements for $\Delta_0(\rho)$; the stabilisation property follows from setting $\rho=0$ (in which case it will be apparent that the proof does not require $\ell<-\ell_c$).

Let $D_n=v_n+\Lambda_n$ be given such that $\liminf_nD_n=\R^d$. By Lemma~\ref{l:ProbExtUnst}, with probability one we may choose $m$ and $N$ such that $\mathcal{U}_\mathrm{FR}(f,p_0+\rho)\subseteq\Lambda_{m-1}$ and $\Lambda_m\subseteq\cap_{n>N}D_n$. 

For $h=f,\widetilde{f}_0+\rho$ by Lemma~\ref{l:ECbasics}
\begin{equation}\label{e:ECStab}
    \EC_\infty[D_n,h]=\EC_\infty[\Lambda_m,h]+\EC_\infty[\overline{D_n\setminus\Lambda_m},h]-\EC_\infty[\partial\Lambda_m,h].
\end{equation}
Since $\mathcal{U}_\mathrm{FR}(f,p_0+\rho)\subseteq\Lambda_{m-1}$, by Lemmas~\ref{l:TopStab} and~\ref{l:ECbasics}
\begin{displaymath}
    \EC_\infty[\overline{D_n\setminus\Lambda_m},f]=\EC_\infty[\overline{D_n\setminus\Lambda_m},\widetilde{f}_0+\rho]\quad\text{and}\quad\EC_\infty[\partial\Lambda_m,f]=\EC_\infty[\partial\Lambda_m,\widetilde{f}_0+\rho].
\end{displaymath}
Substituting into \eqref{e:ECStab} we see that
\begin{align*}
    \Delta_0(D_n,\rho)&=\EC_\infty[D_n,f]-\EC_\infty[D_n,\widetilde{f}_0+\rho]=\EC_\infty[\Lambda_m,f]-\EC_\infty[\Lambda_m,\widetilde{f}_0+\rho]=:\Delta_0(\rho)
\end{align*}
which is constant for all $n>N$, proving convergence. Moreover, for any other sequence $D_n^\prime$ we will have $\Delta_0(D^\prime_n,\rho)=\Delta_0(\rho)$ for sufficiently large $n$, which verifies \eqref{e:EC_stab_gen}.

Turning to the final statement, for any $n\in\N$ by Lemma~\ref{l:ChangeDecompEC}
\begin{align*}
    \E\left[\sup_{n\in\N}\lvert\Delta_0(D_n,\rho)\rvert\right]\leq c_d\E\left[\sum_{v\in\Z^d}\left(\overline{N}_\mathrm{Crit}(B_v,f)+\overline{N}_\mathrm{Crit}(B_v,\widetilde{f}_0+\rho)\right)\ind_{v\in\mathcal{U}_\mathrm{FR}(f,p_0+\rho)}\right].
\end{align*}
By H\"older's inequality, Proposition~\ref{p:TMB} and Lemma~\ref{l:ProbExtUnst} the above quantity is bounded by
\begin{align*}
    c_d&\sum_{v\in\Z^d}\E\left[\left(\overline{N}_\mathrm{Crit}(B_v,f)+\overline{N}_\mathrm{Crit}(B_v,\widetilde{f}_0+\rho)\right)^{k-1}\right]^\frac{1}{k-1}\P(v\in\mathcal{U}_\mathrm{FR}(f,p_0+\rho))^{\frac{k-2}{k-1}}\leq c_\delta\sum_{v\in\Z^d}(1+\lvert v\rvert)^{-(\beta-\delta)\frac{k-2}{k-1}}
\end{align*}
for any $\delta>0$. Taking $\delta$ sufficiently small ensures that this expression is finite, since $\beta>\beta_\EC= 3\frac{k-1}{k-3}d$. We conclude that the sequence $\Delta_0(D_n,\rho)$ is dominated by an integrable random variable, hence by the dominated convergence theorem
\begin{align*}
    \E[\Delta_0(\rho)]=\lim_{n\to\infty}\E[\Delta_0(D_n,\rho)]&=\lim_{n\to\infty}\E[\mu_\EC(D_n,f)-\mu_\EC(D_n,\widetilde{f}_0+\rho)]\\
    &=\lim_{n\to\infty}\E[\mu_\EC(D_n,f)-\mu_\EC(D_n,f+\rho)]
\end{align*}
since $f$ and $\widetilde{f}_0$ have the same distribution. This completes the proof of the lemma.
\end{proof}

\begin{proof}[Proof of Theorem~\ref{t:clt} (for $\star=\EC$)]
    Since the Euler characteristic of an excursion set is dominated by the number of critical points (Lemma~\ref{l:ECBound}) which have finite second moments (Proposition~\ref{p:TMB}) we see that $\mu_\EC$ has finite second moments. Lemma~\ref{l:StabilisationEC} verifies the stabilisation condition for $\mu_\EC$ so if we can prove the bounded moments and moment decay conditions then an application of Theorem~\ref{t:GenCLT} will yield Theorem~\ref{t:clt} for $\star=\EC$.

    Given $v,w\in\Z^d$ we let
    \begin{displaymath}
        Y_v(w):=c_d\big(\overline{N}_\mathrm{Crit}(B_w,f)+\overline{N}_\mathrm{Crit}(B_w,f+p_v)\big)\ind_{w\in\mathcal{U}_\mathrm{FR}(f,p_v)}
    \end{displaymath}
    where $c_d>0$ is taken from the statement of Lemma~\ref{l:ChangeDecompEC}. This lemma then implies that
    \begin{displaymath}
        \lvert\Delta_v(u+\Lambda_n)\rvert\leq \sum_{w\in\Z^d\cap(u+\Lambda_n)}Y_v(w)
    \end{displaymath}
    for any $u\in\Z^d$ and $n\in\N$. Adopting the notation $\overline{N}_\mathrm{Crit}(w):=\overline{N}_\mathrm{Crit}(B_w,f)+\overline{N}_\mathrm{Crit}(B_w,f+p_v)$, by H\"older's inequality, Proposition~\ref{p:TMB} and Lemma~\ref{l:ProbExtUnst}
    \begin{displaymath}
        \E\left[Y_v(w)^{2+\epsilon}\right]\leq c^\prime_d\E\left[\overline{N}_\mathrm{Crit}(w)^{k-1}\right]^{\frac{2+\epsilon}{k-1}}\P\left(w\in\mathcal{U}_\mathrm{FR}(f,p_v)\right)^{\frac{k-3-\epsilon}{k-1}}\leq C_\delta(1+\lvert v-w\rvert)^{-(\beta-\delta)\frac{k-3-\epsilon}{k-1}}
    \end{displaymath}
    for any $\delta>0$. Since $\beta>\frac{3(k-1)}{k-3}d$, taking $\epsilon$ and $\delta$ sufficiently small ensures that the final exponent is less than $-3d$. A similar argument yields the stronger bound
    \begin{displaymath}
        \E\left[Y_v(w)\right]\leq c^\prime_d\E\left[\overline{N}_\mathrm{Crit}(w)^{k-1}\right]^{\frac{1}{k-1}}\P\left(w\in\mathcal{U}_\mathrm{FR}(f,p_v)\right)^{\frac{k-2}{k-1}}\leq C_\delta(1+\lvert v-w\rvert)^{-(\beta-\delta)\frac{k-2}{k-1}}.
    \end{displaymath}
    Together the last three equations allow us to apply Proposition~\ref{p:Sufficient_Moments} which shows that $\mu_\EC$ satisfies the bounded moments and moment decay conditions as required.
\end{proof}

\subsection{First-order behaviour and finitary CLT}
To complete this section, we prove the `law of large numbers' for our functionals (Theorem~\ref{t:FirstOrder}) and the finitary version of our CLT (Corollary~\ref{c:FinitaryCLT}).

Our characterisation of the first-order behaviour of $\mu_\star$ follows from an ergodic argument. 

\begin{proof}[Proof of Theorem~\ref{t:FirstOrder} (for $\star\in\{\Vol,\SA\}$)]
    We recall the white noise representation of our field $f=q\ast W$ and that $W_v=W|_{B_v}$ for $v\in\Z^d$. Given $u\in\Z^d$ we let $\tau_u$ denote translation by $u$ (i.e.\ $(\tau_uW)_v=W_{v-u}$). We also define $T_1,\dots,T_d$ to be translations by distance one in the positive direction of each coordinate axis respectively. Since the $W_v$ are independent and identically distributed, $T_1,\dots,T_d$ are measure preserving transformations of $W$ and the invariant $\sigma$-algebra for each $T_i$ is trivial.

    For $\star\in\{\Vol,\SA\}$, we define $G(W)=\mu_\star([0,1]^d)=\mu_\star(B_0)$. If $\star=\Vol$, then $\lvert G\rvert\leq 1$ and so $G\in L^p$ for all $p>0$. If $\star=\SA$, then by Proposition~\ref{p:SAMoments}, $G\in L^p$ for $p\leq k(k+1)/2-2$. By the definition $f=q\ast W$, for any $v\in\Z^d$ we have
    \begin{displaymath}
        G(\tau_vW)=\mu_\star(-v+[0,1]^d)=\mu_\star(B_{-v}).
    \end{displaymath}
    Hence, since $\tau_v=T_1^{v_1}\dots T_d^{v_d}$, applying the ergodic theorem (Theorem~\ref{t:ergodic}) yields
    \begin{equation}\label{e:ErgodicConv}
        \frac{1}{(2n)^d}\sum_{v\in\Z^d\cap(-n,n]^d}G(\tau_{v}W)=\frac{1}{(2n)^d}\sum_{i=1}^d\sum_{-n< k_i\leq n}G(T_1^{k_1}\dots T_d^{k_d}W)\to\E[G(W)]
    \end{equation}
    where convergence occurs almost surely and in $L^p$. By additivity of $\mu_\star$ (Lemmas~\ref{l:ChangeDecomp} and~\ref{l:additivity_SA})%By Lemma~\ref{l:ChangeDecomp} (when $\star=\Vol$) and Lemma~\ref{l:additivity_SA} (when $\star=\SA$)
    \begin{equation}\label{e:ErgodicAdditivity}
        \sum_{v\in\Z^d\cap(-n,n]^d}\mu_\star(B_{-v})=\mu_\star(\Lambda_n)
    \end{equation}
    almost surely, completing the proof of convergence.
    
    By Fubini's theorem, stationarity and Theorem~\ref{t:PhaseTransition}
    \begin{displaymath}
        \E[\mu_\Vol(B_0)]=\E\left[\int_{[0,1]^d}\ind_{x\in\{f\geq\ell\}_\infty}\;dx\right]=\P(0\in\{f\geq\ell\}_\infty)>0.
    \end{displaymath}
    Recalling that $\{f\geq\ell\}_{<\infty}$ denotes the union of all bounded components of $\{f\geq\ell\}$,
    \begin{displaymath}
        \P(0\in\{f\geq\ell\}_\infty)=\P(0\in\{f\geq\ell\})-\P(0\in\{f\geq\ell\}_{<\infty})<1-\Phi(\ell)
    \end{displaymath}
    since $f(0)$ is a standard Gaussian and the origin has a positive probability of being contained in a bounded excursion component (a precise statement of the latter fact follows from \cite[Appendix~C.2, Lemma~11]{ns16}).
    For almost any realisation of $f$, there exists an $\ell$ such that $0\in\{f\geq\ell\}_\infty$. (To see this, note that $\{f\geq\ell\}_\infty$ exists with probability one and so intersects $\Lambda_N$ for some random $N$. We can then take $\ell<\inf_{\Lambda_N}f$.) Let $\ell_0\in\Z$ be the (random) largest integer such that $0\in\{f\geq\ell_0\}_\infty$. Then
    \begin{displaymath}
        \P(0\in\{f\geq\ell\}_\infty)\geq\P(\ell<\ell_0)\to 1
    \end{displaymath}
    as $\ell\to-\infty$. This completes the proof in the case $\star=\Vol$.

    For $\star=\SA$, by the Kac-Rice theorem and stationarity
    \begin{align*}
        c_\SA(\ell)&=\E[\mu_\SA(B_0)]=\int_{B_0}\E[\lvert\nabla f(x)\rvert\ind_{x\in\{f\geq\ell\}_\infty}|f(x)=\ell]\phi(\ell)\;dx=\E[\lvert\nabla f(0)\rvert\ind_{0\in\{f\geq\ell\}_\infty}|f(0)=\ell]\phi(\ell).
    \end{align*}
    Applying Kac-Rice once more,
    \begin{displaymath}
        \E[\mathcal{H}^{d-1}[B_0\cap\{f=\ell\}]]=\E\big[\lvert\nabla f(0)\rvert\;\big|\;f(0)=\ell\big]\phi(\ell)=\E\big[\lvert\nabla f(0)\rvert\big]\phi(\ell)
    \end{displaymath}
    where the latter equality holds because $f(0)$ and $\nabla f(0)$ are independent (this follows from differentiating the covariance function at zero, see \cite[Chapter~5.6]{at07}). Letting $\mu_\SA^{<\infty}(B_0):=\mathcal{H}^{d-1}[B_0\cap\partial(\{f\geq\ell\}_{<\infty})]$, in order to prove the desired bounds on $c_\SA$, it is enough to show that neither of $\mu_\SA(B_0)$ or $\mu_\SA^{<\infty}(B_0)$ are almost surely zero. With probability one, $\{f\geq\ell\}_\infty$ and $\{f<\ell\}$ are both non-empty and therefore so is the boundary of $\{f\geq\ell\}_\infty$ (which consists of $C^1$-smooth, $(d-1)$-dimensional surfaces). Since $f$ is stationary, there is a positive probability that this boundary intersects the interior of $B_0$, and on such an event we have $\mu_\SA(B_0)>0$. Replacing $\{f\geq\ell\}_\infty$ with $\{f\geq\ell\}_{<\infty}$ in this argument proves the corresponding property for $\mu_\SA^{<\infty}(B_0)$ (the fact that $\{f\geq\ell\}_{<\infty}$ is non-empty follows from \cite[Theorem~1]{ns16} which is stated for $\ell=0$, however the proof of which is valid for all $\ell\in\R$). Since $\phi(\ell)\to 0$ as $\ell\to-\infty$, the same holds true for $c_\SA(\ell)$, completing the proof of the theorem for $\star=\SA$.
\end{proof}

\begin{proof}[Proof of Theorem~\ref{t:FirstOrder} (for $\star=\EC$)]
    We break up the proof for the different statements in the theorem:
    
    \underline{Convergence:} When $\star=\EC$, the previous argument for convergence does not quite work because \eqref{e:ErgodicAdditivity} fails (since the cubes $B_v$ overlap on their boundaries and the Euler characteristic of excursion sets can take non-zero value on these boundaries). To proceed we will consider a modified version of the functional which is additive over unit cubes. Given a box $D=[a_1,b_1]\times\dots\times[a_d,b_d]$ we define the \emph{lower left boundary} of $D$, denoted $\partial_\mathrm{LL}D$ to be
    \begin{displaymath}
        \partial_\mathrm{LL}D:=\{x\in\partial D\;|\;x_i=a_i\text{ for some }i=1,\dots,d\}.
    \end{displaymath}
    (The choice of terminology should be clear from considering the case $d=2$.) We now let
    \begin{displaymath}
        \psi(D):=\mu_\EC(D)-\mu_\EC(\partial_\mathrm{LL}D)\quad\text{and}\quad G(W)=\psi([0,1]^d).
    \end{displaymath}
    Since $G$ is dominated by (a constant times) the number of stratified critical points of $B_0$ (Lemma~\ref{l:ECBound}) and the latter has finite $p$-th moment for $p\leq k-1$ (Proposition~\ref{p:TMB}) we see that $G$ satisfies the requirements of the ergodic theorem and so \eqref{e:ErgodicConv} holds for this choice of $G$. Using the additivity in Lemma~\ref{l:ECbasics} we have
    \begin{displaymath}
        \sum_{v\in\Z^d\cap(-n,n]^d}\psi(B_{-v})=\psi(\Lambda_n)=\mu_\EC(\Lambda_n)-\mu_\EC(\partial_\mathrm{LL}\Lambda_n).
    \end{displaymath}
    The convergence result for $\star=\EC$ will now follow if we can show that $\mu_\EC(\partial_\mathrm{LL}\Lambda_n)n^{-d}\to 0$ almost surely and in $L^1$. Using Lemma~\ref{l:ECBound} once more, it is enough to show the corresponding convergence for $\overline{N}_\mathrm{Crit}(\partial_\mathrm{LL}\Lambda_n)n^{-d}$. By considering each stratum of $\partial_\mathrm{LL}\Lambda_n$ separately, using stationarity of $f$ and Proposition~\ref{p:TMB}, we have
    \begin{equation}\label{e:ErgodicBound}
        \E\big[\overline{N}_\mathrm{Crit}(\partial_\mathrm{LL}\Lambda_n)^2\big]\leq cn^{2(d-1)}
    \end{equation}
    for some constant $c>0$ depending only on the distribution of $f$. Since $\overline{N}_\mathrm{Crit}$ is non-negative, this shows that $\overline{N}_\mathrm{Crit}(\partial_\mathrm{LL}\Lambda_n)n^{-d}$ converges to zero in $L^2$. By the Cauchy-Schwarz inequality, convergence also occurs in $L^1$ as required. Almost sure convergence follows from the fact that
    \begin{displaymath}
        \E\left[\sum_{n=1}^\infty\left(\frac{\overline{N}_\mathrm{Crit}(\partial_\mathrm{LL}\Lambda_n)}{n^{d}}\right)^2\right]<\infty,
    \end{displaymath}
    where we have used \eqref{e:ErgodicBound} once more.

    \underline{Expression for mean functional:} For an open face (of a unit cube) $F$ of dimension $k$ and for $i=0,\dots,k$ we define
    \begin{displaymath}
        N_{\mathrm{Crit},\infty}(F,i)=\#\left\{x\in F\;\middle|\;\nabla|_Ff(x)=0,\;x\in\{f\geq\ell\}_\infty,\;\mathrm{ind}\nabla|_F^2 f(x)=i\right\}
    \end{displaymath}
    where $\mathrm{ind}\nabla|_F^2f(x)$ denotes the number of negative eigenvalues of $\nabla|_F^2f(x)$. Now let $E$ be a closed face of $B_0$ of dimension $k$ which contains the origin. We claim that
    \begin{equation}\label{e:EC_claim}
        \E[\mu_\EC(E)-\mu_\EC(\partial_\mathrm{LL}E)]=\sum_{i=0}^k(-1)^{k-i}\E[N_{\mathrm{Crit},\infty}(\mathrm{int}E,i)]
    \end{equation}
    where $\mathrm{int}E$ denotes the interior of $E$. Applying this claim to $B_0$, using the ergodic convergence already established, the Kac-Rice theorem and stationarity, we have
    \begin{align*}
        c_\EC(\ell)=\E[\psi([0,1]^d])]&=\int_{B_0}\sum_{i=0}^d(-1)^{d-i}\E\left[\lvert\det\nabla^2 f(x)\rvert\ind_{x\in\{f\geq\ell\}_\infty}\ind_{\mathrm{ind}\nabla^2 f(x)=i}\;\middle|\;\nabla f(x)=0\right]\varphi_{\nabla f(x)}(0)\;dx\\
        &=(-1)^d\E\left[\det\nabla^2 f(x)\ind_{x\in\{f\geq\ell\}_\infty}\;\middle|\;\nabla f(x)=0\right]\varphi_{\nabla f(x)}(0)
    \end{align*}
    as required.

    It remains to prove the claim. For $E$ as in the statement of the claim, we first observe that $\partial_\mathrm{LL}E$ can be partitioned into the set of $F\setminus\partial_\mathrm{LL}F$ where $F$ is a closed face of $E$ which contains the origin and $\partial_\mathrm{LL}\{0\}:=\emptyset$. To see this, note that for any two such faces $F$ and $G$, $F\cap G$ is a closed face of $E$ containing the origin which must therefore be contained in $\partial_\mathrm{LL}F\cap\partial_\mathrm{LL}G$. This proves that $F\setminus\partial_\mathrm{LL}F$ and $G\setminus\partial_\mathrm{LL}G$ are disjoint. Moreover for any $x\in E$ one can find such a face $F$ such that $F\setminus\partial_\mathrm{LL}F$ contains $x$ be identifying the non-zero elements of $x$. Therefore, by additivity of the Euler characteristic (Lemma~\ref{l:ECbasics}),
    \begin{equation}\label{e:EC_limiting_formula}
        \E\left[\mu_\EC(\partial_\mathrm{LL}E)\right]=\sum_{F\in\mathrm{Face}(E)\;:\;0\in F}\E\left[\mu_\EC(F)-\mu_\EC(\partial_\mathrm{LL}F)\right]
    \end{equation}
    where $\mathrm{Face}(E)$ denotes the set of closed faces of $E$.

    We now prove the claim by induction on the dimension $k$ of $E$. For $k=0$, we have $E=\{0\}$ so
    \begin{displaymath}
        \mu_\EC(E)-\mu_\EC(\partial_\mathrm{LL}E)=\mu_\EC(\{0\})=\P(0\in\{f\geq\ell\}_\infty)
    \end{displaymath}
    which matches the right-hand side of \eqref{e:EC_claim} (recall that zero-dimensional faces are considered critical points by definition). For general $k$, it is shown in the proof of \cite[Theorem~11.7.2]{at07} (see the penultimate displayed equation) that
    \begin{displaymath}
        \E[\mu_\EC(E)]=\sum_{j=0}^k\sum_{F\in \mathrm{Face}_j(E)\;:\;0\in F}\sum_{i=0}^j(-1)^{j-i}\E[N_{\mathrm{Crit},\infty}(\mathrm{int}F,i)]
    \end{displaymath}
    where $\mathrm{Face}_j(E)$ denotes the set of closed faces of $E$ of dimension $j$. Technically, the above formula is proven for the expected Euler characteristic of the excursion set $\{f\geq\ell\}$ but the proof holds verbatim if one replaces this by $\{f\geq\ell\}_\infty$ (it relies only on counting the number of stratified critical points of different types contained in the excursion set and stationarity of the field). Applying the strong inductive assumption up to dimension $k-1$ with this formula, we see that
    \begin{align*}
        \E[\mu_\EC(E)]&=\sum_{F\in \mathrm{Face}_k(E)\;:\;0\in F}\sum_{i=0}^k(-1)^{k-i}\E[N_{\mathrm{Crit},\infty}(\mathrm{int}F,i)]+\sum_{j=0}^{k-1}\sum_{F\in \mathrm{Face}_j(E)\;:\;0\in F}\E[\mu_\EC(F)-\mu_\EC(\partial_\mathrm{LL}F)]\\
        &=\sum_{F\in \mathrm{Face}_k(E)\;:\;0\in F}\sum_{i=0}^k(-1)^{k-i}\E[N_{\mathrm{Crit},\infty}(\mathrm{int}F,i)]+\E[\mu_\EC(\partial_\mathrm{LL}E)]
    \end{align*}
    where the second equality uses \eqref{e:EC_limiting_formula}. This establishes the result for dimension $k$, completing the inductive argument and so proving the claim.
    
    \underline{Limiting behaviour:} Next we show that $c_\EC(\ell)\to 0$ as $\ell\to-\infty$. With probability one, there exists a (random) $\ell_0$ such that for all $\ell<\ell_0$, $B_0\subset\{f\geq\ell\}_\infty$ and hence
    \begin{displaymath}
        \mu_\EC(B_0)=\EC[\{f\geq\ell\}_\infty\cap B_0]=\EC[B_0]=1
    \end{displaymath}
    and similarly $\mu_\EC(\partial_{\mathrm{LL}}B_0)=1$ (since $\partial_\mathrm{LL}B_0$ is contractible). Therefore
    \begin{displaymath}
        \mu_\EC(B_0),\mu_\EC(\partial_\mathrm{LL}(B_0))\to 1
    \end{displaymath}
    almost surely as $\ell\to-\infty$. Since these two variables are dominated by the number of stratified critical points of $f$ in $B$ (Lemma~\ref{l:ECBound}) which is integrable, using dominated convergence along with our previous application of the ergodic theorem we have
    \begin{displaymath}
        c_\EC(\ell)=\E[\mu_\EC(B_0)]-\E[\mu_\EC(\partial_\mathrm{LL}B_0)]\to 0\qquad\text{as }\ell\to\infty.
    \end{displaymath}
    \underline{Two-dimensional case:} Finally we fix $d=2$, $\ell<\ell_c$ and show that $c_\EC(\ell)<0$. It is well known that the Euler characteristic of a planar set can be decomposed as the number of components of the set minus the number of `holes' (i.e.\ the number of bounded components of its complement). Hence for any $n$
    \begin{equation}\label{e:EC_negative}
        \mu_\EC(\Lambda_n)=N_\mathrm{Comp}\left(\{f\geq\ell\}_\infty\cap \Lambda_n\right)-N_\mathrm{Holes}(\{f\geq\ell\}_\infty\cap\Lambda_n)
    \end{equation}
    where, for a set $A\subset\R^2$, $N_\mathrm{Comp}(A)$ denotes the number of components of $A$ and $N_\mathrm{Holes}(A)$ denotes the number of bounded components of $\R^2\setminus A$. Since there is only one unbounded excursion component, we would expect $\{f\geq\ell\}_\infty\cap\Lambda_n$ to consist of one large component along with some small components near $\partial\Lambda_n$. Hence the first term on the right hand side of \eqref{e:EC_negative} should by $o(n^2)$, as $n\to\infty$, with high probability. In contrast, by stationarity of $f$ we would expect the number of `holes' in the unbounded component (i.e.\ the second term of \eqref{e:EC_negative}) to be of order $n^2$. We make both of these intuitions rigorous below.
    
    Consider first the components of $\{f\geq\ell\}_\infty\cap\Lambda_n$, each of which is either contained in $\Lambda_n\setminus\Lambda_{n-\sqrt{n}}$ or intersects $\Lambda_{n-\sqrt{n}}$. In the former case, the point in the component which maximises $f$ must be a stratified critical point of $f$ inside $\Lambda_n\setminus\Lambda_{n-\sqrt{n}}$. Now suppose that there are two or more components of $\{f\geq\ell\}_\infty\cap\Lambda_n$ which intersect $\Lambda_{n-\sqrt{n}}$. There must exist a curve $P$ in $\{f<\ell\}$ which separates these components in $\Lambda_n$ (i.e.\ the start and end points of $P$ are in $\partial\Lambda_n$ and the two excursion components are in different components of $\Lambda_n\setminus P$). In particular $P$ must cross $\partial\Lambda_{n-\sqrt{n}}$, or in other words
    \begin{displaymath}
        E_n:=\left\{\partial\Lambda_{n-\sqrt{n}}\overset{\{f<\ell\}}{\longleftrightarrow}\partial\Lambda_n\right\}
    \end{displaymath}
    must occur. In this case, the number of components of $\{f\geq\ell\}_\infty\cap\Lambda_n$ is bounded by the number of stratified critical points of $f$ in $\Lambda_n$ (by the same reasoning as before). Combining these observations, we have
    \begin{displaymath}
        N_\mathrm{Comp}\left(\{f\geq\ell\}_\infty\cap \Lambda_n\right)\leq \overline{N}_\mathrm{Crit}( \Lambda_n\setminus\Lambda_{n-\sqrt{n}}) + 1 + \overline{N}_\mathrm{Crit}(\Lambda_n)\ind_{E_n}.
    \end{displaymath}
    By stationarity of $f$
    \begin{displaymath}
        \E[\overline{N}_\mathrm{Crit}( \Lambda_n\setminus\Lambda_{n-\sqrt{n}})]\leq cn^{3/2}
    \end{displaymath}
    for some constant $c>0$ depending only on the distribution of $f$. By the Cauchy-Schwarz inequality and Proposition~\ref{p:TMB}
    \begin{displaymath}
        \E[\overline{N}_\mathrm{Crit}(\Lambda_n)\ind_{E_n}]\leq cn^2\P(E_n)^{1/2}
    \end{displaymath}
    for some $c>0$ as before. By the first part of Theorem~\ref{t:PhaseTransition} and stationarity
    \begin{displaymath}
        \P(E_n)\leq \sum_{v\in\Z^d\;:\;d(v,\partial\Lambda_{n-\sqrt{n}})\leq 2}\P\left(B_v\overset{\{f\leq\ell\}}{\longleftrightarrow}v+\Lambda_{\sqrt{n}}\right)\leq C^\prime ne^{-cn}
    \end{displaymath}
    for some $C,C^\prime,c>0$. Combining the last four equations, we have
    \begin{equation}\label{e:EC_comps_bound}
        \E\left[N_\mathrm{Comp}\left(\{f\geq\ell\}_\infty\cap \Lambda_n\right)\right]=o(n^2)\qquad\text{as }n\to\infty.
    \end{equation}
    %as $n\to\infty$.

    We now consider the holes in the unbounded component. With probability one, the unbounded component $\{f\geq\ell\}_\infty$ exists and is not equal to $\R^2$. Therefore, since the level set $\{f=\ell\}$ consists of smooth curves almost surely, the unbounded component must intersect the boundary of some component of $\{f<\ell\}$. This latter component must be bounded (courtesy of the phase transition in Theorem~\ref{t:PhaseTransition}) and hence it is surrounded by $\{f\geq\ell\}_\infty$. In summary, with probability one $\R^2\setminus\{f\geq\ell\}_\infty$ has at least one bounded component. We claim that there exists $m\in\N$ such that
    \begin{equation}\label{e:EC_holes_positivity}
        \P\left(N_\mathrm{Holes}(\{f\geq\ell\}_\infty\cap\Lambda_m)>0\right)>0.
    \end{equation}
    If this were not the case, then by taking a countable union over $m\in\N$ we could show that $N_\mathrm{Holes}(\{f\geq\ell\}_\infty)=0$ with probability one. This would contradict our previous argument, and so we verify the claim. To complete the proof, we partition $\Lambda_n$ into boxes of side-length 2m and use stationarity. Specifically, for $n>m$ we have
    \begin{displaymath}
        N_\mathrm{Holes}(\{f\geq\ell\}_\infty\cap\Lambda_n)\geq\sum_{v\in\Z^d\;:\;2mv+\Lambda_m\subset\Lambda_n}N_\mathrm{Holes}\left(\{f\geq\ell\}_\infty\cap(2mv+\Lambda_m)\right)
    \end{displaymath}
    since the domains $2mv+\Lambda_m$ will not overlap for distinct $v$. Then by stationarity
    \begin{displaymath}
        \E[N_\mathrm{Holes}(\{f\geq\ell\}_\infty\cap\Lambda_n)]\geq \sum_{v\in\Z^d\;:\;2mv+\Lambda_m\subset\Lambda_n}\E[N_\mathrm{Holes}(\{f\geq\ell\}_\infty\cap\Lambda_m)]\geq c(n/m)^2
    \end{displaymath}
    for some $c>0$ by \eqref{e:EC_holes_positivity}. Combining this with \eqref{e:EC_negative} and \eqref{e:EC_comps_bound} we see that
    \begin{displaymath}
        c_\EC(\ell)=\lim_{n\to\infty}\frac{\E[\mu_\EC(\Lambda_n)]}{n^2}<0
    \end{displaymath}
    completing the proof of the theorem.
\end{proof}
\begin{remark}
    From the previous proofs we see that the convergence stated in Theorem~\ref{t:FirstOrder} actually occurs in $L^p$ for values of $p$ greater than one. Specifically, if $\star=\Vol$ then we may choose any $p<\infty$, if $\star=\SA$ then we may choose $p=\frac{k(k+1)}{2}-2$ and if $\star=\EC$ then we may choose $p=k-1$.
\end{remark}

As a consequence of Theorem~\ref{t:clt}, Corollary~\ref{c:FinitaryCLT} follows from straightforward arguments using the decay of connection probabilities.

\begin{proof}[Proof of Corollary~\ref{c:FinitaryCLT}]
    Taking $\epsilon>0$ as given, for $n\in\N$ we define
    \begin{displaymath}
        A_n=\Big\{\partial\Lambda_n\overset{\{f\geq\ell\}_{<\infty}}{\longleftrightarrow}\partial\Lambda_{(1+\epsilon)n}\Big\}.
    \end{displaymath}
    Then by stationarity of $f$ and Assumption~\ref{a:Decay}, for any $\gamma>0$ there exists $C_\gamma$ such that
    \begin{equation}\label{e:FinitaryProb}
        \P(A_n)\leq\sum_{x\in\Z^d\cap\Lambda_n}\P\Big(B_x\overset{\{f\geq\ell\}_{<\infty}}{\longleftrightarrow} x+\partial\Lambda_{(\epsilon/2)n}\Big)\leq C_\gamma n^{-\gamma+d}
    \end{equation}
    for all $n\in\N$. We also define
    \begin{displaymath}
        \mu_\star^+(\Lambda_n)=\begin{cases}
            (2n)^d &\text{if }\star=\Vol\\
            c_d\overline{N}_\mathrm{Crit}(\Lambda_n,f) &\text{if }\star=\EC\\
            \mathcal{H}^{d-1}[\{f=\ell\}\cap\Lambda_n] &\text{if }\star=\SA
        \end{cases}
    \end{displaymath}
    where $c_d>0$ is taken from the statement of Lemma~\ref{l:ECBound}. Since $f$ is stationary, there exists $c>0$ such that $\E[(\mu_\star^+(\Lambda_n))^3]\leq cn^{3d}$. Recalling that $\mu_\star(\Lambda_n,\epsilon)$ denotes the functional $\mu_\star$ applied to components in $\Lambda_n$ which intersect $\partial\Lambda_{(1+\epsilon)n}$, we have
    \begin{displaymath}
        \lvert \mu_\star(\Lambda_n,\epsilon)-\mu_\star(\Lambda_n)\rvert\leq\mu_\star^+(\Lambda_n)\ind_{A_n}
    \end{displaymath}
    (where the case $\star=\EC$ uses Lemma~\ref{l:ECBound}). The desired result will follow from Theorem~\ref{t:clt} if we can show that the right hand side of this expression tends to zero in $L^2$. By H\"older's inequality and \eqref{e:FinitaryProb}
    \begin{displaymath}
        \E\big[\mu_\star^+(\Lambda_n)^2\ind_{A_n}\big]\leq\E[\mu_\star^+(\Lambda_n)^3]^{2/3}\P(A_n)^{1/3}\leq C_\gamma n^{2d+(d-\gamma)/3}\to 0
    \end{displaymath}
    provided that we choose $\gamma$ sufficiently large. This yields the required convergence.
\end{proof}

\section{Positivity of limiting variance}\label{s:positivity}
In this section we prove that the limiting variance in our CLT is strictly positive. Before proceeding, the reader may wish to revisit the last part of Section~\ref{ss:Outline} which gives an overview of our argument.

Our first step is to generalise the stabilisation property which was used as an input to our general CLT (Theorem~\ref{t:GenCLT}).
\begin{lemma}\label{l:StabGeneral}
    Let $f$ satisfy the conditions of Theorem~\ref{t:posvar} and let $\rho=cq\ast\ind_{2mB_0}$ for some $c\in\R$ and $m\in\N$. For any sequence of cubes $D_n:=v_n+\Lambda_n$ such that $\liminf_n D_n=\R^d$ let
    \begin{displaymath}
        \Delta_0^\star(D_n,\rho)=\mu_\star(D_n, f)-\mu_\star(D_n,\widetilde{f}_0+\rho).
    \end{displaymath}
    Then for each $\star\in\{\Vol,\EC,\SA\}$ and $\ell<-\ell_c$ there exists a random variable $\Delta_0^\star(\rho)$ such that
    \begin{displaymath}
        \Delta_0^\star(D_n,\rho)\to\Delta_0^\star(\rho)
    \end{displaymath}
    almost surely as $n\to\infty$. Furthermore
    \begin{equation}\label{e:StabDCT}
        \E[\Delta_0^\star(\rho)]=\lim_{n\to\infty}\E[\mu_\star(D_n,f)]-\E[\mu_\star(D_n,f+\rho)]
    \end{equation}
    and for any $v\in\Z^d$,
    \begin{equation}\label{e:stab_translation}
        \E[\Delta_0^\star(\rho(\cdot))]=\E[\Delta_0^\star(\rho(v+\cdot))].
    \end{equation}
\end{lemma}
\begin{proof}
First observe that by Assumption~\ref{a:clt}, for all $x\in\R^d$
\begin{equation}\label{e:RhoDecay}
    \sup_{\lvert\alpha\rvert\leq 2}\lvert\partial^\alpha \rho(x)\rvert\leq C(1+\lvert x\rvert)^{-\beta},
\end{equation}
in other words, $\rho$ satisfies the conditions in Lemma~\ref{l:ProbExtUnst}. Hence, when $\star=\EC$ this result (aside from \eqref{e:stab_translation}) follows from Lemma~\ref{l:StabilisationEC}. We therefore consider only $\star\in\{\Vol,\SA\}$ (until we verify \eqref{e:stab_translation} below).

We define
\begin{displaymath}
    Y^\Vol(w)=\Vol[B_w\cap\{\lvert f-\ell\rvert\leq\lvert p_0+\rho\rvert\}]+\ind_{w\in\mathcal{U}_\mathrm{FR}(f,p_0+\rho)}
\end{displaymath}
and $Y^\SA(w)$ to be the variable $Y(w)$ specified in Lemma~\ref{l:sa_perturbation_moments}.
    Then by Lemma~\ref{l:ChangeDecomp} (when $\star=\Vol$) and Lemma~\ref{l:sa_perturbation_moments} (when $\star=\SA$) we have
    \begin{equation}\label{e:StabGen1}
        \lvert\Delta_0^\star(B_w,\rho)\rvert\leq Y^\star(w)
    \end{equation}
    for any $w\in\Z^d$. By Lemma~\ref{l:ProbExtUnst} and Lemma~\ref{l:VolMoments}, for any $\delta>0$ there exists $c>0$ such that
    \begin{equation}\label{e:StabGen2}
        \E\big[Y^\Vol(w)\big]\leq c(1+\lvert w\rvert)^{-\beta+\delta}.
    \end{equation}
    By Lemma~\ref{l:sa_perturbation_moments} (specifically using \eqref{e:sa_perturb_moments2} and the fact that $\rho$ is deterministic) and Lemma~\ref{l:ProbExtUnst}
    \begin{equation}\label{e:StabGen3}
    \begin{aligned}
        \E[Y^\SA(w)]&\leq C_\delta\P\left(w\in\mathcal{U}_\mathrm{FR}(f,p_0+\rho)\right)^\frac{k^\prime-1}{k^\prime}+C_\delta\left(\sqrt{M_2(w)}+\|\rho\|_{C^1(B_w)}\right)^{\frac{k^\prime-1}{3k^\prime-1}-\delta}\\
        &\leq C^\prime_\delta (1+\lvert w\rvert)^{-(\beta-\delta)\frac{k^\prime-1}{k^\prime}}+C^\prime_\delta(1+\lvert w\rvert)^{-\beta\left(\frac{k^\prime}{3k^\prime-1}-\delta\right)}.
    \end{aligned}
    \end{equation}
    Recalling that $\beta_\SA(k)=3d\frac{3k^\prime-1}{k^\prime-2}$ and taking $\delta>0$ sufficiently small in \eqref{e:StabGen2}-\eqref{e:StabGen3} we conclude that
    \begin{displaymath}
        \E[Y^\star(w)]\leq c(1+\lvert w\rvert)^{-3d}
    \end{displaymath}
    for $\star\in\{\Vol,\SA\}$ and $w\in\Z^d$. Hence by the monotone convergence theorem
    \begin{equation}\label{e:StabGen5}
        \E\left[\sum_{w\in\Z^d}Y^\star(w)\right]\leq\sum_{w\in\Z^d}c(1+\lvert w\rvert)^{-3d}<\infty
    \end{equation}
    and in particular $\sum_{w\in\Z^d}Y^\star(w)<\infty$ almost surely.
    
    Now given $m\in\N$, let $N\in\N$ be large enough that $\cap_{n\geq N}D_n\supseteq\Lambda_{m+1}$ and let $n_1,n_2>N$. By Lemma~\ref{l:ChangeDecomp} (when $\star=\Vol$) and Lemma~\ref{l:additivity_SA} (when $\star=\SA$)
    \begin{align*}
        \lvert\Delta_0^\star(D_{n_1},\rho)-\Delta_0^\star(D_{n_2},\rho)\rvert\leq\lvert\Delta_0^\star(D_{n_1}\setminus D_{n_2},\rho)\rvert+\lvert\Delta_0^\star(D_{n_2}\setminus D_{n_1},\rho)\rvert\leq \sum_{w\in\Z^d\setminus\Lambda_m}Y^\star(w).
    \end{align*}
    The right hand term converges to zero as $m\to\infty$, and so we conclude that $\Delta_0^\star(D_n,\rho)$ is almost surely Cauchy and hence convergent to some limit $\Delta_0^\star(\rho)$. To see that the limit does not depend on the choice of domains $D_n$, we may take any other sequence $D_n^\prime$ such that $\liminf_nD_n^\prime=\R^d$ and consider $(D_1,D_1^\prime,D_2,D_2^\prime,\dots)$. By our above argument we see that $\Delta_0^\star(D_n^\prime,\rho)$ and $\Delta_0^\star(D_n,\rho)$ converge almost surely to some limit, but the latter convergence implies that this limit must equal $\Delta_0^\star(\rho)$ almost surely.

    Using Lemma~\ref{l:ChangeDecomp} (when $\star=\Vol$) and Lemma~\ref{l:additivity_SA} (when $\star=\SA$) along with \eqref{e:StabGen1}, we see that the sequence $\Delta_0^\star(D_n,\rho)$ is uniformly bounded in absolute value by
    \begin{displaymath}
        \sum_{w\in\Z^d}Y^\star(w)
    \end{displaymath}
    which, by \eqref{e:StabGen5}, is integrable. Therefore by the dominated convergence theorem and the fact that $f\overset{d}{=}\widetilde{f}_0$
    \begin{align*}
        \E[\Delta_0^\star(\rho)]=\lim_{n\to\infty}\E\left[\mu_\star(D_n,f)-\mu_\star\big(D_n,\widetilde{f}_0+\rho\big)\right]=\lim_{n\to\infty}\E[\mu_\star(D_n,f)-\mu_\star(D_n,f+\rho)]
    \end{align*}
    and so \eqref{e:StabDCT} holds.
    
    Finally, letting $\star\in\{\Vol,\EC,\SA\}$, we observe that since $f$ is stationary,
    \begin{displaymath}
        \mu_\star(f+\rho(v+\cdot),\Lambda_n)\overset{d}{=}\mu_\star(f+\rho(\cdot),v+\Lambda_n).
    \end{displaymath}
    Therefore applying \eqref{e:StabDCT} for the sequence of domains $(\Lambda_1,v+\Lambda_1,\Lambda_2,v+\Lambda_2,\dots)$ proves that $\E[\Delta_0^\star(\rho(\cdot))]=\E[\Delta_0^\star(\rho(v+\cdot))]$.
\end{proof}

We next relate positivity of the limiting variance to the first order behaviour of our functional when $f$ is perturbed by $q\ast\ind_{\Lambda_m}$.

\begin{lemma}\label{l:positivity_sufficient}
    % Let the assumptions of Theorem~\ref{t:posvar} hold and suppose that there exists $m\in\N$ and a set $I\subset\R$ of positive measure such that for all $s\in I$
    % \begin{equation}\label{e:PositivitySufficient}
    %     \lim_{n\to\infty}\E[\mu_\star(\Lambda_n,f+sq\ast\ind_{\Lambda_m})]-\E[\mu_\star(\Lambda_n,f)]\neq 0,
    % \end{equation}
    % then $\sigma_\star(\ell)>0$.
    Let the assumptions of Theorem~\ref{t:posvar} hold and suppose that there exists $m\in\N$ such that
    \begin{equation}\label{e:PositivitySufficient}
        \lim_{n\to\infty}\Big(\E[\mu_\star(\Lambda_n,f+sq\ast\ind_{\Lambda_m})]-\E[\mu_\star(\Lambda_n,f)]\Big)\neq 0,
    \end{equation}
    for all $s$ in some set $I\subset\R$ of positive measure, then $\sigma_\star(\ell)>0$.
\end{lemma}
\begin{proof}
    We begin by applying our CLT to a rescaled version of $f$, which will allow us to express the limiting variance $\sigma^2_\star(\ell)$ in terms of resampling on a larger cube. Given $g\in L^2(\R^d)$ and $m$ as in the statement of the lemma, we define $\mathcal{W}(g)=(2m)^{-d/2}W(g(\cdot/2m))$, so that $\mathcal{W}$ is a standard Gaussian white noise on $\R^d$. We then define $f^{(m)}=q^{(m)}\ast \mathcal{W}$ where $q^{(m)}(x)=(2m)^{d/2}q(2mx)$. From this definition, we have $f^{(m)}(x)=f(2mx)$ for any $x\in\R^d$. Finally we define
    \begin{displaymath}
        \mathcal{F}_0^{(m)}=\sigma(\mathcal{W}_u\;|\;u\preceq 0)=\sigma(W|_{2mB_u}\;|\;u\preceq 0)\qquad\text{and}\qquad \widetilde{f}^{(m)}_0=q^{(m)}\ast\widetilde{\mathcal{W}}^{(0)}
    \end{displaymath}
    % and
    % \begin{displaymath}
    %     \widetilde{f}^{(m)}_0=q^{(m)}\ast\widetilde{\mathcal{W}}^{(0)}
    % \end{displaymath}
    where $\widetilde{\mathcal{W}}^{(0)}$ resamples $\mathcal{W}_0$ independently. Since $q^{(m)}$ satisfies the conditions of Theorem~\ref{t:clt}, we see that
    \begin{displaymath}
        \frac{\Var[\mu_\star(\Lambda_n,f^{(m)})]}{(2n)^d}\to\sigma^2_{\star,m}(\ell):=\E\left[\E\big[\Delta_0^{\star,m}\;|\;\mathcal{F}^{(m)}_0\big]^2\right]
    \end{displaymath}
    where
    \begin{equation}
        \Delta_0^{\star,m}:=\lim_{n\to\infty}\mu_\star(\Lambda_n,f^{(m)})-\mu_\star(\Lambda_n,\widetilde{f}^{(m)}_0).
    \end{equation}
    Since $f^{(m)}(x)=f(2mx)$, by definition of $\mu_\star$ we have
    \begin{equation}\label{e:Rescaling}
    \begin{aligned}
        \mu_\Vol(\Lambda_n,f^{(m)})&=(2m)^{-d}\mu_\Vol(\Lambda_{2mn},f),\\\mu_\SA(\Lambda_n,f^{(m)})&=(2m)^{-d+1}\mu_\SA(\Lambda_{2mn},f),\\
        \mu_\EC(\Lambda_n,f^{(m)})&=\mu_\EC(\Lambda_{2mn},f).
    \end{aligned}
    \end{equation}
    Therefore it is enough for us to verify positivity of $\sigma_{\star,m}(\ell)$.

    Now we let $Z_0:=\mathcal{W}(B_0)$. Since this variable is $\mathcal{F}_0^{(m)}$-measurable, by the conditional Jensen inequality and the tower property of conditional expectation
    \begin{displaymath}
        \sigma^2_{\star,m}(\ell)\geq \E\left[\E\big[\Delta_0^{\star,m}\;|\;Z_0\big]^2\right].
    \end{displaymath}
    Positivity of $\sigma_{\star,m}^2(\ell)$ will then follow if we can show that $\E\big[\Delta_0^{\star,m}\;|\;Z_0\big]$ is non-zero with positive probability. By a standard property of conditional expectation,
    \begin{displaymath}
        \E\big[\Delta_0^{\star,m}\;|\;Z_0\big]=F(Z_0)
    \end{displaymath}
    for some measurable function $F$. For a given $s\in\R$, the random variables $Z_0$ and $Z_0+s$ are mutually absolutely continuous. Combining these facts, it is enough to show that 
    \begin{equation}\label{e:PositivitySufficient2}
        \E[F(Z_0+s)]\neq 0\quad\text{for all }s\in J
    \end{equation}
    where $J\subset\R$ has positive Lebesgue measure.

    We now find an explicit expression for this expectation. Let $\mathcal{W}^{-}$ denote the part of the white noise $\mathcal{W}$ which is orthogonal to $Z_0$ (explicitly we can set $\mathcal{W}^-=\mathcal{W}-Z_0\ind_{B_0}$ viewed as a generalised function). We also define $\mathcal{W}^\prime$ to be the white noise used for resampling $\mathcal{W}$ (i.e.\ this is the analogue of $W^\prime$ defined in Section~\ref{ss:Outline}). Since $\Delta_0^{\star,m}$ is defined in terms of $\mathcal{W}$ and $\mathcal{W}^\prime$, we can write
    \begin{displaymath}
        \Delta_0^{\star,m}=G(Z_0,\mathcal{W}^-,\mathcal{W}^\prime)
    \end{displaymath}
    for some measurable function $G$. Since $Z_0$ is independent of the other two arguments, we have 
    \begin{displaymath}
        F(t)=\E[G(t,\mathcal{W}^-,\mathcal{W}^\prime)]
    \end{displaymath}
    for almost all $t\in\R$. Hence, using independence of these arguments once more, for almost all $s\in I$
    \begin{align*}
        \E[F(Z_0+s)]&=\E[G(Z_0+s,\mathcal{W}^-,\mathcal{W}^\prime)]=\E\Big[\lim_{n\to\infty}\mu_\star(\Lambda_n,f^{(m)}+sq^{(m)}\ast\ind_{B_0})-\mu_\star(\Lambda_n,\widetilde{f}_0^{(m)})\Big]
    \end{align*}
    where the second equality follows from the definition of $\Delta_0^{\star,m}$. By \eqref{e:Rescaling} the latter expression is equal to a constant (depending only on $\star$) times
    \begin{displaymath}
        \E\Big[\lim_{n\to\infty}\mu_\star(\Lambda_{2mn},f+s(2m)^{-d/2}q\ast\ind_{2mB_0})-\mu_\star(\Lambda_{2mn},\widetilde{f}_0)\Big].
    \end{displaymath}
    We now define $\rho=s(2m)^{-d/2}q\ast\ind_{2mB_0}$ and let $v=(m,\dots,m)\in\Z^d$. Since $(f,\widetilde{f}_0)\overset{d}{=}(\widetilde{f}_0,f)$ we can apply Lemma~\ref{l:StabGeneral} with the roles of $f$ and $\widetilde{f}_0$ reversed, to see that
    \begin{align*}
        \E\Big[\lim_{n\to\infty}\mu_\star(\Lambda_n,f+\rho)-\mu_\star(\Lambda_n,\widetilde{f}_0)\Big]=-\E[\Delta_0^\star&(\rho)]=-\E[\Delta_0^\star(\rho(v+\cdot))]\\
        &=\lim_{n\to\infty}\E\Big[\mu_\star(\Lambda_{n},f+s(2m)^{-d/2}q\ast\ind_{\Lambda_m})]-\E[\mu_\star(\Lambda_n,f)\Big].
    \end{align*}
    By the statement of the lemma, this limit is non-zero for all $s(2m)^{-d/2}$ in some set $I$ of positive measure, which verifies \eqref{e:PositivitySufficient2} and hence positivity of $\sigma_\star(\ell)$, as required.
\end{proof}

Theorem~\ref{t:posvar} will follow if we can verify condition \eqref{e:PositivitySufficient}. To do so we must decompose the effects of the perturbation $sq\ast\ind_{\Lambda_m}$ on three separate domains: inside $\Lambda_m$, near the boundary of $\Lambda_m$ and far from $\Lambda_m$ respectively. The following two lemmas will help us in this endeavour. They use techniques very similar to those we have employed earlier (in the proofs of Lemma~\ref{l:ProbExtUnst} and Theorem~\ref{t:clt} when $\star=\SA$) but adjusted to the current setting of a deterministic perturbation.

We denote $\overline{q}=q/(\int_{\R^d} q(x)\;dx)$ so that $\int_{\R^d}\overline{q}=1$. Also given $m\in\N$, we define
\begin{equation}\label{e:positivity_domains}
    m^+=m+m^{1-\lambda},\qquad m^-=m-m^{1-\lambda},\qquad m^{--}=m-2m^{1-\lambda}
\end{equation}
for a small parameter $\lambda>0$ to be determined below.

\begin{lemma}\label{l:FR_general}
Let $f$ and $\ell$ satisfy the conditions of Theorem~\ref{t:posvar}. Given an absolute constant $C>0$, $s\in[0,C]$ and $m\in\N$ let $\rho=s\overline{q}\ast\ind_{\Lambda_m}$. Then for any $\delta>0$, taking $\lambda>0$ sufficiently small in \eqref{e:positivity_domains} there exists $c_\delta>0$ such that
    \begin{enumerate}
        \item for $v\in\Z^d\cap\Lambda_{m^{--}}$
        \begin{displaymath}
            \P(v\in\mathcal{U}_\mathrm{FR}(f,\rho-s,\ell-s))\leq c_\delta m^{-\beta+2d+\delta}.
        \end{displaymath}
        \item for $v\in\Z^d\setminus\Lambda_{m^{+}}$
        \begin{displaymath}
            \P(v\in\mathcal{U}_\mathrm{FR}(f,\rho,\ell))\leq c_\delta \dist(v,\Lambda_m)^{-\beta+2d+\delta}
        \end{displaymath}
    \end{enumerate}
where $\mathcal{U}_\mathrm{FR}(g,p,a)$ denotes the finite-range unstable set at level $a$ defined in \eqref{e:finite_range_unstable}.
\end{lemma}
\begin{proof}
    Proof of \textit{(1)}: First observe that $\rho-s=s(\overline{q}\ast\ind_{\Lambda_m}-1)=s(\overline{q}\ast(\ind_{\Lambda_m}-\ind_{\R^d}))$ by definition of $\overline{q}$. Hence for $v\in\Z^d\cap\Lambda_{m^-}$
    \begin{equation}\label{e:FR_general1}
    \mathtoolsset{multlined-width=0.9\displaywidth}
    \begin{multlined}
        \|\rho-s\|_{C^1(v+\Lambda_2)}=\sup_{\lvert\alpha\rvert\leq 1}\sup_{x\in v+\Lambda_2} s \left\lvert\int_{\R^d\setminus\Lambda_m}\partial^\alpha \overline{q}(x-y)\;dy\right\rvert\\
        \leq\sup_{\lvert\alpha\rvert\leq 1} C\int_{\R^d\setminus\Lambda_{m^{1-\lambda}-2\sqrt{d}}}\lvert\partial^\alpha \overline{q}(y)\rvert\;dy
        \leq C^\prime m^{-(\beta-d)(1-\lambda)}
    \end{multlined}
    \end{equation}
    where we have used the third point of Assumption~\ref{a:clt} to differentiate inside the integral and also for the final inequality. Now let $v\in\Z^d\cap\Lambda_{m^{--}}$, then by definition of $\mathcal{U}_\mathrm{FR}$,
    \begin{equation}\label{e:FR_general2}
    %\mathtoolsset{multlined-width=0.9\displaywidth}
    % \begin{multlined}
    %     \P(v\in\mathcal{U}_\mathrm{FR}(f,\rho-s,\ell-s))\leq\\
    % \begin{aligned}
    %     &\P(v\in\mathcal{U}_\mathrm{Loc}(f,\rho-s,\ell-s))\\
    %     &+\sum_{w\in\Z^d}\P\left(\left\{w\in\mathcal{U}_\mathrm{Loc}(f,\rho-s,\ell-s)\right\}\cap\left\{v\overset{\{f\geq\ell-s\}_{<\infty}}{\longleftrightarrow}w\right\}\right)\\
    %     &+\sum_{w\in\Z^d}\P\left(\left\{w\in\mathcal{U}_\mathrm{Loc}(f,\rho-s,\ell-s)\right\}\cap\left\{v\overset{\{f+\rho\geq\ell\}_{<\infty}}{\longleftrightarrow}w\right\}\right).
    % \end{aligned}
    % \end{multlined}
    \begin{aligned}
        \P(v\in\mathcal{U}_\mathrm{FR}(f,\rho-s,\ell-s))&\leq\P(v\in\mathcal{U}_\mathrm{Loc}(f,\rho-s,\ell-s))\\
        &\quad+\sum_{w\in\Z^d}\P\left(\left\{w\in\mathcal{U}_\mathrm{Loc}(f,\rho-s,\ell-s)\right\}\cap\left\{v\overset{\{f\geq\ell-s\}_{<\infty}}{\longleftrightarrow}w\right\}\right)\\
        &\quad+\sum_{w\in\Z^d}\P\left(\left\{w\in\mathcal{U}_\mathrm{Loc}(f,\rho-s,\ell-s)\right\}\cap\left\{v\overset{\{f+\rho\geq\ell\}_{<\infty}}{\longleftrightarrow}w\right\}\right).
    \end{aligned}
    \end{equation}
    By Lemma~\ref{l:Prob_unstable_decay} and \eqref{e:FR_general1}, for any $\delta>0$
    \begin{displaymath}
        \P(v\in\mathcal{U}_\mathrm{Loc}(f,\rho-s,\ell-s))\leq c_\delta \|\rho-s\|_{C^1(v+\Lambda_2)}^{1-\delta}\leq c_\delta^\prime m^{-(\beta-d)(1-\lambda)(1-\delta)}.
    \end{displaymath}
    Taking $\lambda$ and $\delta$ sufficiently small, the exponent above is less than $-\beta+2d$ and so we need only consider the second and third terms on the right-hand side of \eqref{e:FR_general2}.

    By Assumption~\ref{a:Decay} and Lemma~\ref{l:Prob_unstable_decay}, for any $\gamma>d$ the second term on the right-hand side of \eqref{e:FR_general2} is at most
    \begin{multline*}
        \sum_{w\in\Z^d\setminus\Lambda_{m^-}}\P\left(v\overset{\{f\geq\ell-s\}_{<\infty}}{\longleftrightarrow}w\right)+\sum_{w\in\Z^d\cap\Lambda_{m^-}}\P\left(w\in\mathcal{U}_\mathrm{Loc}(f,\rho-s,\ell-s)\right)\\
        \leq \sum_{w\in\Z^d\setminus\Lambda_{m^-}}c_\gamma \lvert v-w\rvert^{-\gamma}+\sum_{w\in\Z^d\cap\Lambda_{m^-}}c_\delta\|\rho-s\|_{C^1(w+\Lambda_2)}^{1-\delta}\\
        \leq c_\gamma^\prime m^{-\gamma(1-\lambda)+d}+c_\delta^\prime m^dm^{-(\beta-d)(1-\lambda)(1-\delta)}
    \end{multline*}
    where the final inequality uses Lemma~\ref{l:ElementarySum}, the fact that $\dist(v,\Lambda_{m^-})\geq m^{1-\lambda}$ and \eqref{e:FR_general1}. For $\lambda,$ $\delta$ sufficiently small and $\gamma$ sufficiently large, this expression is at most $c_{\delta} m^{-\beta+2d+\delta}$ so we see that the second term of \eqref{e:FR_general2} also satisfies the necessary bound. A near-identical argument shows that the third term of \eqref{e:FR_general2} satisfies the same bound; the only difference in the argument is that Proposition~\ref{p:UnifTruncConnecDecay} is used in place of Assumption~\ref{a:Decay}. This completes the proof of the first part of the lemma.

    Proof of \textit{(2)}: Using the decay of $\overline{q}$ specified in Assumption~\ref{a:clt}, for $v\in\Z^d\setminus\Lambda_{m^+}$
    \begin{equation}\label{e:FR_general3}
    \mathtoolsset{multlined-width=0.9\displaywidth}
        \begin{multlined}
        \|\rho\|_{C^1(v+\Lambda_2)}=\sup_{\lvert\alpha\rvert\leq 1}\sup_{x\in v+\Lambda_2}s \left\lvert\int_{\Lambda_m}\partial^\alpha \overline{q}(x-y)\;dy\right\rvert\\
        \leq\sup_{\lvert\alpha\rvert\leq 1} C\int_{\lvert y\rvert\geq \dist(v,\Lambda_m)-2\sqrt{d}}\lvert\partial^\alpha \overline{q}(y)\rvert\;dy\leq C^\prime \dist(v,\Lambda_m)^{-\beta+d}.
        \end{multlined}
    \end{equation}
    Once again, by definition of $\mathcal{U}_\mathrm{FR}$
    \begin{equation}\label{e:FR_general4}
    %\mathtoolsset{multlined-width=0.9\displaywidth}
    \begin{aligned}
        \P(v\in\mathcal{U}_\mathrm{FR}(f,\rho,\ell))\leq\P(v\in\mathcal{U}_\mathrm{Loc}(f,\rho,\ell))
        &+\sum_{w\in\Z^d}\P\left(\left\{w\in\mathcal{U}_\mathrm{Loc}(f,\rho,\ell)\right\}\cap\left\{v\overset{\{f\geq\ell\}_{<\infty}}{\longleftrightarrow}w\right\}\right)\\
        &+\sum_{w\in\Z^d}\P\left(\left\{w\in\mathcal{U}_\mathrm{Loc}(f,\rho,\ell)\right\}\cap\left\{v\overset{\{f+\rho\geq\ell\}_{<\infty}}{\longleftrightarrow}w\right\}\right).
    \end{aligned}
    \end{equation}
    For $v\in\Z^d\setminus\Lambda_{m^+}$, by Lemma~\ref{l:Prob_unstable_decay} and \eqref{e:FR_general3}
    \begin{equation*}
        \P(v\in\mathcal{U}_\mathrm{Loc}(f,\rho,\ell))\leq c_\delta\|\rho\|_{C^1(w+\Lambda_2)}^{1-\delta}\leq c_\delta^\prime \dist(v,\Lambda_m)^{-(\beta-d)(1-\delta)}.
    \end{equation*}
    Hence the first term on the right-hand side of \eqref{e:FR_general4} satisfies the required bound (for $\delta$ sufficiently small).

    Using Assumption~\ref{a:Decay} and Lemma~\ref{l:Prob_unstable_decay} once more, the second term on the right-hand side of \eqref{e:FR_general4} is bounded above by
    \begin{multline*}
        \sum_{w:\lvert w-v\rvert\geq \dist(v,\Lambda_m)/2}\P\left(v\overset{\{f\geq\ell\}_{<\infty}}{\longleftrightarrow}w\right)+\sum_{w:\lvert w-v\rvert< \dist(v,\Lambda_m)/2}\P\left(w\in\mathcal{U}_\mathrm{Loc}(f,\rho,\ell)\right)\\
        \leq \sum_{w:\lvert w-v\rvert\geq \dist(v,\Lambda_m)/2}c_\gamma \lvert v-w\rvert^{-\gamma}+\sum_{w:\lvert w-v\rvert<\dist(v,\Lambda_m)/2}c_\delta\|\rho\|_{C^1(w+\Lambda_2)}^{1-\delta}\\
        \leq c_\gamma^\prime \dist(v,\Lambda_m)^{-\gamma+d}+c_\delta^\prime \dist(v,\Lambda_m)^d\dist(v,\Lambda_m)^{-(\beta-d)(1-\delta)}
    \end{multline*}
    for any $\gamma>d$, where the final inequality uses \eqref{e:FR_general3}. Taking $\delta>0$ sufficiently small and $\gamma>0$ sufficiently large, the final term above is less than $c_{\delta} \dist(v,\Lambda_m)^{-\beta+2d+\delta}$, as required. Once again, we can repeat this argument for the third term on the right-hand side of \eqref{e:FR_general4} (using Lemma~\ref{p:UnifTruncConnecDecay} instead of Assumption~\ref{a:Decay}) and deduce the same upper bound. This completes the proof of the second part of the lemma.
\end{proof}

\begin{lemma}\label{l:sa_general_perturbation_bounds}
Let $f$ and $\ell$ satisfy the conditions of Theorem~\ref{t:posvar} for $\star=\SA$, then for $\lambda>0$ chosen sufficiently small
\begin{equation}\label{e:sa_general_perturbation_a}
    m^{-d}\left\lvert\E[\mu_\SA(\Lambda_{m^{--}},f+s,\ell)-\mu_\SA(\Lambda_{m^{--}},f+\rho,\ell)]\right\rvert\to 0
\end{equation}
and
\begin{equation}\label{e:sa_general_perturbation_b}
    \sum_{v\in\Z^d\setminus\Lambda_{m^+}}\lvert\E[\mu_\SA(B_v,f,\ell)-\mu_\SA(B_v,f+\rho,\ell)]\rvert]\to 0
\end{equation}
as $m\to\infty$.
\end{lemma}

\begin{proof}
By additivity of $\mu_\SA$ (i.e.\ Lemma~\ref{l:additivity_SA})
\begin{equation}\label{e:sa_general_perturbation_1}
    \begin{aligned}
    &\left\lvert\E[\mu_\SA(\Lambda_{m^{--}},f+s,\ell)- \mu_\SA(\Lambda_{m^{--}},f+\rho,\ell)]\right\rvert\\
    &\qquad\qquad\qquad\qquad\qquad\leq \sum_{v\in\Z^d\cap\Lambda_{m^{--}}}\lvert\E[\mu_\SA(B_v,f+\rho,\ell)-\mu_\SA(B_v,f+s,\ell)]\rvert\\
    &\qquad\qquad\qquad\qquad\qquad\leq (2m)^d\sup_{v\in\Z^d\cap\Lambda_{m^{--}}}\lvert\E[\mu_\SA(B_v,f+\rho,\ell)-\mu_\SA(B_v,f+s,\ell)]\rvert.
    \end{aligned}
\end{equation}
Let $v\in\Z^d\cap\Lambda_{m^{--}}$, then applying Lemma~\ref{l:sa_perturbation_moments} with $p=\rho-s$, for any $\delta>0$ we have
    \begin{equation}\label{e:sa_general_perturbation_2}
    % \begin{aligned}
    %     &\lvert\E[\mu_\SA(B_v,f+\rho,\ell)-\mu_\SA(B_v,f+s,\ell)]\rvert\\
    %     &\qquad\qquad\qquad\qquad\qquad=\lvert\E[\mu_\SA(B_v,f+\rho-s,\ell-s)-\mu_\SA(B_v,f,\ell-s)]\rvert\\
    %     &\qquad\qquad\qquad\qquad\qquad\leq C_\delta\P(v\in\mathcal{U}_{FR}(f,\rho-s,\ell-s))^{\frac{k^\prime-1}{k^\prime}}+C_\delta\|\rho-s\|_{C^1(B_v)}^{\frac{k^\prime-1}{3k^\prime-1}-\delta}\\
    %     &\qquad\qquad\qquad\qquad\qquad\leq C_\delta\left(m^{-(\beta-2d-\delta)\frac{k^\prime-1}{k^\prime}}+m^{-(\beta-d)(1-\lambda)\left(\frac{k^\prime-1}{3k^\prime-1}-\delta\right)}\right)
    % \end{aligned}
    \begin{aligned}
        \lvert\E[\mu_\SA(B_v,f+\rho,\ell)-\mu_\SA(B_v,f+s,\ell)]\rvert&=\lvert\E[\mu_\SA(B_v,f+\rho-s,\ell-s)-\mu_\SA(B_v,f,\ell-s)]\rvert\\
        &\leq C_\delta\P(v\in\mathcal{U}_{FR}(f,\rho-s,\ell-s))^{\frac{k^\prime-1}{k^\prime}}+C_\delta\|\rho-s\|_{C^1(B_v)}^{\frac{k^\prime-1}{3k^\prime-1}-\delta}\\
        &\leq C_\delta\left(m^{-(\beta-2d-\delta)\frac{k^\prime-1}{k^\prime}}+m^{-(\beta-d)(1-\lambda)\left(\frac{k^\prime-1}{3k^\prime-1}-\delta\right)}\right)
    \end{aligned}
    \end{equation}
where the final inequality uses Lemma~\ref{l:FR_general} along with \eqref{e:FR_general1}. Since $\beta>\beta_\SA(k)>3d$, the right-hand side of \eqref{e:sa_general_perturbation_2} decays uniformly in $v$ (provided that $\delta<1$). Combining this with \eqref{e:sa_general_perturbation_1} proves \eqref{e:sa_general_perturbation_a}.

If $v\in\Z^d\setminus\Lambda_{m^+}$ and $\delta>0$, then an application of Lemma~\ref{l:sa_perturbation_moments} with $p=\rho$ yields
\begin{equation}\label{e:sa_general_perturbation_3}
    % \begin{aligned}
    %     &\lvert\E[\mu_\SA(B_v,f+\rho,\ell)-\mu_\SA(B_v,f,\ell)]\rvert\\
    %     &\qquad\qquad\qquad\leq C_\delta\P(v\in\mathcal{U}_{FR}(f,\rho))^{\frac{k^\prime-1}{k^\prime}}+C_\delta\|\rho\|_{C^1(B_v)}^{\frac{k^\prime-1}{3k^\prime-1}-\delta}\\
    %     &\qquad\qquad\qquad\leq C_\delta\left(\dist(v,\Lambda_m)^{-(\beta-2d-\delta)\frac{k^\prime-1}{k^\prime}}+\dist(v,\Lambda_m)^{-(\beta-d)\left(\frac{k^\prime-1}{3k^\prime-1}-\delta\right)}\right)
    % \end{aligned}
    \begin{aligned}
        \lvert\E[\mu_\SA(B_v,f+\rho,\ell)-&\mu_\SA(B_v,f,\ell)]\rvert\leq C_\delta\P(v\in\mathcal{U}_{FR}(f,\rho))^{\frac{k^\prime-1}{k^\prime}}+C_\delta\|\rho\|_{C^1(B_v)}^{\frac{k^\prime-1}{3k^\prime-1}-\delta}\\
        &\qquad\qquad\quad\leq C_\delta\left(\dist(v,\Lambda_m)^{-(\beta-2d-\delta)\frac{k^\prime-1}{k^\prime}}+\dist(v,\Lambda_m)^{-(\beta-d)\left(\frac{k^\prime-1}{3k^\prime-1}-\delta\right)}\right)
    \end{aligned}
    \end{equation}
where the final inequality uses Lemma~\ref{l:FR_general} and \eqref{e:FR_general3}. Since $\beta>\beta_\SA(k)=3d\frac{3k^\prime-1}{k^\prime-2}$, elementary manipulations show that
\begin{displaymath}
    \min\left\{(\beta-2d)\frac{k^\prime-1}{k^\prime},(\beta-d)\frac{k^\prime-1}{3k^\prime-1}\right\}>d.
\end{displaymath}
Therefore taking $\lambda,\delta>0$ sufficiently small, we may apply Lemma~\ref{l:ElementarySum} to see that
\begin{align*}
    \sum_{v\in\Z^d\setminus\Lambda_{m^+}}\lvert\E[\mu_\SA(B_v,f,\ell)-\mu_\SA(B_v,f+\rho,\ell)]\rvert&\leq\sum_{v\in\Z^d\setminus\Lambda_{m^+}}C_\delta \dist(v,\Lambda_m)^{-d-2\eta}\\
    &\leq C_\delta^\prime m^{d-1}m^{-(d+2\eta-1)(1-\lambda)}\leq C^\prime_\delta m^{-\eta}
\end{align*}
for some $\eta>0$, which proves \eqref{e:sa_general_perturbation_b}.
\end{proof}

\begin{proof}[Proof of Theorem~\ref{t:posvar}]
    The desired statement will follow from Lemma~\ref{l:positivity_sufficient} provided that we verify condition \eqref{e:PositivitySufficient}. By Theorem~\ref{t:FirstOrder}
    \begin{displaymath}
        \lim_{s\to\infty}c_\star(\ell-s)-c_\star(\ell)\neq 0.
    \end{displaymath}
    Therefore we can find a bounded interval $I\subset(0,\infty)$ such that for all $s\in I$
    \begin{equation}\label{e:limiting_constant_positive}
        \lvert c_\star(\ell-s)-c_\star(\ell)\rvert\geq c_0
    \end{equation}
    for some constant $c_0>0$. We then let $\rho=s\overline{q}\ast\ind_{\Lambda_m}$ for some $m\in\N$ (which will be chosen large, as specified in the arguments below).
    
    By the additivity of $\mu_\star$ proven in Lemma~\ref{l:ChangeDecomp} (when $\star=\Vol$), Lemma~\ref{l:additivity_SA} (when $\star=\SA$) and Lemma~\ref{l:ECbasics} (when $\star=\EC$), for any $n>m^+$
    \begin{equation}\label{e:positivity_final_1}
        \E[\mu_\star(\Lambda_n,f,\ell)-\mu_\star(\Lambda_n,f+\rho,\ell)]=\Circled{1}+\Circled{2}+\Circled{3}+\Circled{4}+\Circled{5}+\Circled{6}
    \end{equation}
    where
    \begin{align*}
        \Circled{1}&=\E[\mu_\star(\Lambda_{m^{--}},f,\ell)-\mu_\star(\Lambda_{m^{--}},f+s,\ell)]\\
        \Circled{2}&=\E[\mu_\star(\Lambda_{m^{--}},f+s,\ell)-\mu_\star(\Lambda_{m^{--}},f+\rho,\ell)]\\
        \Circled{3}&=\E[\mu_\star(\overline{\Lambda_{m^+}\setminus\Lambda_{m^{--}}},f,\ell)-\mu_\star(\overline{\Lambda_{m^+}\setminus\Lambda_{m^{--}}},f+\rho,\ell)]\\
        \Circled{4}&=-\E[\mu_\star(\partial\Lambda_{m^{--}},f,\ell)-\mu_\star(\partial\Lambda_{m^{--}},f+\rho,\ell)]\\
        \Circled{5}&=\E[\mu_\star(\overline{\Lambda_{n}\setminus\Lambda_{m^{+}}},f,\ell)-\mu_\star(\overline{\Lambda_{n}\setminus\Lambda_{m^{+}}},f+\rho,\ell)]\\
        \Circled{6}&=-\E[\mu_\star(\partial\Lambda_{m^+},f,\ell)-\mu_\star(\partial\Lambda_{m^+},f+\rho,\ell)].
    \end{align*}
    By Theorem~\ref{t:FirstOrder} and \eqref{e:limiting_constant_positive}
    \begin{displaymath}
        \lvert\Circled{1}\rvert=(2m^{--})^d\lvert c_\star(\ell)-c_\star(\ell-s)\rvert+o(m^d)\geq c_0m^d+o(m^d)
    \end{displaymath}
    as $m\to\infty$. We will now show that terms $\Circled{2}-\Circled{6}$ above are $o(m^d)$ as $m\to\infty$ (uniformly in $n$). Then by taking $m$ sufficiently large, the absolute value of \eqref{e:positivity_final_1} will be bounded below by a positive constant, uniformly in $n$, as required.

    We first note that when $\star\in\{\Vol,\SA\}$, we have $\Circled{4}=\Circled{6}=0$ (see the proofs of Lemmas~\ref{l:ChangeDecomp} and~\ref{l:additivity_SA}). When $\star=\EC$, by Lemma~\ref{l:ECBound}
    \begin{align*}
        \lvert\Circled{4}\rvert&\leq c_d\E\left[\overline{N}_\mathrm{Crit}(\partial\Lambda_{m^{--}},f)+\overline{N}_\mathrm{Crit}(\partial\Lambda_{m^{--}},f+\rho)\right],\\
        \lvert\Circled{6}\rvert&\leq c_d\E\left[\overline{N}_\mathrm{Crit}(\partial\Lambda_{m^{+}},f)+\overline{N}_\mathrm{Crit}(\partial\Lambda_{m^{+}},f+\rho)\right].
    \end{align*}
    Since $\rho\in C^2(\R^d)$ and $f$ is stationary, applying Proposition~\ref{p:TMB} to each stratum of $\partial\Lambda_{m^{--}}$ and $\partial\Lambda_{m^+}$ shows that the two expressions above are $O(m^{d-1})$ as $m\to\infty$. Hence for each choice of $\star$, $\Circled{4}$ and $\Circled{6}$ are $o(m^d)$.
    
    We next consider $\Circled{3}$; if $\star=\Vol$ then
    \begin{displaymath}
        \lvert\Circled{3}\rvert\leq \Vol[\Lambda_{m^+}\setminus\Lambda_{m^{--}}]\leq C_dm^{d-\lambda}
    \end{displaymath}
    for some constant $C_d>0$ depending only on $d$. If $\star=\SA$ then by Lemma~\ref{l:SATrivialBound}
    \begin{align*}
        \lvert\Circled{3}\rvert&\leq \E\left[\mathcal{H}^{d-1}[\{f=\ell\}\cap \Lambda_{m^+}\setminus\Lambda_{m^{--}}]+ \mathcal{H}^{d-1}[\{f+\rho=\ell\}\cap \Lambda_{m^+}\setminus\Lambda_{m^{--}}]\right]\\
        &\leq c \Vol[\Lambda_{m^+}\setminus\Lambda_{m^{--}}]\leq C_d^\prime m^{d-\lambda}
    \end{align*}
    where the penultimate inequality uses Proposition~\ref{p:SAMoments} (and the fact that the $C^2$-norm of $\rho$ on any unit cube is bounded uniformly in $m$). If $\star=\EC$, then by Lemma~\ref{l:ECBound}
    \begin{displaymath}
        \lvert\Circled{3}\rvert\leq\E\left[\overline{N}_\mathrm{Crit}(\overline{\Lambda_{m^+}\setminus\Lambda_{m^{--}}},f)+\overline{N}_\mathrm{Crit}(\overline{\Lambda_{m^+}\setminus\Lambda_{m^{--}}},f+\rho)\right]=O(m^{d-\lambda})
    \end{displaymath}
    by the same argument as above. So for each choice of $\star$, $\Circled{3}=o(m^d)$.

    When $\star=\SA$, Lemma~\ref{l:sa_general_perturbation_bounds} immediately yields that $\Circled{2}=o(m^d)$ and, when combined with additivity of $\mu_\SA$ (Lemma~\ref{l:additivity_SA}), that
    \begin{displaymath}
    \lvert\Circled{5}\rvert\leq\sum_{v\in\Z^d\setminus\Lambda_{m^+}}\lvert\E[\mu_\SA(B_v,f,\ell)-\mu_\SA(B_v,f+\rho,\ell)]\rvert=o(1)
    \end{displaymath}
    as $m\to\infty$. This completes the proof for $\star=\SA$.
    
    Turning to $\star=\Vol$; by Lemmas~\ref{l:ChangeDecomp},~\ref{l:VolMoments} and~\ref{l:FR_general}
    \begin{align*}
    \lvert\Circled{2}\rvert&\leq\sum_{v\in\Z^d\cap\Lambda_{m^{--}}}\E\left[\Vol\left[B_v\cap\{\lvert f+s-\ell\rvert\leq\lvert\rho-s\rvert\}\right]\right]+\P(v\in\mathcal{U}_\mathrm{FR}(f,\rho-s,\ell-s))\\
        &\leq\sum_{v\in\Z^d\cap\Lambda_{m^{--}}}c\|\rho-s\|_{C(B_v)}+c_\delta m^{-\beta+2d+\delta}\\
        &\leq (2m+1)^d\left(Cm^{-(\beta-d)(1-\lambda)}+c_\delta m^{-\beta+2d+\delta}\right)
    \end{align*}
    where the final inequality follows from the bound on $\|\rho-s\|$ in \eqref{e:FR_general1}. Since $\beta>3d$, taking $\lambda,\delta>0$ sufficiently small ensures that this bound is $o(1)$.

    Making use of Lemmas~\ref{l:ChangeDecomp},~\ref{l:VolMoments} and~\ref{l:FR_general} once more, we have
    \begin{align*}
        \lvert\Circled{5}\rvert&\leq\sum_{v\in\Z^d\setminus\Lambda_{m^{+}}}\E\left[\Vol\left[B_v\cap\{\lvert f-\ell\rvert\leq\lvert\rho\rvert\}\right]\right]+\P(v\in\mathcal{U}_\mathrm{FR}(f,\rho,\ell))\\
        &\leq\sum_{v\in\Z^d\setminus\Lambda_{m^{+}}}c\|\rho\|_{C(B_v)}+c_\delta \dist(v,\Lambda_m)^{-\beta+2d+\delta}\\
        &\leq C_\delta m^{d-1}m^{-(\beta-2d-1-\delta)(1-\lambda)}
    \end{align*}
    where the final inequality uses \eqref{e:FR_general3} and Lemma~\ref{l:ElementarySum}. Since $\beta>3d$, taking $\lambda,\delta>0$ sufficiently small ensures that this expression is $o(1)$ as $m\to\infty$, as required.

    Finally we let $\star=\EC$, then by Lemma~\ref{l:ChangeDecompEC}, H\"older's inequality, Proposition~\ref{p:TMB} and Lemma~\ref{l:FR_general}
    \begin{align*}
    \lvert\Circled{2}\rvert&\leq\sum_{v\in\Z^d\cap\Lambda_{m^{--}}}\E\left[\left(\overline{N}_\mathrm{Crit}(B_v,f)+\overline{N}_\mathrm{Crit}(B_v,f+\rho)\right)^{k-1}\right]^{\frac{1}{k-1}}\P(v\in\mathcal{U}_\mathrm{FR}(f,\rho-s,\ell-s))^\frac{k-2}{k-1}\\
    &\leq\sum_{v\in\Z^d\cap\Lambda_{m^{--}}}c_\delta m^{-(\beta-2d-\delta)\frac{k-2}{k-1}}\leq c_\delta^\prime m^{d-(\beta-2d-\delta)\frac{k-2}{k-1}}.
    \end{align*}
    Since $\beta>2d$, taking $\delta>0$ sufficiently small ensures that $\Circled{2}=o(m^d)$. By similar reasoning
    \begin{align*}
    \lvert\Circled{5}\rvert&\leq\sum_{v\in\Z^d\setminus\Lambda_{m^+}}\E\left[\left(\overline{N}_\mathrm{Crit}(B_v,f)+\overline{N}_\mathrm{Crit}(B_v,f+\rho)\right)^{k-1}\right]^{\frac{1}{k-1}}\P(v\in\mathcal{U}_\mathrm{FR}(f,\rho,\ell))^\frac{k-2}{k-1}\\
    &\leq\sum_{v\in\Z^d\setminus\Lambda_{m^+}}c_\delta \dist(v,\Lambda_m)^{-(\beta-2d-\delta)\frac{k-2}{k-1}}\leq c_\delta^\prime m^{d-1}m^{-(\beta-2d-\delta)(1-\lambda)\frac{k-2}{k-1}+1-\lambda}
    \end{align*}
    where the final inequality uses Lemma~\ref{l:ElementarySum}. Since $\beta>\frac{k-1}{k-3}3d$, taking $\lambda$ and $\delta$ sufficiently small ensures that this expression is $o(1)$ as $m\to\infty$. This completes the proof in the case that $\star=\EC$, and so completes the proof of the theorem.    
\end{proof}

\section{Surface area perturbation}\label{s:SA_perturb}
In this section we prove Lemma~\ref{l:SA_Stable} which controls the geometric contribution to changes in the surface area functional (for a deterministic function) under perturbation. The proof is based on projecting the original level set onto its perturbed analogue. Before defining this projection, we need a preliminary result.
\begin{lemma}\label{l:TopSeparation}
    Let $g,p,D$ and $H$ be as given in Lemma~\ref{l: morse continuity} (or Lemma~\ref{l:TopStab}). If $A$ and $A^\prime$ are two components of $\{g=\ell\}\cap D$ such that $H(A,t)\cap H(A^\prime,s)\neq\emptyset$ for some $t,s\in[0,1]$ then $A=A^\prime$.
\end{lemma}
\begin{proof}
    If $t=s$ then this follows from the fact that $H(\cdot,t)$ is a homeomorphism, so we may assume $t\neq s$. We define
    \begin{displaymath}
        \varepsilon=\min_{A_1\neq A_2}\min_{u\in[0,1]}\dist(H(A_1,u),H(A_2,u))
    \end{displaymath}
    where the minimum is taken over all distinct components $A_1$ and $A_2$ of $\{g=\ell\}\cap B$. Note that $\varepsilon>0$ since the distance between image components is continuous in $u$. By assumption we may choose $x\in H(A,t)\cap H(A^\prime,s)$. Then by definition of $H$, $(g+tp)(x)=\ell=(g+sp)(x)$. Since $t\neq s$, we see that $p(x)=0$ and so
    \begin{displaymath}
        x\in \bigcap_{u\in[0,1]}\{g+up=\ell\}.
    \end{displaymath}
    Let us assume, for a contradiction, that $A\neq A^\prime$, then
    \begin{displaymath}
        \dist(x,H(A^\prime,t))\geq\varepsilon\quad\text{and}\quad\dist(x,H(A^\prime,s))=0.
    \end{displaymath}
    Hence by the intermediate value theorem, $\dist(x,H(A^\prime,u))=\varepsilon/2$ for some $u\in(t,s)$. By definition of $\varepsilon$ this implies that $x\notin\{g+up=\ell\}$ which yields the required contradiction.
\end{proof}

We can now define our normal projection and derive its properties.
\begin{lemma}\label{l:SA_bijection}
    Let $g,p$ and $B:=B_w$ satisfy Assumption~\ref{a:SA} and suppose also that $p(x)=0$ for all $x$ within distance $2A_0/A_1$ of $\partial B$. Then there exists an injective, $C^1$ function  $\varphi$ defined on a neighbourhood (in $B$) of $\{g=\ell\}\cap B$ such that for all $x$
    \begin{displaymath}
        \varphi(x)=x+r_x\frac{\nabla g(x)}{\lvert\nabla g(x)\rvert}\quad\text{where }\lvert r_x\rvert\leq \frac{2A_0}{A_1}\qquad\text{and}\qquad (g+p)(\varphi(x))=g(x).
    \end{displaymath}
    % and
    % \begin{displaymath}
    %     (g+p)(\varphi(x))=g(x).
    % \end{displaymath}
    Moreover if $L$ is a component of $\{g=\ell\}\cap B$ and $H$ is the stratified isotopy defined in Lemma~\ref{l: morse continuity} (or Lemma~\ref{l:TopStab} when this lemma can be applied) then $\varphi(L)=H(L,1)$.
\end{lemma}
\begin{proof}
    The continuity properties of $\varphi$ will be proven by applying the implicit function theorem locally, however first we establish some global existence and uniqueness properties. We begin by showing that $\varphi$ is well-defined. For $x\in B$ such that $\nabla g(x)\neq 0$, define $e_x=\nabla g(x)/\lvert\nabla g(x)\rvert$. Let $x$ be within distance $A_1/(3A_2)$ of $\{g=\ell\}\cap B$, then by Assumption~\ref{a:SA} $\lvert\nabla g(x)\rvert>2A_1/3$ and so $e_x$ is well defined. If $x$ is within a distance of $2A_0/A_1$ from $\partial B$ then $g(x)=(g+p)(x)$ so we may take $r_x=0$ and $\varphi(x)=x$. Otherwise, by a Taylor expansion
\begin{displaymath}
\lvert g(x+re_x)-g(x)-r\partial_{e_x}g(x)\rvert\leq A_2r^2
\end{displaymath}
for any $r$ (such that $x+re_x\in B$) where $\partial_{e_x}$ denotes the partial derivative in the direction $e_x$. Setting $r=2A_0/A_1$ and noting that $\partial_{e_x}g(x)> 2A_1/3$ and $\lvert p\rvert<A_0$ we have
\begin{align*}
    (g+p)(x+re_x)&>g(x)+4A_0/3-A_2r^2-A_0>g(x)\quad\text{and}\\
    (g+p)(x-re_x)&<g(x)-4A_0/3+A_2r^2+A_0<g(x).
\end{align*}
Hence, by the intermediate value theorem, there exists an $r_x$ satisfying $\lvert r_x\rvert\leq 2A_0/A_1$ and $(g+p)(x+r_xe_x)=g(x)$. To see that this value is unique, we note that for $\lvert r\rvert\leq 2A_0/A_1$
\begin{displaymath}
    \lvert \nabla g(x+re_x)-\nabla g(x)\rvert\leq A_2\lvert r\rvert
\end{displaymath}
(by Assumption~\ref{a:SA}) and so
\begin{align*}
    \frac{d}{dr}g(x+re_x)=e_x\cdot\nabla g(x+re_x)\geq\lvert\nabla g(x)\rvert-A_2\lvert r\rvert\geq \frac{2A_1}{3}-A_2\frac{2A_0}{A_1}>0.
\end{align*}
Hence $\varphi$ is well-defined for all $x$ in a neighbourhood of $\{g=\ell\}\cap B$. Note that the uniqueness argument also applies for $x$ close to the boundary of $B$ (provided we restrict to $r$ such that $x+re_x\in B$).

Next we show injectivity of $\varphi$. Let $x,y\in B$ be within distance $A_1/(3A_2)$ of $\{g=\ell\}\cap B$ such that $\varphi(x)=\varphi(y)$. By definition we then have
\begin{displaymath}
    y-x=s_x\nabla g(x)-s_y\nabla g(y)=(s_x-s_y)\nabla g(x)+s_y(\nabla g(x)-\nabla g(y))
\end{displaymath}
where $\lvert s_x\rvert,\lvert s_y\rvert\leq 3A_0/A_1^2$. Taking the dot product with $y-x$ and using the triangle inequality
\begin{displaymath}
    \lvert y-x\rvert^2\leq \lvert s_x-s_y\rvert\lvert\nabla g(x)\cdot(y-x)\rvert+\lvert s_y\rvert\;\lvert(\nabla g(x)-\nabla g(y))\cdot(y-x)\rvert.
\end{displaymath}
Since $g(x)=g(y)$, a Taylor expansion of $g$ at $x$ shows that $\lvert\nabla g(x)\cdot(y-x)\rvert\leq A_2\lvert y-x\rvert^2$. Substituting this in above (and using our assumption on $\nabla^2 g$) we have
\begin{displaymath}
    \lvert y-x\rvert^2\leq A_2\lvert s_x-s_y\rvert\lvert y-x\rvert^2+A_2\lvert s_y\rvert \lvert y-x\rvert^2\leq\frac{9A_0A_2}{A_1^2}\lvert y-x\rvert^2.
\end{displaymath}
By Assumption~\ref{a:SA} the above inequality can only hold if $x=y$, yielding injectivity.

Next we show that if $L$ is a component of $\{g=\ell\}\cap B$ then $\varphi(L)\subseteq H(L,1)$. For $t\in[0,1]$, since $\|tp\|_{C^1(B)}\leq \|p\|_{C^1(B)}$ we may define $\varphi_t:\{g=\ell\}\to\{g+tp=\ell\}$ as $\varphi_t(x)=x+r_{t,x}e_x$ using the same argument for defining $\varphi$ as above. It will be enough to show that $\varphi_t$ is continuous in $t$ and apply Lemma~\ref{l:TopSeparation}.

For any $x\in\{g=\ell\}\cap\mathrm{Int}B$, $t\in[0,1]$ and $\lvert r\rvert\leq 2A_0/A_1$
\begin{align*}
    \frac{\partial}{\partial r}(g+tp)(x+re_x)&=e_x\cdot\nabla (g+tp)(x+re_x)\geq \lvert\nabla g(x)\rvert-A_2\lvert r\rvert-A_0>0.
\end{align*}
Hence we may apply the implicit function theorem to
\begin{equation}\label{e:IFT}
    (r,t,x)\mapsto(g+tp)(x+re_x)-g(x)\;:\;\R^{d+2}\to\R
\end{equation}
at the point $(r_{t,x},t,x)$ to find a $C^1$ mapping $(t,x)\to \tilde{r}_{t,x}$ such that $(g+tp)(x+\tilde{r}_{t,x}e_x)=g(x)$. By the earlier uniqueness argument, $r_{t,x}=\tilde{r}_{t,x}$. By compactness, the mapping extends to the Cartesian product of $[0,1]$ and a neighbourhood of $\{g=\ell\}$ in $B$. In particular, $\varphi(x)=x+r_{1,x}e_x$ is $C^1$ and for any given $x\in\{g=\ell\}\cap B$ the mapping $t\mapsto\varphi_t(x)=x+r_{t,x}e_x$ is continuous in $t$.

By Lemma~\ref{l:TopSeparation} and compactness
\begin{displaymath}
    \eta:=\inf_{L_1\neq L_2}\inf_{t,s\in[0,1]}\dist( H(L_1,t),H(L_2,s))>0
\end{displaymath}
where the infimum is taken over distinct components $L_1$ and $L_2$ of $\{g=\ell\}\cap B$. Let $x\in\{g=\ell\}\cap B$ and $t_0\in[0,1]$, then $\varphi_{t_0}(x)\in\{g+t_0p=\ell\}=H(\{g=\ell\},t_0)$. We then let $L$ be the component of $\{g=\ell\}\cap B$ such that $\varphi_{t_0}(x)\in H(L,t_0)$. By continuity, for all $\lvert t-t_0\rvert$ sufficiently small $\lvert\varphi_t(x)-\varphi_{t_0}(x)\rvert<\eta/2$. Hence by the definition of $\eta$, we must have $\varphi_t(x)\in H(L,t)$. This is true for all $t_0$, so by compactness we can cover $[0,1]$ by finitely many such intervals to conclude that $\varphi_t(x)\in H(L,t)$ for all $t\in[0,1]$. In particular $x=\varphi_0(x)\in H(L,0)=L$ and $\varphi(x)=\varphi_1(x)\in H(L,1)$ as required.

Finally we show that $\varphi(L)=H(L,1)$ whenever $L$ is a component of $\{g=\ell\}\cap B$. Our previous use of the implicit function theorem shows that $\varphi$ is continuous on $\{g=\ell\}\cap \mathrm{Int} B$, and continuity on a neighbourhood of $\partial B$ follows since $\varphi$ is equal to the identity here. Therefore compactness of $L$ implies that $\varphi(L)$ is compact and hence closed (in the Euclidean topology on $\R^d$). In particular $\varphi(L)$ is also closed in the subspace topology for $\{g+p=\ell\}\cap B$.

Next we claim that $\varphi:L\to H(L,1)$ is an open map (with respect to the subspace topologies on $L$ and $H(L,1)$). For $x\in L$ such that $\dist(x,\partial B)\geq 2A_0/A_1$ and any neighbourhood $U$ (in $L$) of $x$, by the implicit function theorem there is a neighbourhood $V\subseteq U$ such that $\varphi(V)$ contains the intersection of $\{g+p=\ell\}$ and some neighbourhood of $\varphi(x_0)$. For $x\in L$ satisfying $\dist(x,\partial B)<2A_0/A_1$, by definition $\varphi$ is equal to the identity on a neighbourhood of $x$. Hence $\varphi$ is an open map, proving the claim.

From the previous three paragraphs, we see that $\varphi(L)$ is a non-empty subset of $H(L,1)$ which is both open and closed in the subspace topology for $H(L,1)$. Since $H(L,1)$ is connected (recall that $H(\cdot,t)$ is a homeomorphism) we conclude that $\varphi(L)=H(L,1)$.
\end{proof}

\begin{lemma}\label{l:SA_AreaFormula}
    There exist constants $C_d,C_d^\prime>0$ depending only on $d$ such that the following holds. Let $g,p,B$ satisfy Assumption~\ref{a:SA} with $C_d$ and suppose that $p(x)=0$ for all $x$ within distance $2A_0/A_1$ of $\partial B$. Let $E\subset B$ be an open set with diameter at most $A_1/3A_2$ and distance at least $2A_0/A_1$ from $\partial B$. Then
    \begin{displaymath}
        \left\lvert\mathcal{H}^{d-1}\left[\varphi(\{g=\ell\}\cap E)\right]-\mathcal{H}^{d-1}\left[\{g=\ell\}\cap E\right]\right\rvert\leq C_d^\prime\frac{A_0A_2}{A_1^2}\mathcal{H}^{d-1}\left[\{g=\ell\}\cap E\right].
    \end{displaymath}
\end{lemma}

\begin{proof}
In the proof we let $C_d^\prime$ be a positive constant depending only on $d$, the value of which may change from line to line. To begin with we find a parameterisation of the $(d-1)$-dimensional set $\{g=\ell\}\cap E$. We fix $x^*\in\{g=\ell\}\cap E$ and choose a basis $e_1,\dots,e_d$ for $\R^d$ such that $e_d=e_{x^*}$. For all $x\in E$
\begin{displaymath}
    \partial_{e_d}g(x)\geq\partial_{e_d}g(x^*)-A_2\lvert x-x^*\rvert>A_1-A_2\lvert x-x^*\rvert
\end{displaymath}
which is strictly positive since $\diam E<A_1/(3A_2)$. Therefore if we define
\begin{displaymath}
    \widetilde{E}=\{(x_1,\dots,x_{d-1})\;|\;(x_1,\dots,x_d)\in\{g=\ell\}\cap E\text{ for some }x_d\in\R\}
\end{displaymath}
we see that for all $(x_1,\dots,x_{d-1})\in\widetilde{E}$ there exists a unique $z=z(x_1,\dots,x_{d-1})$ such that $g(x_1,\dots,x_{d-1},z)=\ell$. Moreover since $E$ is open, the implicit function theorem implies that $z$ is a $C^1$ function of $x_1,\dots,x_{d-1}$. We write $G(x_1,\dots,x_{d-1})=(x_1,\dots,x_{d-1},z(x_1,\dots,x_{d-1}))$.

We recall that the area formula states that for a Lipschitz function $F:\R^n\to\R^m$ where $n\leq m$ and a measurable $A\subseteq\R^n$,
\begin{equation*}
    \int_A J_F(x)\;dx=\int_{\R^m}\mathcal{H}^0\left[A\cap F^{-1}(y)\right]\;d\mathcal{H}^n(y)
\end{equation*}
where the Jacobian of $F$ is defined by $J_F^2:=\det DF(DF)^*$. Applying this to $F=G$ and $A=\widetilde{E}=G^{-1}(E\cap\{g=\ell\})$ yields
\begin{equation}\label{e:AreaFormula1}
\begin{aligned}
    \int_{\widetilde{E}}J_G(x)\;dx&=\int_{\R^d}\mathcal{H}^0\left[\widetilde{E}\cap G^{-1}(y)\right]\;d\mathcal{H}^{d-1}(y)=\int_{E\cap\{g=\ell\}}1\;d\mathcal{H}^{d-1}(y)=\mathcal{H}^{d-1}\left[E\cap\{g=\ell\}\right]
\end{aligned}
\end{equation}
since $G$ is a bijection from $\widetilde{E}$ to $E\cap\{g=\ell\}$. Applying the area formula to $\varphi\circ G$ with similar reasoning yields
\begin{equation}\label{e:AreaFormula2}
    \int_{\widetilde{E}}J_{\varphi\circ G}(x)\;dx=\mathcal{H}^{d-1}\left[\varphi(E\cap\{g=\ell\})\right].
\end{equation}
Next we observe that since the derivative of $\varphi$ is a square matrix
\begin{align*}
    J_{\varphi\circ G}^2=\det D\varphi DG (D\varphi DG)^*=\det D\varphi \det DG(DG)^* \det (D\varphi)^*=(\det D\varphi)^2 J_G^2.
\end{align*}
Combining this with \eqref{e:AreaFormula1} and \eqref{e:AreaFormula2}, we see that the lemma will follow if we can prove that
\begin{equation}\label{e:AreaFormula3}
    \lvert \det D\varphi -1\rvert\leq C_d^\prime\frac{A_0A_2}{A_1^2}.
\end{equation}
Recalling that $\varphi(x)=x+r_xe_x$ we have
\begin{displaymath}
    \frac{\partial(\varphi(x))_i}{\partial x_j}=\ind_{i=j}+\frac{\partial r_x}{\partial x_j}(e_x)_i+r_x\frac{\partial (e_x)_i}{\partial x_j}.
\end{displaymath}
We now claim that there exists $C_d^\prime>0$ depending only on $d$ such that for $j\in\{1,\dots,d\}$
\begin{equation}\label{e:SA_Claim}
    \left\lvert\frac{\partial r_x}{\partial x_j}\right\rvert\leq C_d^\prime\frac{A_0A_2}{A_1^2}\quad\text{and}\quad\left\lvert\frac{\partial e_{x}}{\partial x_j}\right\rvert\leq C_d^\prime\frac{A_2}{A_1}.
\end{equation}
Combined with the facts that $\lvert e_{x}\rvert=1$ and $\lvert r_x\rvert\leq 2A_0/A_1$ this claim implies that each entry of the Jacobian matrix differs from the corresponding entry of the identity matrix by at most $C_d^\prime A_0A_2/A_1^2$. Since $A_0A_2/A_1^2<1$, this yields \eqref{e:AreaFormula3} and hence the statement of the lemma.

It remains to prove the claim \eqref{e:SA_Claim}. First we note that by Assumption~\ref{a:SA} and the fact that $\diam(E)<A_1/(2A_2)$, for all $x\in E$
    \begin{equation}\label{e:SA_Jacob1}
        \frac{\partial g(x)}{\partial x_d}>\frac{A_1}{2}\quad\text{and}\quad\frac{\partial g(x)}{\partial x_i}<\frac{A_1}{2}\quad\text{for }i=1,\dots,d-1.
    \end{equation}
    By the quotient rule, for $i=1,\dots,d$
    \begin{align*}
        \frac{\partial e_{x}}{\partial x_i}=\frac{\partial}{\partial x_i}\frac{\nabla g(x)}{\lvert \nabla g(x)\rvert}=\frac{\partial_{e_i}\nabla g(x)}{\lvert \nabla g(x)\rvert}-\frac{\nabla g(x)}{\lvert\nabla g(x)\rvert}\sum_{k=1}^d\frac{\partial_{e_k}g(x)}{\lvert\nabla g(x)\rvert}\frac{\partial_{e_k}\partial_{e_i}g(x)}{\lvert\nabla g(x)\rvert}.
    \end{align*}
    By considering each fraction on the right hand side in turn, we see from Assumption~\ref{a:SA}  that 
    \begin{equation}\label{e:AreaFormula4}
        \left\lvert \frac{\partial e_{x}}{\partial x_i}\right\rvert\leq C_d^\prime\frac{A_2}{\lvert \nabla g(x)\rvert}.
    \end{equation}
    Next, by applying the implicit function theorem to $(x,r)\mapsto (g+p)(x+re_x)-g(x)=:P(x,r)$ we see that
    \begin{equation}\label{e:AreaFormula5}
        \frac{\partial r_x}{\partial x_j}=-\left(\frac{\partial P}{\partial r}\right)^{-1}\frac{\partial P}{\partial x_j}=-\frac{\partial_{x_j}(g+p)(\star)-\partial_{x_j}g(x)+r\nabla(g+p)(\star)\cdot\partial_{x_j}e_x}{\nabla(g+p)(\star)\cdot e_x}
    \end{equation}
    where $\star:=x+r_xe_x$. By Assumption~\ref{a:SA}
    \begin{align*}
        \left\lvert\nabla g(\star)-\nabla g(x)\right\rvert\leq C_d^\prime A_2\lvert r\rvert,\;\;
        \lvert\nabla p(\star)\rvert\leq A_0,\;\;
        \nabla(g+p)(\star)\cdot e_x&\geq\lvert\nabla g(x)\rvert-A_2\lvert r\rvert-A_0>\frac{A_1}{4}.
    \end{align*}
    Substituting these and \eqref{e:AreaFormula4} into \eqref{e:AreaFormula5} proves the left hand side of \eqref{e:SA_Claim}.
\end{proof}

\begin{proof}[Proof of Lemma~\ref{l:SA_Stable}]
By Lemma~\ref{l:SAComponentDecomp} the change in $\SA_\infty$ is bounded above by
\begin{multline}\label{e:SA_Bound1}
    \sum_{L\in\mathrm{Comp}}\lvert\mathcal{H}^{d-1}[H(L,1)]-\mathcal{H}^{d-1}[L]\rvert\\
    \leq\sum_{L\in\mathrm{Comp}}\left\lvert\mathcal{H}^{d-1}\Big[H(L,1)\setminus(\partial B)_{+\frac{7A_0}{A_1}}\Big]-\mathcal{H}^{d-1}\Big[L\setminus(\partial B)_{+\frac{5A_0}{A_1}}\Big]\right\rvert\\
    +\mathcal{H}^{d-1}\Big[\{g+p=\ell\} \cap (\partial B)_{+\frac{7A_0}{A_1}}\Big]+\mathcal{H}^{d-1}\Big[\{g=\ell\} \cap (\partial B)_{+\frac{5A_0}{A_1}}\Big]
\end{multline}
where the latter bound follows from the triangle inequality.

Choose $\theta\in C^2(B)$ such that $0\leq \theta\leq 1$ and
\begin{displaymath}
    \theta(x)=\begin{cases}
        1 &\text{if }x\in B\setminus(\partial B)_{+\frac{3A_0}{A_1}}\\
        0 &\text{if }x\in(\partial B)_{+\frac{2A_0}{A_1}}.
    \end{cases}
\end{displaymath}
If we define $\overline{p}:=\theta p$, then $g$ and $\overline{p}$ satisfy the conditions of Lemma~\ref{l:SA_bijection} and we obtain a bijection $\varphi:\{g=\ell\}\cap B\to\{g+\overline{p}=\ell\}\cap B$. Let $\overline{H}$ denote the stratified isotopy of $B$ mapping $\{g=\ell\}\cap B$ to $\{g+\overline{p}=\ell\}\cap B$ which exists by Lemma~\ref{l: morse continuity}. Since $\lvert\varphi(x)-x\rvert\leq 2A_0/A_1$, for any $L\in \mathrm{Comp}$ we have
\begin{equation}\label{e:SA_Boundary}
    \overline{H}(L,1)\setminus(\partial B)_{+\frac{7A_0}{A_1}}\subseteq\varphi\left(L\setminus(\partial B)_{+\frac{5A_0}{A_1}}\right)\subseteq \overline{H}(L,1)\setminus(\partial B)_{+\frac{3A_0}{A_1}}.
\end{equation}
Let $H$ denote the stratified isotopy of $B$ defined by Lemma~\ref{l:TopStab} which maps $\{g=\ell\}\cap B$ to $\{g+p=\ell\}\cap B$ (whilst also preserving connections to unbounded components). We claim that
\begin{equation}\label{e:SA_isotopyclaim}
    H(L,1)\setminus(\partial B)_{+\frac{3A_0}{A_1}}=\overline{H}(L,1)\setminus(\partial B)_{+\frac{3A_0}{A_1}},
\end{equation}
which intuitively follows because $p$ and $\overline{p}$ agree on $B\setminus(\partial B)_{+\frac{3A_0}{A_1}}$. Assuming this claim, by \eqref{e:SA_Boundary} we have
\begin{multline*}
    \sum_{L\in\mathrm{Comp}}\left\lvert\mathcal{H}^{d-1}\Big[\varphi\big(L\setminus(\partial B)_{+\frac{5A_0}{A_1}}\big)\Big]-\mathcal{H}^{d-1}\Big[H(L,1)\setminus(\partial B)_{+\frac{7A_0}{A_1}}\Big]\right\rvert\\
\begin{aligned}
    &\leq\mathcal{H}^{d-1}\Big[\{g+\overline{p}=\ell\}\cap(\partial B)_{+\frac{7A_0}{A_1}}\setminus(\partial B)_{+\frac{3A_0}{A_1}}\Big]\\
    &\leq\mathcal{H}^{d-1}\Big[\{g+p=\ell\}\cap(\partial B)_{+\frac{7A_0}{A_1}}\Big].    
\end{aligned}
\end{multline*}
Combining this with \eqref{e:SA_Bound1} we find that
\begin{align*}
    \lvert\SA_\infty[g,B]-\SA_\infty[g+p,B]\rvert\leq&\sum_{L\in\mathrm{Comp}}\left\lvert\mathcal{H}^{d-1}\Big[\varphi\big(L\setminus(\partial B)_{+\frac{5A_0}{A_1}}\big)\Big]-\mathcal{H}^{d-1}\Big[L\setminus(\partial B)_{+\frac{5A_0}{A_1}}\Big]\right\rvert\\
        &+2\mathcal{H}^{d-1}\Big[\{g=\ell\}\cap (\partial B)_{+\frac{7A_0}{A_1}}\Big]\\
        &+2\mathcal{H}^{d-1}\Big[\{g+p=\ell\}\cap (\partial B)_{+\frac{7A_0}{A_1}}\Big].
\end{align*}
Tiling $L\setminus(\partial B)_{+\frac{5A_0}{A_1}}$ by open dyadic cubes with diameter less than $A_1/3A_2$ and applying Lemma~\ref{l:SA_AreaFormula} on each cube, we conclude that
\begin{displaymath}
    \left\lvert\mathcal{H}^{d-1}\Big[\varphi\big(L\setminus(\partial B)_{+\frac{5A_0}{A_1}}\big)\Big]-\mathcal{H}^{d-1}\Big[L\setminus(\partial B)_{+\frac{5A_0}{A_1}}\Big]\right\rvert\leq C_d^\prime\frac{A_0A_2}{A_1^2}\mathcal{H}^{d-1}\Big[L\setminus(\partial B)_{+\frac{5A_0}{A_1}}\Big]
\end{displaymath}
which, together with the previous equation, yields the statement of the lemma.

It remains only to prove the claim \eqref{e:SA_isotopyclaim}. Since $\theta(x)\in[0,1]$
\begin{align*}
    \bigcup_{L\in\mathrm{Comp}}\overline{H}(L,[0,1])=\bigcup_{t\in[0,1]}\{g+t\overline{p}=\ell\}\subseteq\bigcup_{t\in[0,1]}\{g+tp=\ell\}=\bigcup_{L\in\mathrm{Comp}}H(L,[0,1]).
\end{align*}
By Lemma~\ref{l:TopSeparation}, for distinct $L_1,L_2\in\mathrm{Comp}$ the distance between $H(L_1,[0,1])$ and $H(L_2,[0,1])$ is strictly positive. Therefore each of the connected sets $\overline{H}(L,[0,1])$ for $L\in\mathrm{Comp}$ must be contained in at most one $H(L^\prime,[0,1])$. Since $L\in\overline{H}(L,[0,1])\cap H(L,[0,1])$ we conclude that $L=L^\prime$, that is
\begin{displaymath}
    \overline{H}(L,[0,1])\subseteq H(L,[0,1])\quad\text{for all }L\in\mathrm{Comp}.
\end{displaymath}
Since the isotopies respect level sets and $\overline{p}=p$ on $B\setminus(\partial B)_{+\frac{3A_0}{A_1}}$ we have
\begin{align*}
    \overline{H}(L,1)\setminus (\partial B)_{+\frac{3A_0}{A_1}}\subseteq H(L,[0,1])\cap\{g+\overline{p}=\ell\}\setminus(\partial B)_{+\frac{3A_0}{A_1}}=H(L,1)\setminus(\partial B)_{+\frac{3A_0}{A_1}}.
\end{align*}
The union over $L$ of the sets on the left and right are both equal to $\{g+p=\ell\}\setminus(\partial B)_{+\frac{3A_0}{A_1}}$ and since the sets on the right are disjoint for different $L$, we conclude that
\begin{displaymath}
    \overline{H}(L,1)\setminus (\partial B)_{+\frac{3A_0}{A_1}}=H(L,1)\setminus(\partial B)_{+\frac{3A_0}{A_1}}
\end{displaymath}
as required.
\end{proof}

\bigskip
\bibliographystyle{abbrv}
\bibliography{paper}

\end{document}